\documentclass[11pt]{amsart}
\usepackage[utf8]{inputenc}

\usepackage{overpic, amssymb, times, graphicx, tikz-cd, yfonts, wrapfig}
\usepackage[all]{xy}
\usepackage[backref=page]{hyperref}
\usepackage{verbatim}
\usepackage{enumitem}

\hypersetup{
    colorlinks=true,
    linkcolor=blue,
    filecolor=blue,
    urlcolor=blue,
    citecolor=blue,
}
\DeclareMathOperator{\counit}{\thicc{c}}
\DeclareMathOperator{\Counit}{\thicc{C}}
\DeclareMathOperator{\cyc}{cyc}
\DeclareMathOperator{\ring}{R}
\DeclareMathOperator{\Aug}{Aug}
\DeclareMathOperator{\Det}{det}

\DeclareMathOperator{\com}{com}

\DeclareMathOperator{\LagGras}{Lag}

\DeclareMathOperator{\Unitary}{U}
\DeclareMathOperator{\Orthog}{O}
\DeclareMathOperator{\SOrthog}{SO}
\DeclareMathOperator{\inc}{inc}

\DeclareMathOperator{\ord}{ord}

\DeclareMathOperator{\Morse}{Mo}
\DeclareMathOperator{\Crit}{Crit}
\DeclareMathOperator{\aug}{\thicc{\epsilon}}
\DeclareMathOperator{\spec}{spec} 
\DeclareMathOperator{\im}{im} 
\DeclareMathOperator{\ind}{ind} 
\DeclareMathOperator{\Aut}{Aut}

\DeclareMathOperator{\Id}{Id}
\DeclareMathOperator{\Maslov}{Mas}

\DeclareMathOperator{\Flow}{Flow}

\DeclareMathOperator{\rot}{rot}
\DeclareMathOperator{\lk}{lk}
\DeclareMathOperator{\tb}{tb}

\DeclareMathOperator{\word}{w}

\DeclareMathOperator{\Diag}{Diag}
\DeclareMathOperator{\const}{const}

\DeclareMathOperator{\SympBase}{W}

\newcommand{\Variety}{\mathcal{V}}
\newcommand{\Stab}{\mathcal{S}}
\newcommand{\filtration}{\mathcal{F}}
\newcommand{\filtLevel}{\ell}
\newcommand{\tensorAlg}{\mathcal{T}}
\newcommand{\tensorAlgCom}{\wedge}

\newcommand{\pt}{pt} 

\newcommand{\B}{\mathbb{B}}
\newcommand{\C}{\mathbb{C}}
\newcommand{\R}{\mathbb{R}}
\newcommand{\Z}{\mathbb{Z}}

\newcommand{\disk}{\mathbb{D}}
\newcommand{\grad}{\nabla}

\newcommand{\region}{\mathcal{R}}

\newcommand{\action}{\mathfrak{E}}

\newcommand{\Cinfty}{\mathcal{C}^{\infty}}
\newcommand{\Rthree}{(\R^{3},\xi_{std})}
\newcommand{\Rtwonminusone}{(\R^{2n - 1}, \xi_{std})}

\newcommand{\Circle}{\R/\Z}
\newcommand{\Circlepi}{\R/\pi\Z}
\newcommand{\Circletwopi}{\R/2\pi\Z}

\newcommand{\half}{\frac{1}{2}}

\newcommand{\be}{\begin{enumerate}}
\newcommand{\ee}{\end{enumerate}}

\newcommand{\Mxi}{(M,\xi)}

\newcommand{\orientation}{\mathfrak{o}}

\newcommand{\thicc}[1]{\pmb{#1}}

\newcommand{\ModSpace}{\mathcal{M}}

\newcommand{\tree}{\thicc{T}}

\newcommand{\vertex}{\thicc{V}}
\newcommand{\edge}{\thicc{E}}

\newcommand{\orbitVS}{V}

\newcommand{\ContForm}{\alpha}

\newcommand{\leg}{\lambda}
\newcommand{\Leg}{\Lambda}
\newcommand{\LegGrouped}{(\Leg, \Partition)}
\newcommand{\LegUp}{\Lambda^{+}}
\newcommand{\LegDown}{\Lambda^{-}}
\newcommand{\LegUpDown}{\Leg^{\pm}}
\newcommand{\LegGroupedUp}{(\Lambda^{+}, \Partition^{+})}
\newcommand{\LegGroupedDown}{(\Lambda^{-}, \Partition^{-})}

\newcommand{\Lag}{L}
\newcommand{\LagGrouped}{(\Lag, \Partition)}

\newcommand{\Partition}{\mathcal{P}}
\newcommand{\Chords}{K}

\newcommand{\ChordsUpDown}{\Chords^{\pm}}

\newcommand{\chord}{\kappa}

\newcommand{\Algebra}{\mathcal{A}}
\newcommand{\AlgebraOther}{\mathcal{B}}
\newcommand{\AlgebraUp}{\Algebra^{+}}
\newcommand{\AlgebraDown}{\Algebra^{-}}
\newcommand{\newRSFT}{PDH}
\newcommand{\newRSFTA}{PDA}

\newcommand{\planarDiagram}{\mathcal{PD}}

\newcommand{\cappingPath}{\eta}

\newcommand{\splittingArc}{\vec{a}}

\newcommand{\boundaryPunc}{z}
\newcommand{\orderedBoundaryPunc}{\thicc{z}}
\newcommand{\holoDom}{\disk_{\orderedBoundaryPunc}}
\newcommand{\boundaryWord}{\word^{\partial}}
\newcommand{\boundaryWordThicc}{\thicc{\word}^{\partial}}

\newcommand{\field}{\mathbb{F}}

\newcommand{\CE}{CE}

\newcommand{\LoopSpace}{\mathcal{L}}

\newtheorem{thm}{Theorem}[section]

\newtheorem{ex}[thm]{Example}
\newtheorem{assump}[thm]{Assumptions}

\newtheorem{prop}[thm]{Proposition}

\newtheorem{defn}[thm]{Definition}

\newtheorem{lemma}[thm]{Lemma}
\newtheorem{cor}[thm]{Corollary}

\newtheorem{rmk}[thm]{Remark}

\newtheorem{edits}[thm]{Editor notes}

\newtheorem{notation}[thm]{Notation}

\title{A filtered generalization of the Chekanov-Eliashberg algebra}
\author{Russell Avdek}
\date{\today}

\begin{document}

\begin{abstract}
We define a new algebra associated to a Legendrian submanifold $\Lambda$ of a contact manifold of the form $\R_{t} \times \SympBase$, called the \emph{planar diagram algebra} and denoted $\newRSFTA(\Leg, \Partition)$. It is a non-commutative, filtered, differential graded algebra whose \emph{filtered stable tame isomorphism class} is an invariant of $\Lambda$ together with a partition $\Partition$ of its connected components. When $\Lambda$ is connected, $\newRSFTA$ is the Chekanov-Eliashberg algebra. In general, the $\newRSFTA$ differential counts holomorphic disks with multiple positive punctures using a combinatorial framework inspired by string topology.
\end{abstract}
\maketitle

\numberwithin{equation}{subsection}
\setcounter{tocdepth}{1}
\tableofcontents
\pagebreak

\section{Introduction}

The Chekanov-Eliashberg algebra $\CE(\Leg)$ \cite{Chekanov:LCH, Eliashberg:LCH} is an indispensable tool in the study of Legendrian knots \cite{EtnyreNg:LCHSurvey}, smooth knots \cite{Ng:KCHIntro}, and Weinstein manifolds \cite{BEE:LegendrianSurgery}. It is a differential graded algebra (DGA) assigned to a Legendrian submanifold $\Leg$ of a contact manifold, constructed within the framework of symplectic field theory (SFT) \cite{EGH:SFTIntro}. The \emph{stable tame isomorphism class} of $\CE$ and the isomorphism class of its homology -- \emph{Legendrian contact homology}, denoted $LCH$ -- are invariants of the Legendrian isotopy class of $\Leg$.\footnote{Stable tame isomorphism is a refinement of homotopy equivalence of DGAs, analogous to simple homotopy equivalence of cellular complexes refining homotopy equivalence.}

This article describes a generalization of $\CE$ which enhances its fundamental algebraic properties. We call our new invariant the \emph{planar diagram algebra}, denoted
\begin{equation*}
\newRSFTA = (\Algebra, \partial, \filtration),
\end{equation*}
which is a DGA $(\Algebra, \partial)$ equipped with a special ascending filtration $\filtration$. More precisely $\newRSFTA$ is a \emph{max-filtered DGA} (mfDGA), a type of algebraic object which is manifest in algebraic topology and symplectic geometry but which did not (to the author's knowledge) previously have a name.

The planar diagram algebra depends on a partition $\Partition$ of the connected components of $\Leg$. A Legendrian isotopy of $(\Leg, \Partition)$ is a $1$ parameter family $(\Leg_{T}, \Partition_{T})$ with $\Leg_{0}=\Leg$ for which the partition assignment function is continuous. If $\Partition$ is trivial (eg. when $\Leg$ is connected), $\newRSFTA = \CE$. Each filtered piece
\begin{equation*}
\filtration^{\filtLevel}\newRSFTA = (\filtration^{\filtLevel}\Algebra, \partial, \filtration)
\end{equation*}
of $\newRSFTA$ is itself a mfDGA.

The filtration $\filtration$ determines $\Z_{> 0} \cup \{\infty\}$-valued torsion invariants $\tau_{H}, \tau_{A}$ \cite{KMVW16, LW:Torsion, MZ:Torsion, Wendl:Torsion} of $\LegGrouped$ and each $\filtration^{\filtLevel}\newRSFTA$ automatically comes equipped with rich, well-studied algebraic structures such as associated augmentation varieties \cite{Ng:ComputableInvariants} and bilinearized (co)homologies \cite{BC:Bilinearized, CEKSW:AugCat}. When combined, augmentations and filtrations generate a wealth of easy-to-compare invariants of $\LegGrouped$ from $\newRSFTA$ which we package as graphs and polynomials in \S \ref{Sec:mfDGA}. From a computational point of view, $\newRSFTA$ is accessible using the $\CE$ toolkit combined with spectral sequences and analysis of crossingless matchings on disks.

Differentials and cobordism maps for $\newRSFTA$ are defined by counting the holomorphic disks of Ekholm's RSFT \cite{Ekholm:Z2RSFT} in a new way, thereby determining our new algebraic structures. The story here is analogous to that of Hutchings' ``$q$-variables only'' closed string RSFT \cite{Hutchings:QOnly, MZ:Torsion, Siegel:RSFT} reformulating the closed string RSFT of Eliashberg-Givental-Hofer \cite{EGH:SFTIntro}.

Because we are using the disks of \cite{Ekholm:Z2RSFT}, we avoid having to worry about multiple covers and string topological degenerations of disks \cite{CL:SFTStringTop, Ng:RSFT}. The way we organize counts of holomorphic disks is however motivated by Chas-Sullivan's ``chord diagram'' reformulation of string topology \cite{StringTopology, CS:StringDiagrams}.

Before further motivating $\newRSFTA$ and reviewing its construction, we state our main results:

\begin{thm}\label{Thm:Main}
Let $\Leg$ be a Legendrian submanifold with a partition $\Partition$ of its connected components inside of a contactization $\R_{t} \times W$ of an exact symplectic manifold $(W, \beta)$ with finite geometry at infinity. Then $\newRSFTA(\Leg, \Partition)$ is a max-filtered DGA whose filtered stable tame isomorphism class is an invariant of the Legendrian isotopy class of $(\Leg, \Partition)$.

An exact Lagrangian cobordism $\Lag \subset \R_{s} \times \R_{t} \times \SympBase$ with positive end $\Leg^{+}$, negative end $\Leg^{-}$, and partition $\Partition$ inducing partitions $\Partition^{\pm}$ of the $\Leg^{\pm}$ induces a morphism of mfDGAs,
\begin{equation*}
\newRSFTA(\Lag, \Partition): \newRSFTA\LegGroupedUp \rightarrow \newRSFTA\LegGroupedDown.
\end{equation*}
The chain homotopy class of the morphism $\newRSFTA(\Lag, \Partition)$ is an invariant of $(\Lag, \Partition)$.
\end{thm}

Here \emph{filtered stable tame isomorphism} is a refinement of stable tame isomorphism for mfDGAs described in \S \ref{Sec:mfDGA}. Using the techniques of \cite{Ekholm:Z2RSFT}, chain homotopy invariance of $\newRSFTA$ is easily established. Filtered stable tame isomorphism provides a more precise notion of invariance, which simplifies analysis of augmentations. The proof is addressed in Appendix \ref{App:Invariance}. We say that a pair $(\Lag, \Partition)$ as described in the theorem is a \emph{partitioned Lagrangian cobordism} which we'll denote 
\begin{equation*}
(\Lag, \Partition): (\LegUp, \Partition^{+}) \rightarrow (\LegDown, \Partition^{-}).
\end{equation*}
Our convention is that Lagrangian cobordisms point downward in symplectizations so that $CE$ and $\newRSFTA$ are covariant functors.

Applying homology to a $\filtration^{\filtLevel}\newRSFTA$, we obtain a \emph{max-filtered graded algebra} (mfGA), which we call the \emph{$\filtLevel$th planar diagram homology}, denoted
\begin{equation*}
\newRSFT^{\filtLevel} = \newRSFT^{\filtLevel}(\Leg, \Partition).
\end{equation*}
Here is the homological version of the above theorem:

\begin{thm}
The mfGA isomorphism class of $\newRSFT^{\filtLevel}(\Leg, \Partition)$ is an invariant of the Legendrian isotopy class of $\LegGrouped$. An exact Lagrangian cobordism $(\Lag, \Partition)$ as above induces a mfGA morphism
\begin{equation*}
\newRSFT(\Lag, \Partition): \newRSFT^{\filtLevel}\LegGroupedUp \rightarrow \newRSFT^{\filtLevel}\LegGroupedDown,
\end{equation*}
which is an invariant of $(\Lag, \Partition)$ up to homotopy as in Theorem \ref{Thm:Main}.
\end{thm}

The following is then immediate from the above theorems and the definition of the torsions in \S \ref{Sec:mfDGA}.

\begin{thm}
The torsions $\tau_{A}, \tau_{H} \in \Z_{\geq 0} \cup \{ \infty\}$ are such that $\tau_{A}(\emptyset) = \tau_{H}(\emptyset) = \infty$ and $\tau_{A}\LegGrouped \leq \tau_{H}\LegGrouped$ for all $\LegGrouped$. If there exists a $\LagGrouped: \LegGroupedUp \rightarrow \LegGroupedDown$ then
\begin{equation*}
\tau_{A}\LegGroupedDown \leq \tau_{A}\LegGroupedUp, \quad \tau_{H}\LegGroupedDown \leq \tau_{H}\LegGroupedUp.
\end{equation*}
In particular if one of $\tau_{A}\LegGrouped$ or $\tau_{H}\LegGrouped$ is finite, then $\LegGrouped$ can not be filled by any $\LagGrouped$.
\end{thm}

As expected, $\newRSFTA$ will depend on choices of almost complex structures which become irrelevant in light of deformation invariance. We assume throughout that $c_{1}(\SympBase) = 0$. When $H_{1}(\SympBase) \neq 0$, the gradings of the algebras vary according to homotopy classes of trivializations of some circle bundles over $\SympBase$. See \S \ref{Sec:Gradings}.

In this article we work with coefficients $\ring = \field$ or $\ring = \field[H_{1}(\Leg)]$ with $\field = \Z/2\Z$.\footnote{Other twisted systems -- such as $\field[H_{2}(M, \Leg)]$ used in knot contact homology -- are also available but will not be addressed.} As is the case with $\CE$, choices of coefficient systems together with Maslov classes of Legendrians impose restrictions on the possible gradings. Although this will not be addressed here, our invariants can be upgraded to $\Z$ and $\Z[H_{1}(\Leg)]$ coefficients using the techniques of \cite{EES:Orientations, ENS:Orientations, Karlsson:Orientations}.

Basic computations of $\newRSFTA$ and $\newRSFT$ are carried out -- mostly for $2$-component links -- in \S \ref{Sec:Computations}. Perhaps surprisingly, $\newRSFT$ vanishes for many partitioned links $\LegGrouped$ for which $LCH(\Leg) \neq 0$, meaning that that $\Leg$ does not admit a disconnected Lagrangian filling. See \cite{CGKS:Polyfilling} for examples of disconnected fillings, one of which we study in \S \ref{Sec:Polyfilling}.

\subsection{Background and motivation}

Now we delve into slightly more detail, assuming familiarity with the basics of symplectic topology and $\CE$. To simplify, we use $\field$ coefficients for the remainder of this section. Our intention is that by the end of this introduction, readers who are intimately familiar with $\CE$ will have a good idea as to how $\newRSFTA$ works and be able to skip ahead to \S \ref{Sec:mfDGA} to read more about the underlying algebra or \S \ref{Sec:Computations} to see some examples. Readers can also perform their own calculations using software we've written to accompany this project \cite{Avdek:Software}.

Fix a $(2n-1)$-dimensional contact manifold $\Mxi$ with contact form $\alpha$ containing an oriented Legendrian $\Leg$. Throughout this paper we'll assume that
\begin{equation*}
M = \R_{t} \times W, \quad \alpha = dt + \beta, \quad \xi = \ker \alpha
\end{equation*}
for an exact symplectic manifold $(W, \beta)$ so that the Reeb vector field is $\partial_{t}$. See \S \ref{Sec:AmbientSetup} for details. The reader is encouraged to pretend that $\Mxi$ is specifically the standard contact Euclidean space $\Rtwonminusone$ defined by
\begin{equation*}
\R^{2n-1} = \R_{t} \times \C_{x, y}^{n-1}, \quad \alpha_{std} = dt - \sum y_{i}dx_{i}, \quad \xi_{std} = \ker(\alpha_{std}).
\end{equation*}
Further specifying $n = 2$ is even better. The usual genericity and transversality assumptions for chords and holomorphic curves are in play for now and will be addressed in detail later in the text.

The $\CE$ differential is defined by counting holomorphic disks with one positive puncture and any number of negative punctures \cite{Chekanov:LCH, Eliashberg:LCH}. Like products in wrapped Floer homology \cite{AS:Wrapped} or differentials of Cthulhu invariants \cite{Cthulhu, Legout:Cthulhu}, the differentials of Legendrian RSFT invariants count disks with multiple positive punctures \cite{CL:SFTStringTop, Ekholm:Z2RSFT, Ng:RSFT}.

When multiple positive boundary punctures of a holomorphic map are allowed to touch the same connected component of a Lagrangian cylinder $\R \times \Leg$ in a symplectization, families of index $2$ disks may degenerate into single level SFT buildings \cite{SFTCompactness} whose domains are nodal Riemann surfaces. As described in \cite{CL:SFTStringTop}, such degenerations can be organized using string topology \cite{StringTopology}. String topological degenerations are dealt with combinatorially for $(n-1 = 1)$-dimensional Legendrians in \cite{Ng:RSFT}, resulting in \emph{curved} DGAs with $\partial^{2} \neq 0$ in general.\footnote{Cyclic versions of Ng's RSFT are non-curved DGAs.} The $n - 1 > 1$ case requires analysis of high dimensional moduli spaces of holomorphic disks. Moreover RSFT disks may be multiply covered, further complicating the underlying analytical theory.

\begin{figure}[h]
\begin{overpic}[scale=.35]{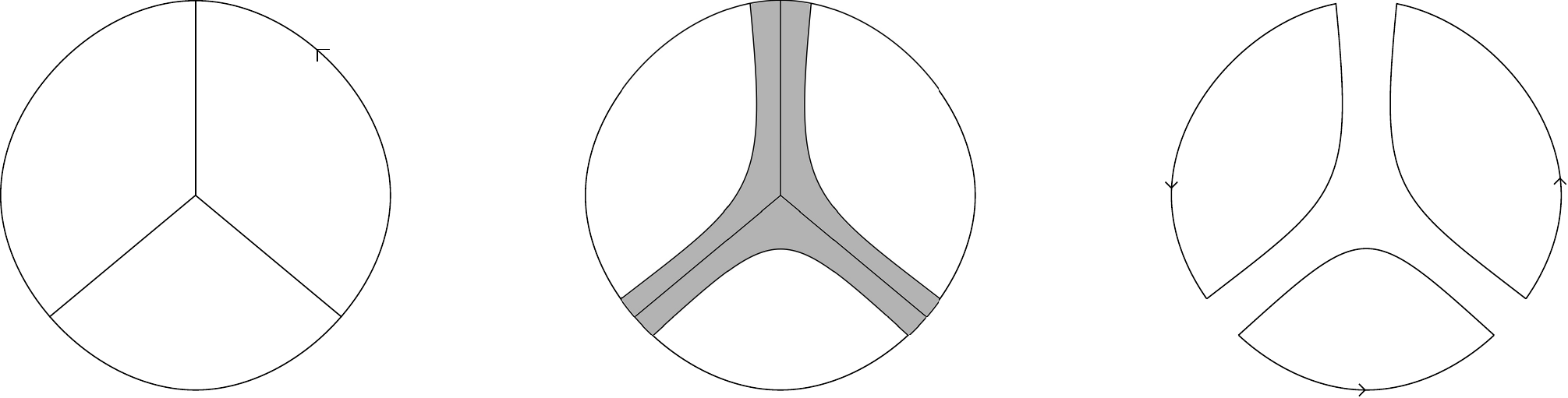}
\end{overpic}
\caption{A string triple coproduct $\mu_{3}^{1}$ is described by a chord diagram (left). The thickening of the diagram (center) together with cylinders over input and output circles (right) yields a genus $0$ cobordism.}
\label{Fig:StringCop}
\end{figure}

Like the RSFT of \cite{Ekholm:Z2RSFT}, $\newRSFTA$ is designed to avoid these string topological degeneracies (and subsequent analytical difficulties) by only counting disks with multiple positive punctures when they touch distinct connected Legendrians. This is not so unnatural from a Floer or Morse theoretical perspective, where multiple Lagrangians or Morse functions are required to define product structures \cite{AS:Wrapped, BC:Bilinearized, Legout:Cthulhu, AugSheaves, Schwarz:Morse}. Partitions of Legendrians also appear in the link gradings of  \cite{Mishachev, AugSheaves}. String topology still plays an essential motivating role in defining the $\newRSFTA$ differential and its cobordism maps.

Let $L$ be a smooth, oriented manifold with free loop space $\LoopSpace$. In \cite{CS:StringDiagrams} Chas and Sullivan use certain fatgraphs they call \emph{chord diagrams} to organize string topology operations \cite{StringTopology} on (tensor powers of) the homology of $\LoopSpace$. The left-hand side of Figure \ref{Fig:StringCop} depicts a chord diagram for coproduct operations
\begin{equation*}
\mu_{k}^{1}: H_{d}(\LoopSpace) \rightarrow \bigoplus \left( H_{i_{1}}(\LoopSpace) \otimes \cdots \otimes H_{i_{k}}(\LoopSpace) \right), \quad \sum_{j = 1}^{k} i_{j} = d - k\dim(L)
\end{equation*}
in the case $k = 3$.

The outer loop in the diagram represents a $d$-chain in $\LoopSpace$ with interior edges representing self-intersections of loops in $L$. When a loop in the $d$-chain satisfies the self-intersection determined by the interior edges, we split it up into segments determining $k$ loops in $L$. The thickening of the chord diagram sweeps out a cobordism between the input loop and the output loops.

\subsection{Overview of $\newRSFTA$}

A slight variation of the chord diagram picture for the loop coproduct operators takes a holomorphic disk with any numbers of positive and negative punctures and defines from it a ``one in, many out'' operator on an algebra of Reeb chords associated to a Legendrian submanifold of $\Mxi$. In more colorful language, we'll get a curved $A_{\infty}$ coalgebra, the type of algebraic object underlying $\CE$. An advantage of finding inspiration in \cite{CS:StringDiagrams} is that pictures will perform a great deal of mathematical labor for us.

Let $\Leg = \sqcup_{1}^{N} \Leg_{k}$ be a Legendrian submanifold of $\Mxi$. The $\Leg_{i}$ are closed, connected, and grouped together as follows: Let $\sim_{\Partition}$ be an equivalence relation on $\{1, \dots, N\}$ which partitions it into $N^{\Partition} \geq 1$ non-empty subsets $\Partition_{1} \sqcup \dots \sqcup \Partition_{N^{\Partition}} = \{1, \dots, N\}$. Then write
\begin{equation*}
\Leg^{\Partition}_{j} = \cup_{i \in \Partition_{j}} \Leg_{i}.
\end{equation*}
The $\Leg^{\Partition}_{i}$ are all connected if and only if $N^{\Partition} = N$. We say that an ordered pair of chords $\chord_{1}, \chord_{2}$ on $\Leg$ are \emph{composable} if $\chord_{1}$ ends on the same $\Leg_{i}$ on which $\chord_{2}$ begins. The pair is \emph{$\Partition$ composable} if $\chord_{1}$ ends on the same $\Leg_{j}^{\Partition}$ on which $\chord_{2}$ begins. A $\chord$ is \emph{pure} if it is composable with itself and \emph{mixed} otherwise. A $\chord$ is \emph{$\Partition$ pure} if it is $\Partition$ composable with itself and \emph{$\Partition$ mixed} otherwise.

Let $\word = (\chord_{1}\cdots\chord_{\filtLevel})$ be a \emph{$\Partition$ cyclic word of chords} on $\Leg$. This means that each pair $\chord_{i}, \chord_{i+1}$ is $\Partition$ composable (with $k$ modulo $\filtLevel$). Say that each $\chord_{i}$ ends on $\Leg^{\Partition}_{j_{i}}$. To avoid the aforementioned string topological degenerations we'll want the $j_{i}$ to be pairwise distinct in which case we say that $\word$ is \emph{admissible}. In order that our algebra be non-commutative, it is essential that we \emph{do not} consider cyclic permutations of words to be equivalent: $(\chord_{1}\chord_{2}\cdots\chord_{\filtLevel}) \nsim (\chord_{2}\cdots\chord_{\filtLevel}\chord_{1})$.

\begin{figure}[h]
\vspace{3mm}
\begin{overpic}[scale=.5]{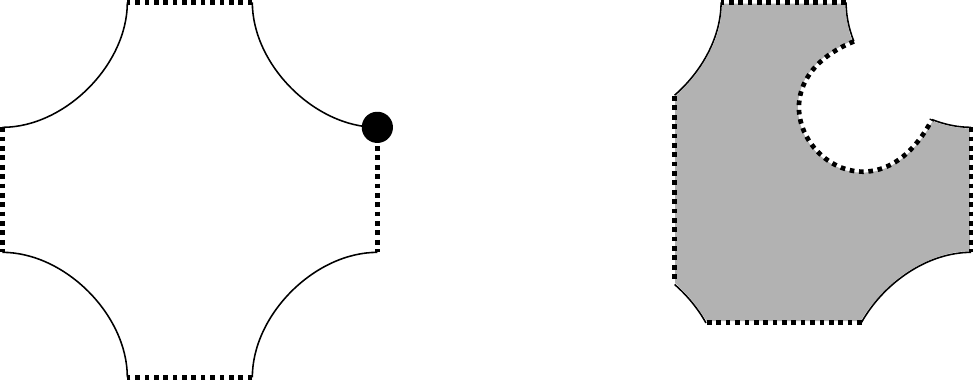}
\put(-8, 18){$\chord_{3}$}
\put(42, 18){$\chord_{1}$}
\put(18, 42){$\chord_{2}$}
\put(18, -5){$\chord_{4}$}
\put(32, 32){$j_{1}$}
\put(5, 32){$j_{2}$}
\put(5, 6){$j_{3}$}
\put(32, 6){$j_{4}$}

\put(62, 18){$\chord_{6}^{-}$}
\put(86, 26){$\chord_{5}^{-}$}
\put(102, 18){$\chord_{1}^{+}$}
\put(79, 42){$\chord_{2}^{+}$}
\put(79, 1){$\chord_{7}^{-}$}
\end{overpic}
\vspace{3mm}
\caption{On the left, a planar diagram $\planarDiagram(\word)$ for some $\word = (\chord_{1}\chord_{2}\chord_{3}\chord_{4})$. Integers indicate boundary labels so that $\chord_{1}$ begins on $\Leg_{j_{4}}$ and ends on $\Leg_{j_{1}}$. The marker at the tip of $\chord_{1}$ indicates the beginning of the word. On the right, a planar diagram $\planarDiagram(u)$ for some holomorphic disk $u$. Signs indicate which boundary punctures asymptotics are positive and negative.}
\label{Fig:CycWord}
\end{figure}

Our word $\word$ can be represented as the boundary of $2\filtLevel$-gon, $\planarDiagram(\word) \subset \C$, with half of its boundary segments drawn as dashed arcs representing the $\chord_{i}$ traversing $\partial \planarDiagram(\word)$ according to the (counter-clockwise) boundary orientation. The remaining boundary segments represent (indeterminate) capping paths -- as typically considered in $\CE$ -- connecting the ending point of $\chord_{i}$ to the starting point of $\chord_{i + 1}$.\footnote{When using enhanced coefficient systems, we will want to keep track of homotopy classes of capping paths. When the $\Leg^{\Partition}_{j}$ are disconnected, the capping paths may need to jump across components of $\Leg$. This can be managed with additional choices of \emph{connecting paths}. See \S \ref{Sec:ChordBasics}.} We say that $\planarDiagram(\word)$ is the \emph{planar diagram} for $\word$.\footnote{Readers familiar with \cite{Ekholm:Z2RSFT} can think of a planar diagram as a drawing of a ``formal disk''.} See the left-hand side of Figure \ref{Fig:CycWord}.

Let $u = (u_{\R}, u_{M}): \disk \setminus \{ \boundaryPunc_{j} \} \rightarrow \R_{s} \times M$ be a holomorphic map from a disk with boundary punctures $\boundaryPunc_{j} \in \partial \disk$ into the symplectization of $M$, with its boundary arcs mapped to $\Lag = \R_{s} \times \Leg$ and the $\boundaryPunc_{j}$ asymptotic to $\partial_{t}$ chords of $\Leg$ via $u$. Here and throughout, we use $\disk$ for the unit disk in $\C$ whose interior is denoted $\B$. A $\boundaryPunc_{j}$ is declared \emph{positive} (\emph{negative}) if $\lim_{z \rightarrow \boundaryPunc_{j}} u_{\R} = \infty$ (respectively, $-\infty$). A \emph{planar diagram for $u$} is a $2\#(z_{j})$-gon, $\planarDiagram(u) \subset \C$ with dashed boundary segments associated to the $\boundaryPunc_{j}$ and solid boundary segments associated with components of $\disk \setminus \{ \boundaryPunc_{j} \}$. Dashed arcs are decorated with chords and signs indicating positive or negative asymptotics. See the right-hand side of Figure \ref{Fig:CycWord}.

\begin{figure}[h]
\vspace{3mm}
\begin{overpic}[scale=.5]{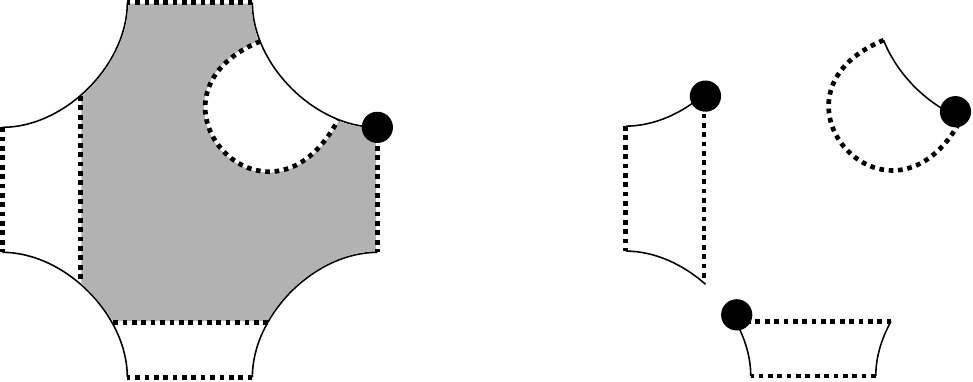}
\put(42, 20){$\chord_{1}$}
\put(18, 42){$\chord_{2}$}
\put(-6, 20){$\chord_{3}$}
\put(10, 20){$\chord_{6}$}
\put(18, 20){$\chord_{5}$}
\put(18, 8){$\chord_{7}$}
\put(18, -3){$\chord_{4}$}
\put(50, 21){\vector(1, 0){6}}
\put(59, 20){$\chord_{3}$}
\put(74, 20){$\chord_{6}$}
\put(83, 20){$\chord_{5}$}
\put(83, 8){$\chord_{7}$}
\put(83, -3){$\chord_{4}$}
\end{overpic}
\vspace{3mm}
\caption{An inscription of planar diagrams from Figure \ref{Fig:CycWord} determines an association $(\chord_{1}\chord_{2}\chord_{3}\chord_{4}) \mapsto (\chord_{5} )\cdot (\chord_{6} \chord_{3}) \cdot (\chord_{7}\chord_{4})$.}
\label{Fig:Differential}
\end{figure}

\begin{defn}\label{Def:Inscription}
An \emph{inscription of $\planarDiagram(u)$ into $\planarDiagram(\word)$} is an inclusion $\planarDiagram(u) \subset \planarDiagram(\word)$ such that the arcs in $\planarDiagram(u)$ representing positive chords are sent to their corresponding chords in $\planarDiagram(\word)$ and such that the arcs in $\planarDiagram(u)$ representing components of $\partial (\disk \setminus \{ \boundaryPunc_{j} \})$ are contained within the (non-dashed) arcs of $\partial \planarDiagram(\word)$ not labeled by chords.
\end{defn}

If an inscription $\planarDiagram(u) \subset \planarDiagram(\word)$ exists it is unique up to isotopy. Observe that the way in which positive chords and non-chord arcs of $\planarDiagram(u)$ are mapped into $\planarDiagram(\word)$ completely determines how the negative chords of $\planarDiagram(u)$ are mapped into $\planarDiagram(\word)$.

\begin{ex}\label{Ex:LengthOne}
If a $\word$ consists only of $\Partition$ pure chords, then it must have word length $1$ and $\planarDiagram(\word)$ is a bigon with one dashed arc and one solid arc. Say $\word = (\chord)$ is such a word with $\chord$ beginning and ending on some $\Leg^{\Partition}_{j}$. Given a holomorphic disk $u$, $\planarDiagram(u)$ can be inscribed in $\planarDiagram(\word)$ if and only if $u$ has a single positive puncture given by $\chord$ and all of the negative punctures of $u$ are chords which also begin and end on $\Leg^{\Partition}_{j}$. In other words, $u$ contributes to the differential of $\chord$ for $\CE(\Leg^{\Partition}_{j})$.
\end{ex}

Provided an inscription $\planarDiagram(u) \subset \planarDiagram(\word)$ with $\word$ admissible, the complement $\planarDiagram(\word) \setminus \planarDiagram(u)$ is a disjoint union of planar diagrams for admissible $\Partition$ cyclic words of chords, and these words come with an ordering determined by the boundary orientation of $\planarDiagram(u)$. Details are provided in \S \ref{Sec:MuOperators}.

Say this ordered collection of words is $\word_{1}, \cdots, \word_{m_{-}}$ where $m_{-}$ is the number of negative punctures of $u$. We view the inscription as associating $\word \mapsto 1$ when $m_{-} = 0$ (whence $\planarDiagram(u) = \planarDiagram(\word))$ and $\word \mapsto \word_{1}\cdots \word_{m_{-}}$ otherwise. See Figure \ref{Fig:Differential}.

Writing $\orbitVS$ for the free $\field = \Z/2\Z$ module generated by the admissible cyclic words of chords, our disk $u$ determines an operator
\begin{equation*}
\mu_{u}: \orbitVS \rightarrow (\orbitVS)^{\otimes m_{-}}
\end{equation*}
such that $\mu_{u}(\word) = 0$ if we cannot inscribe $\planarDiagram(u)$ into $\planarDiagram(\word)$ and with the convention that $\orbitVS^{\otimes 0} = \field$. After grading the $\word$, weighted counts of the $\mu_{u}(\word)$ over all $\ind = 1$ holomorphic maps $u$ determine a differential operator
\begin{equation*}
\partial: \Algebra \rightarrow \Algebra, \quad \Algebra = \tensorAlg(\orbitVS), \quad \deg(\partial) = -1, \quad \partial^{2} = 0
\end{equation*}
on the tensor algebra of $\orbitVS$, extending the domain of $\partial$ from $\orbitVS \subset \Algebra$ to all of $\Algebra$ by the Leibniz rule.

The admissibility condition for words together with the fact that $\Leg$ has only a finite number of chords implies that $\orbitVS$ is finitely generated. It comes equipped with an ascending filtration
\begin{equation*}
0 = \filtration^{0}\orbitVS \subset \filtration^{1}\orbitVS \subset\cdots \subset \filtration^{N^{\Partition}}\orbitVS = \filtration^{N^{\Partition} + 1}\orbitVS = \cdots = \orbitVS
\end{equation*}
where each $\filtration^{\filtLevel}\orbitVS$ is generated by words of length $\leq \filtLevel$. The induced filtration on $\Algebra$ is also denoted $\filtration$,
\begin{equation}\label{Eq:AlgebraFilt}
\filtration^{\filtLevel}\Algebra = \tensorAlg(\filtration^{\filtLevel}\orbitVS), \quad \field = \filtration^{0}\Algebra \subset \filtration^{1}\Algebra \subset \cdots \subset \filtration^{N^{\Partition}}\Algebra = \filtration^{N^{\Partition} + 1}\Algebra = \cdots = \Algebra.
\end{equation}
By construction the filtration satisfies
\begin{equation}\label{Eq:FiltrationProduct}
\filtration^{\filtLevel}\Algebra \cdot \filtration^{\filtLevel'}\Algebra \subset \filtration^{\max(\filtLevel, \filtLevel')}\Algebra
\end{equation}
and the definition of inscription ensures that our differential is filtration preserving,
\begin{equation}\label{Eq:FiltrationPreservingDel}
\partial (\filtration^{\filtLevel}\Algebra) \subset \filtration^{\filtLevel}\Algebra,
\end{equation}
so that each $(\filtration^{\filtLevel}\Algebra, \partial)$ is a unital DGA and each inclusion $\filtration^{\filtLevel}\Algebra \rightarrow F^{\filtLevel+1}\Algebra$ is a unital DGA morphism.

\begin{defn}
A filtered graded algebra $(\Algebra, \filtration)$ satisfying Equation \eqref{Eq:FiltrationProduct} is a \emph{max-filtered graded algebra} (mfGA). A triple $(\Algebra, \partial, \filtration)$ for which $(\Algebra, \filtration)$ is a max-filtered algebra and for which $(\Algebra, \partial)$ is a DGA satisfying Equation \eqref{Eq:FiltrationPreservingDel} is a \emph{max-filtered DGA} (mfDGA).
\end{defn}

See \S \ref{Sec:mfDGA} for further details. With the above notation, we define the planar diagram algebra and planar diagram homologies,
\begin{equation*}
\newRSFTA = (\Algebra, \partial, \filtration), \quad \filtration^{\filtLevel}\newRSFTA = (\filtration^{\filtLevel}\Algebra, \partial, \filtration), \quad \newRSFT^{\filtLevel} = H(\filtration^{\filtLevel}\newRSFTA).
\end{equation*}
By forgetting the orderings of words in $\Algebra$ we obtain a commutative mfDGA, $\newRSFTA^{\com}$. By forgetting both the orderings of words and the cyclic orderings of chords in each word, a cyclic version, $\newRSFTA^{\cyc}$, is determined. The three algebras are related by quotient maps
\begin{equation*}
( \newRSFTA = \tensorAlg(\orbitVS) ) \rightarrow ( \newRSFTA^{\com} = \tensorAlgCom(\orbitVS) ) \rightarrow ( \newRSFTA^{\cyc} = \tensorAlgCom(\orbitVS^{\cyc}) ).
\end{equation*}
where $\tensorAlgCom$ is the exterior algebra and $\orbitVS^{\cyc}$ is $\orbitVS$ modulo cyclic equivalence of words. There are analogous relationships between the algebras of \cite{Ng:RSFT}. See \S \ref{Sec:DiffsAlgStructures}.

\subsection{Subalgebra structures}

Passing to homology, the (injective) inclusion morphisms of Equation \eqref{Eq:AlgebraFilt} determine a sequence of (not necessarily injective) mfGA morphisms,
\begin{equation*}
\field = \newRSFT^{0} \rightarrow \newRSFT^{1} \rightarrow \cdots \rightarrow \newRSFT^{N^{\Partition}} = \newRSFT^{N^{\Partition}+1} = \cdots \newRSFT^{\infty} =: \newRSFT.
\end{equation*}

In addition to using word length filtrations, geometrically relevant subalgebras of $\newRSFTA$ are generated by pieces of $\Leg$. While relabeling of indices $j$ of $\Leg^{\Partition}_{j}$ induces isomorphisms of $\newRSFTA$, we can such encode subalgebras by fixing this indexing data. An ordering $\ord$ of the indices $j$ for the $\Leg^{\Partition}_{j}$ determines a filtration $\filtration_{\ord}$ of $\newRSFTA$ where $\filtration^{k}_{\ord}\newRSFTA$ is generated by admissible $\Partition$ cyclic chords with endpoints on $\sqcup_{j=1}^{k} \Leg^{\Partition}_{j}$. Then the inclusion morphisms for the $\filtration$ and $\filtration_{\ord}$ fit together as
\begin{equation*}
\begin{tikzcd}
\filtration^{\filtLevel}\filtration_{\ord}^{k}\Algebra \arrow[r]\arrow[d] & \filtration^{\filtLevel + 1 }\filtration_{\ord}^{k}\Algebra \arrow[d]\\
\filtration^{\filtLevel}\filtration_{\ord}^{k+1}\Algebra \arrow[r] & \filtration^{\filtLevel + 1}\filtration_{\ord}^{k+1}\Algebra,
\end{tikzcd} \quad \filtration^{\filtLevel}\filtration_{\ord}\Algebra = \filtration_{\ord}\filtration^{\filtLevel}\Algebra.
\end{equation*}
Clearly each $\filtration^{k}_{\ord}\newRSFTA = (\filtration_{\ord}^{k}\Algebra, \filtration_{\ord})$ is a mfGA and the definition of inscription ensures that $\partial$ preserves $\filtration_{\ord}$. So each $(\filtration^{\filtLevel}\filtration_{\ord}^{k}\Algebra, \partial)$ is a DGA with two max-filtration structures and the above equation is a commutative diagram of mfDGA morphisms.

In particular, we have mfDGA functoriality for $\newRSFTA$ and mfGA functoriality for $\newRSFT$ with respect to inclusions,
\begin{equation*}
\begin{gathered}
\newRSFTA(\sqcup_{j=1}^{k}\Leg^{\Partition}_{j}, \Partition) \rightarrow \newRSFTA(\sqcup_{j=1}^{k+1}\Leg^{\Partition}_{j}, \Partition),\\
\newRSFT^{\filtLevel}(\sqcup_{j=1}^{k}\Leg^{\Partition}_{j}, \Partition) \rightarrow \newRSFT^{\filtLevel}(\sqcup_{j=1}^{k+1}\Leg^{\Partition}_{j}, \Partition).
\end{gathered}
\end{equation*}
Consequently $\newRSFTA$ is amenable to direct limit constructions by consideration of Legendrians with infinitely many pieces. Note that $LCH$ and the RSFT of \cite{Ng:RSFT} are not functorial with respect to inclusions of Legendrians into one another.\footnote{Consider for example a two copy $\Leg_{2}$ of a stabilized Legendrian knot $\Leg$ so that $LCH(\Leg) = 0$. As $\Leg_{2}$ is fillable by an annulus (cf. \cite{EHK:LagrangianCobordisms, NT:Torus}), $LCH(\Leg_{2}) \neq 0$ so there cannot exist a unital morphism $LCH(\Leg) \rightarrow LCH(\Leg_{2})$.}

\subsection{Filtration-enhanced numerical invariants}

Due to their non-commutativity, (free) DGAs are difficult to use directly as invariants even when differentials can be explicitly computed. Thankfully the $\CE$ literature provides a variety of techniques from commutative algebra to distinguish stable tame isomorphism classes of free DGAs which can be implemented by hand (in simple cases) or computer. See \cite{Ng:ComputableInvariants}.

In \S \ref{Sec:mfDGA} we demonstrate how known invariants of DGAs can be upgraded for mfDGAs such as $\newRSFTA$:
\be
\item The vanishing or non-vanishing of the homology of a DGA is upgraded to the \emph{H-torsion}, $\tau_{H}(\Algebra, \partial, \filtration)$.
\item The existence (or lack thereof) of an augmentation is upgraded to the \emph{A-torsion}, $\tau_{A}(\Algebra, \partial, \filtration)$.
\item The set of homotopy classes of augmentations \cite{Chekanov:LCH} is upgraded to the \emph{augmentation tree}, $\tree_{\Aug}$.
\ee
These are all filtered stable tame isomorphism invariants.

A pair of augmentations $\aug_{0}, \aug_{1}$ of some $\filtration^{\filtLevel}\newRSFTA$ determines a bilinearized chain complex \cite{BC:Bilinearized, Chekanov:LCH} and so a single-variable Poincar\'{e} polynomial
\begin{equation*}
P_{\aug_{0}, \aug_{1}}, P^{\aug_{0}, \aug_{1}} \in \Z[t^{\pm 1}]
\end{equation*}
which count dimensions of associated homology and cohomology groups. Since the bilinearized complexes are filtered by $\filtration$, these Poincar\'{e} polynomials can be upgraded to three-variable polynomials
\begin{equation}\label{Eq:SpecPolynomial}
P^{\spec}_{\aug_{0}, \aug_{1}}, P_{\spec}^{\aug_{0}, \aug_{1}} \in \Z[t^{\pm 1}, x, y]
\end{equation}
by counting dimensions of spectral sequence homology and cohomology groups $E^{r}_{p, q}(\aug_{0}, \aug_{1})$ and $E^{p, q}_{r}(\aug_{0}, \aug_{1})$. As our chain complexes are finitely generated, the spectral sequences converge after $N^{\Partition}$ pages and as we're working over a field, the $P^{\spec}_{\aug_{0}, \aug_{1}}, P_{\spec}^{\aug_{0}, \aug_{1}}$ determine the $P_{\aug_{0}, \aug_{1}}, P^{\aug_{0}, \aug_{1}}$. These polynomials are invariants of the homotopy classes of $\aug_{0}, \aug_{1}$ and the collections of all such polynomials is a filtered stable tame isomorphism invariant. It's well known that the single variable polynomials associated to linearized homologies (when $\aug_{0} = \aug_{1}$) are powerful enough to distinguish many smoothly isotopic but Legendrian non-isotopic knots \cite{KnotAtlas}.

\subsection{Further comparison with existing theories}

While $\newRSFTA$ is an RSFT invariant in the sense that it counts disks with multiple positive punctures, its algebraic properties indicate that it is closer in nature to $\CE$ than existing RSFT constructions. We further compare with \cite{Ekholm:Z2RSFT} and \cite{Ng:RSFT}, assuming familiarity with the Weyl algebra formalism for (closed string) RSFT from \cite{EGH:SFTIntro} which assigns two variables $p_{\gamma}$ and $q_{\gamma}$ to each closed Reeb orbit in a contact manifold (not necessarily of the form $\R_{t} \times \SympBase$).

First let's address analytical issues: The admissibility of the $\word$ ensures that when $\mu_{u}(\word) \neq 0$ then $u$ must be admissible in the sense of \cite{Ekholm:Z2RSFT} and in particular $u$ must be somewhere injective. Hence we can prove that $\partial^{2} = 0$ without having to worry about string-topological corrections or multiple covers. In fact, $\newRSFTA$ moduli spaces for disks contributing to $\partial$ are homeomorphic to $\CE$ moduli spaces as described in Lemma \ref{Lemma:LCHShift} (which is dependent on the fact that our ambient contact manifold is a contactization). Therefore well-established compactness and transversality results are applicable to holomorphic curves contributing to $\partial$ and cobordism maps \cite{Ekholm:Z2RSFT, EES:LegendriansInR2nPlus1, EES:PtimesR}. There are no analytical innovations in this article.

As previously mentioned, $\newRSFTA$ counts the same holomorphic disks used to define the RSFT of \cite{Ekholm:Z2RSFT} which we will call $RSFT_{E}$. It is a spectral sequence which can be packed as a triply graded homology group $RSFT_{E} = \oplus_{p, q, r}(RSFT_{E})^{p, q}_{r}$. The chain level generators of $RSFT_{E}$ are ``formal disks'' whose boundaries are sequences of positive and negative chords spanning a vector space we'll call $\orbitVS_{E}$. Modulo our reorganization of homotopy classes of maps using group ring coefficients, the generators of $\orbitVS_{E}$ correspond to the \emph{boundary words} of \S \ref{Sec:BoundaryWords} which are considered equivalent up to cyclic rotation both here and in \cite{Ekholm:Z2RSFT}.

The generators of $\orbitVS_{E}$ with only positive chords corresponds to our quotient $\orbitVS^{\cyc}$ of $\orbitVS$ \cite[\S 2.1]{Ekholm:Z2RSFT}. Cyclic rotations aside, this roughly parallels the way in which the closed string RSFT of \cite{Hutchings:QOnly, MZ:Torsion, Siegel:RSFT} uses only the $q$-variables of the RSFT of \cite{EGH:SFTIntro} to generate its underlying algebra. The RSFT of \cite{Ng:RSFT}, which we'll call $RSFT_{N}$, also uses both $p$ and $q$ variables -- one for each chord.

The algebraic formalisms of $RSFT_{E}$ and $RSFT_{N}$ are such that differentials are not action decreasing and so the differential of a single generator may contain infinitely many summands. Like $\CE$, closed string contact homology \cite{BH:ContactDefinition, Pardon:Contact}, or the RSFT of \cite{Hutchings:QOnly, MZ:Torsion, Siegel:RSFT}, the $\newRSFTA$ differential is action decreasing so that only finite numbers of holomorphic disks are counted. Therefore there is no need for (Novikov or $p$-adic) completion to define $\newRSFTA$.\footnote{Note that for $RSFT_{E}$ holomorphic disks are counted cohomologically (with ``inputs'' being negative chords and ``outputs'' being positive chords) so that differentials are not action decreasing and Novikov completion is employed.}

On the subject of commutativity: Like $RSFT_{N}$, $\newRSFTA$ is non-commutative whereas $RSFT_{E}$ is an abelian object. Because $\newRSFTA$ is non-commutative and finitely generated, the constructions of bilinearized (co)homology theories and augmentation categories of \cite{BC:Bilinearized, CEKSW:AugCat} can be applied directly without modification.\footnote{We do not address possible analogues of the $\Aug_{+}$ category of \cite{AugSheaves} for $\newRSFTA$ in this article. Currently $\Aug_{+}$ is only defined for $\CE$ with $1$-dimensional Legendrians.}

The linearized cohomology for $\newRSFTA^{\cyc}$ (associated to a single augmentation) recovers the Lagrangian invariants of \cite{Ekholm:Z2RSFT} as follows: Suppose that $\aug_{\Lag}$ is an augmentation of $\newRSFTA^{\cyc}$ determined by an exact, partitioned Lagrangian filling $\LagGrouped$. Equip $\orbitVS^{\cyc}$ with a descending filtration $\filtration^{\filtLevel}_{op}$, where the $\filtration^{\filtLevel}_{op}\orbitVS^{\cyc}$ are additively generated by words of word length $\geq \filtLevel$. Then the linearized cohomology differential
\begin{equation*}
d_{\aug_{L}}: \orbitVS^{\cyc}_{\ast} \rightarrow \orbitVS^{\cyc}_{\ast + 1}
\end{equation*}
preserves the ascending filtration, yielding a filtered abelian cochain complex and spectral sequence cohomology groups
\begin{equation*}
E^{p, q}_{r}(\orbitVS^{\cyc}, d_{\aug_{\Lag}}) = (RSFT_{E})^{p, q}_{r}.
\end{equation*}
Indeed $(\orbitVS^{\cyc}, d_{\aug_{\Lag}})$ is by definition the filtered complex of \cite{Ekholm:Z2RSFT} determined by $\LagGrouped$ and so determines the same homology.

To our knowledge, the constructions of bilinearized theories and augmentation categories of a free DGA over a finite field (such as $\field = \Z/2\Z$) depend essentially on non-commutativity, which is a feature of $\newRSFTA$. Linearization with chain level $A_{\infty}$ Fukaya-Massey product structures associated to a single augmentation are still available when working with a commutative free DGA such as $\newRSFTA^{\cyc}$ and the second $A_{\infty}$ relation states that the ``two-input, one-output'' operation descends to cohomology. Hence the $RSFT_{E}$ of a partitioned Lagrangian filling is naturally equipped with a pair-of-pants product (after a grading shift so that the product has degree $0$), making it a multiplicative spectral sequence.

\subsection{Organization of this article}

\S \ref{Sec:GeometricSetup} outlines the geometric structures of this article, establishes notation, and defines gradings on the generators of $\newRSFTA$. In \S \ref{Sec:BoundaryWords} we describe the $\mu$ operators in detail and describe moduli spaces of holomorphic disks required to define differentials and cobordism maps for $\newRSFTA$. In \S \ref{Sec:DiffsAlgStructures} we define these differentials and cobordism maps and then outline some basic algebraic structures of $\newRSFTA$. Essential properties of abstract mfDGAs -- including the definition of filtered stable tame isomorphism -- are covered in \S \ref{Sec:mfDGA}. \S \ref{Sec:Computations} covers elementary computations including generalities for working with links in $\Rthree$. Finally, invariance proofs appear in Appendix \ref{App:Invariance}.

\subsection{Acknowledgments}

\begin{wrapfigure}{r}{0.07\textwidth}
\centering
\includegraphics[width=.07\textwidth]{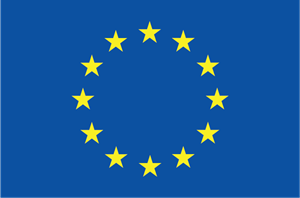}
\end{wrapfigure}

The first version of this article was completed at Uppsala University, where the author was partly supported by grant KAW 2016.0198 from the Knut and Alice Wallenberg Foundation. During revisions the author was supported by Cofund MathInGreaterParis and the Fondation Mathématique Jacques Hadamard as a member of the Laboratoire de Math\'{e}matiques d'Orsay at Universit\'{e} Paris-Saclay and by the CNRS as a member of Institute Math\'{e}matiques de Jussieu at Sorbonne Université. This project has received funding from the European Union’s Horizon 2020 research and innovation programme under the Marie Skłodowska-Curie grant agreement No 101034255.

We thank Georgios Dimitroglou Rizell and Lenhard Ng for interesting discussions and their guidance. We also thank Lu\'{i}s Diogo, Roman Golovko, and No\'{e}mie Legout for providing feedback on early drafts of this article as well as Joshua Sabloff for sharing Python code with us. This article has greatly benefited from the careful reading and thoughtful comments of our anonymous referee.

\section{Geometric setup, chords, and gradings}\label{Sec:GeometricSetup}

Here we describe the geometric objects of interest in this article.  Most of the material in this section is standard, following \cite{EES:PtimesR}. However, we'll want to pay special attention to labelings and partitions of components of Legendrian and Lagrangians as described in the introduction and in \cite{Ekholm:Z2RSFT}.

\subsection{Ambient geometric setup}\label{Sec:AmbientSetup}

Contact manifolds will be written $\Mxi$. The dimension of $M$ will always be $2n - 1$ and we assume $M$ is equipped with a contact form $\ContForm$ for which $\xi = \ker \ContForm$. We say that $\Mxi$ is a \emph{contactization} if
\begin{equation*}
M = \R_{t} \times \SympBase, \quad \ContForm = dt + \beta
\end{equation*}
for $\beta\in \Omega^{1}(\SympBase)$ a \emph{Liouville form}, meaning that $d\beta$ is symplectic on $\SympBase$. This implies that the Reeb vector field is $\partial_{t}$. Denote by $J_{\SympBase}$ a complex structure in $\SympBase$ which is compatible with $d\beta$ in the sense that $d\beta(J_{\SympBase}\ast, \ast)$ is a $J_{\SympBase}$-invariant Riemannian metric on $\SympBase$.

\begin{assump}
Contact manifolds in this paper will always be contactizations. We require that $(\SympBase, \beta, J_{\SympBase})$ has \emph{finite geometry at infinity} in the sense of \cite[Definition 2.1]{EES:PtimesR} and vanishing first Chern class, $c_{1}(\SympBase) = 0\in H^{2}(\SympBase, \Z)$.
\end{assump}

\begin{rmk}
As in \cite[Remark 1.4]{EES:PtimesR}, the $c_{1} = 0$ hypothesis could be removed at the expense of complicating our exposition. This condition is satisfied in the most studied cases when $M$ is a $1$-jet space of an oriented manifold \cite{Chekanov:LCH, EES:LegendriansInR2nPlus1}. This also includes the cases of unit cotangent bundles of Euclidean spaces studies in knot contact homology \cite{Ng:KCHIntro}.
\end{rmk}

The symplectization of $M$ is
\begin{equation*}
\R_{s} \times M \simeq \C_{s, t}\times \SympBase
\end{equation*}
which is equipped with the symplectic potential $e^{s}\alpha$. Except when we are defining cobordism maps, symplectizations are always equipped with almost complex structures $J$ determined by $J_{\SympBase}$ via the equations
\begin{equation*}
J\partial_{s} = \partial_{t}, \quad J|_{T\SympBase} = J_{\SympBase}.
\end{equation*}
With this choice,
\be
\item $J$ is invariant under $s$ and $t$ translations and
\item both of the projections onto the factors of $\C_{s, t} \times \SympBase$ are holomorphic.
\ee

\subsection{Legendrians and Lagrangian cobordisms}

Let $\Leg^{\pm}$ be Legendrians in $\Mxi$.

\begin{defn}
An \emph{exact Lagrangian cobordism}
\begin{equation*}
L: \LegUp \rightarrow \LegDown
\end{equation*}
is a submanifold $L \subset \R_{s} \times M$ such that there exists some $C \gg 0$ satisfying the following conditions:
\be
\item $L$ is a collar of the $\Leg^{\pm}$ outside of $[-C, C]\times M$:
\begin{equation*}
\begin{gathered}
\big((-\infty, -C]\times M \big) \cap L = (-\infty, -C] \times \Leg^{-} = L^{-},\\
\big([C, \infty) \times M\big) \cap L = [C, \infty) \times \Leg^{+} = L^{+}.
\end{gathered}
\end{equation*}
\item There exists $f \in \Cinfty(L)$ such that $e^{s}\alpha|_{TL} = df$ and there are constants $f^{\pm}$ such that $f|_{L^{\pm}} = f^{\pm}$.
\ee
\end{defn}

See \cite{Chantraine:DisconnctedEnds} regarding the second condition. Various types of cobordisms have special importance:
\be
\item $\Lag$ is an \emph{exact Lagrangian filling} if $\LegDown = \emptyset$.
\item $\Lag$ is a \emph{concordance} if it is homeomorphic to $\R_{s}\times \Leg^{+}$.
\item $\Lag$ is a \emph{homology cobordism} for $i=0, 1$ the inclusion maps $H_{i}(\LegUpDown) \rightarrow H_{i}(\Lag)$ are isomorphisms.
\ee
When $n=\dim \Lag = 2$, concordance and homology cobordism are equivalent notions. The most important example of a concordance is the \emph{Lagrangian cylinder}
\begin{equation*}
L_{\Leg} = \R_{s} \times \Leg
\end{equation*}
associated to a Legendrian $\Leg$. The trace of a Legendrian isotopy sweeps out a Lagrangian concordance as described in \cite[Theorem 1.1]{Chantraine:Concordance}, \cite[\S 6.1]{EHK:LagrangianCobordisms}.

\begin{assump}
Henceforth we assume that Legendrians and Lagrangians are equipped with partitions $\Partition$ as described in the introduction. Legendrians and Lagrangians are also equipped with orientations so that the orientation $\orientation(\Lag)$ of $\Lag: \Leg^{+} \rightarrow \Leg^{-}$ induces orientations on the $\Leg^{\pm}$ for which $\orientation(\Lag) = \partial_{s} \wedge \orientation(\Leg^{\pm})$ over collared ends.
\end{assump}

\subsection{Chords and words of chords}\label{Sec:ChordBasics}

The chords of $\Leg$ will be denoted by $\chord$ with actions
\begin{equation*}
\action (\chord) = \int_{\chord} \alpha > 0.
\end{equation*}
so that $\chord$ is naturally parameterized by $[0, \action(\chord)]$. For a chord $\chord$, we'll write the initial and terminal points of $\chord$ as $q^{-}(\chord)$ and $q^{+}(\chord)$ respectively with $q^{\pm}(\chord) \in \Leg$.

The chords of $\Leg$ are in one-to-one correspondence with double points of $\pi_{\SympBase}\Lambda$ where $\pi_{\SympBase}$ is the \emph{Lagrangian projection},
\begin{equation*}
\pi_{\SympBase}: \R_{t} \times \SympBase \rightarrow \SympBase
\end{equation*}
For each $\chord$, $\pi_{\SympBase}(\chord)$ is the associated double point. For such a $\chord$, subspaces $\pi_{\SympBase}T_{q^{\pm}(\chord)}\Leg \subset T_{\pi_{\SympBase}(\chord)}\SympBase$ are Lagrangian. We say that $\Leg$ is \emph{non-degenerate} if these subspaces are transverse.

The following technical definition helps us when studying moduli spaces of holomorphic curves. See example \cite{DR:Lifting, EES:LegendriansInR2nPlus1, EES:PtimesR}: The pair $(\Leg, J_{\SympBase})$ is \emph{admissible} if for each $\chord$ there is a neighborhood $U(\chord) \subset \SympBase$ about $\pi_{\SympBase}(\chord) \in \SympBase$ within which $J_{\SympBase}$ is integrable and $\pi_{\SympBase}\Leg$ is real-analytic.

\begin{assump}
Legendrians are non-degenerate and $(\Leg, J_{\SympBase})$ is admissible unless otherwise indicated.
\end{assump}

We have two ways of partitioning the chords of $\Leg$:
\begin{equation*}
\begin{gathered}
\Chords = \bigsqcup_{1 \leq j^{-}, j^{+} \leq N} \Chords_{j^{-}, j^{+}} = \bigsqcup_{1 \leq j^{-}, j^{+} \leq N_{\Partition}} \Chords^{\Partition}_{j^{-}, j^{+}},\\
\Chords_{j^{-}, j^{+}} = \{ \chord \ :\ q^{\pm}(\chord) \in \Leg_{j^{\pm}} \}, \quad \Chords_{j^{-}, j^{+}}^{\Partition} = \{ \chord \ :\ q^{\pm}(\chord) \in \Leg_{j^{\pm}}^{\Partition} \}
\end{gathered}
\end{equation*}

Recall the notions of \emph{pure}, \emph{mixed}, \emph{$\Partition$ pure}, and \emph{$\Partition$ mixed}, \emph{cyclic words of chords}, and \emph{admissible cyclic words of chords} from the introduction. We'll take the following notational shortcut moving forward:

\begin{notation}
Henceforth, a \emph{word} will refer to an admissible $\Partition$ cyclic word of chords.
\end{notation}

For a word $\word = (\chord_{1}\cdots \chord_{\filtLevel})$ the action is defined
\begin{equation*}
\action(\word) = \sum_{i=1}^{\filtLevel} \action (\chord_{i}).
\end{equation*}

\subsection{Maslov classes and gradings}\label{Sec:Gradings}

The words $\word$ are going to generate our algebra $\newRSFTA$ so we'll have to assign them gradings. The story here is analogous to that of Legendrian contact homology:
\be
\item $\Z/2\Z$ gradings are available for any choice of coefficient system.
\item Suppose the $\Leg^{\Partition}_{j}$ are each connected and $2\rho \in \Z$ is the divisibility of the Maslov class of $\Leg$ -- to be defined momentarily. Then we will get $\Z/2\rho \Z$ gradings when using $\field$ coefficients.
\item When the $\Leg^{\Partition}_{j}$ are each connected, $\field[H_{1}(\Leg, \Z)]$ coefficients yield $\Z$ gradings.
\item When the $\Leg^{\Partition}_{j}$ are not connected, we can still get $\Z/2\rho\Z$ gradings with $\field$ coefficients or $\Z$ gradings with $\field[H_{1}(\Leg, \Z)]$ coefficients with additional data, to be described below.
\ee
Again, matters are simplified by the assumption $c_{1}(T\SympBase) = 0$. 

To assign Maslov numbers and $\newRSFTA$ gradings to words $\word$ we will need to work out the ``additional data'' described above -- special paths and a trivialization of a circle bundle over $W$. The following constructions are standard in relative SFT.

\subsubsection{Special paths}

Choose a basepoint $q(\Leg_{i})$ in each connected component $\Leg_{i}$ of $\Leg$ which we assume is not a $q^{\pm}(\chord)$. If some $\Leg^{\Partition}_{j}$ is disconnected, for each pair of connected components $\Leg_{i_{1}}, \Leg_{i_{2}} \subset \Leg^{\Partition}_{j}$ choose an unoriented path $\gamma[i_{1}, i_{2}] = \gamma[i_{2}, i_{1}] \subset M$ connecting $q(\Leg_{i_{1}})$ to $q(\Leg_{i_{2}})$ whose interior lies in $M \setminus \Leg$. This determines oriented paths $\gamma(i_{1}, i_{2})$ and $\gamma(i_{2}, i_{1})$ from $q(\Leg_{i_{1}})$ to $q(\Leg_{i_{2}})$ and vice-versa, respectively, whose images agree with $\gamma[i_{1}, i_{2}]$. We call the $\gamma(i_{1}, i_{2})$ \emph{connecting paths}.

\begin{rmk}
To simplify matters, the reader may want to work under the assumption that the $\Leg^{\Partition}_{j}$ are connected in which case the connecting paths may be ignored.
\end{rmk}

Let $\chord \in \Chords$. Choose a path $\eta^{+}(\chord)$ in $\Leg$ starting at $q^{+}(\chord)$ and ending at the basepoint of $\Leg$ belonging to the same connected component of $\Leg$ as $q^{+}(\chord)$. Similarly choose a path $\eta^{-}(\chord)$ in $\Leg$ which starts at a basepoint and ends at $q^{-}(\chord)$. We call the $\eta^{\pm}(\chord)$ \emph{preferred base paths}.

Let $\chord_{i_{1}}, \chord_{i_{2}}$ be a $\Partition$ composable pair of chords with $\chord_{i_{1}} \in \Chords_{j_{1}, j}$ and $\chord_{i_{2}} \in \Chords_{j, j_{2}}$ for some index $j$. The \emph{preferred capping path for $\chord_{i_{1}}, \chord_{i_{2}}$} is the oriented path
\begin{equation*}
\cappingPath(i_{1}, i_{2}): [0, 1] \rightarrow M
\end{equation*}
which begins at $q^{+}(\chord_{i_{1}})$, travels to the basepoint $q(\Leg_{j})$ via $\eta^{+}(\chord_{i_{1}})$, traverses a connecting path if necessary, and ends at $q^{-}(\chord_{i_{2}})$ by traveling along $\eta^{-}(\chord_{i_{2}})$.

Given a word $\word = (\chord_{i_{1}}\dots\chord_{i_{\filtLevel}})$ we say that $(\cappingPath(i_{1}, i_{2}), \dots, \cappingPath(i_{\filtLevel}, i_{1}))$ is the \emph{preferred capping path} for $\word$ and denote it by $\cappingPath(\word)$. The chords and preferred capping paths can then be concatented to give a closed loop $\gamma(\word)$ in $M$,
\begin{equation}\label{Eq:WordLoop}
\gamma(\word) = \cappingPath(i_{1}, i_{2})\# \chord_{i_{2}}\# \dots \# \cappingPath(i_{\filtLevel}, i_{1}) \# \chord_{i_{1}}.
\end{equation}

When $\Lag: \LegUp \rightarrow \LegDown$ is a homology cobordism with connected components $\Lag_{i}$ we also choose paths $\gamma_{i}$ connecting the basepoint of $\Leg^{-}_{i}$ to the basepoint of $\Leg^{+}_{i}$. When $\Lag$ is more specifically a $\Lag_{\Leg} = \R_{s} \times \Leg$ we assume that connecting paths are of the form $\R_{s} \times \{q(\Leg_{i})\}$.

\subsubsection{Clockwise rotations and $\Z/2\Z$ gradings}

Let $V_{0}^{+}, V_{1}^{+}$ be an ordered pair of transverse Lagrangian subspaces of $\C^{n-1}$. We write $V_{k}, k=0, 1$ for the $V_{k}^{+}$ with their orientations forgotten. Following \cite[\S 3.1.2]{EES:Orientations}, we say that a Hermitian coordinate system $z = x + J y$ on $\C^{n-1}$ is a \emph{canonical coordinate system} for the pair $V^{+}_{0}, V^{+}_{1}$ if $V^{+}_{0} = \R^{n-1}_{x}$ (equipped with the orientation determined by a standard basis) and
\begin{equation}\label{Eq:LocalRotation}
V_{1} = e^{J\vec{\theta}}V_{0} = \Diag(e^{J\theta_{1}}, \dots, e^{J\theta_{n}})V_{0}, \quad \theta_{k} \in (-\pi, 0).
\end{equation}
We emphasize that the $\theta_{k}$ are negative.

The preceding data is sufficient to define $\Z/2\Z$ Maslov numbers and gradings to all words (without the assumption that the $\Leg^{\Partition}_{j}$ are connected). Assume that at each $\chord_{i}$ we have chosen canonical coordinate systems on $T_{\pi_{\SympBase}(\chord_{i})}\SympBase$.

\begin{defn}\label{Def:PosNegLagrangianSubspaces}
Define $\sigma(V^{+}_{0}, V^{+}_{1}) \in \{ 0, 1\}$ by
\begin{equation*}
e^{J\vec{\theta}}V^{+}_{0} = (-1)^{\sigma(V^{+}_{0}, V^{+}_{1})}V^{+}_{1}
\end{equation*}
as oriented spaces. For a chord $\chord$ define
\begin{equation*}
\sigma(\chord) = \sigma(\pi_{\SympBase}T_{q^{-}(\chord)}\Leg, \pi_{\SympBase}T_{q^{+}(\chord)}\Leg)
\end{equation*}
using the aforementioned canonical coordinate systems. For a word $\word = (\chord_{1}\cdots \chord_{\filtLevel})$, define
\begin{equation*}
\begin{gathered}
\Maslov_{2}(\word) = \sum_{i} \sigma(\chord_{i}) \in \Z/2\Z,\\
|\word|_{2} = n-3 + \filtLevel + \Maslov_{2}(\word) \in \Z/2\Z.
\end{gathered}
\end{equation*}
\end{defn}

\subsubsection{$\Z$ gradings}

Now we upgrade the $\Z/2\Z$-valued gradings, $|\word|_{2}$, to $\Z$-valued gradings, $|\word|$. Write $\LagGras_{n-1}$ for the space of linear Lagrangian subspaces of $\C^{n-1}$ and $\LagGras_{n-1}^{+}$ for the double cover of oriented Lagrangian subspaces. Using the $V_{k}, V^{+}_{k}$ and $\vec{\theta}$ as above, we say that
\begin{equation*}
e^{J\vec{\theta}t}V_{0}: [0, 1]_{t} \rightarrow \LagGras_{n-1}
\end{equation*}
is a \emph{CW rotation from $V_{0}$ to $V_{1}$}. We can also view $CW$ rotations as paths in $\LagGras_{n-1}^{+}$.

There are naturally defined bundles over $M$
\begin{equation*}
\LagGras^{+}_{n-1} \rightarrow \LagGras^{+}_{\SympBase} \rightarrow M, \quad \LagGras_{n-1} \rightarrow \LagGras_{\SympBase} \rightarrow M
\end{equation*}
whose fibers at $(t, q)$ are the Lagrangian Grassmanians of $T_{q}\SympBase$. Identifying
\begin{equation*}
\LagGras_{n-1}^{+}\simeq \Unitary(n-1)/\SOrthog(n-1),\quad \LagGras_{n-1}\simeq \Unitary(n-1)/\Orthog(n-1)
\end{equation*}
the $\C^{\ast}$-valued determinants $\det: \LagGras_{n-1}^{+} \rightarrow \Circletwopi$ induce isomorphisms on $H_{1}$ \cite[\S 2.2]{MS:SymplecticIntro}. This determines determinant circle bundles $\Det_{\SympBase}$ and $\Det^{+}_{\SympBase}$ fitting into a diagram
\begin{equation*}
\begin{tikzcd}
\Circletwopi \arrow[r] \arrow[d] & \Det^{+}_{\SympBase} \arrow[r] \arrow[d] &\SympBase \arrow[d, "\Id"]\\
\Circlepi \arrow[r] & \Det_{\SympBase} \arrow[r]& \SympBase
\end{tikzcd}
\end{equation*}
where each row is a ``fiber, total space, base space'' diagram and the vertical arrows forget orientations of Lagrangian subspaces. As $c_{1}(\SympBase) = 0$, these determinant circle bundles admit trivializations over all of $\SympBase$ and we'll assume that such trivializations are fixed.

The (oriented) Lagrangian Gauss map ($G^{+}$) $G$ is the section of ($\LagGras^{+}_{\SympBase}$) $\LagGras_{\SympBase}$ over $\Leg$ which assigns to each $q \in \Leg$ its (oriented) tangent space. Extend $G$ over each $\chord$ by a CW rotation. This extension is unique up to boundary-relative homotopy. Choose an extension of $G^{+}$ over the connecting paths $\gamma[i, i']$.

Consider a loop in the union of $\Leg$ with the chords and connecting paths,
\begin{equation*}
\gamma: \Circle \rightarrow \Leg \cup \Chords \cup \left( \cup \gamma[i, i'] \right).
\end{equation*}
Then $\det G\gamma$ is a map from $\Circle$ to $\Circlepi$. We define
\begin{equation*}
\Maslov(\gamma) \in \Z
\end{equation*}
to be the degree of this map. If $\im(\gamma) \subset \Leg$, then we can factor through $G^{+}$ as 
\begin{equation*}
\deg \det G \gamma = 2\deg \Det^{+} G^{+}\gamma
\end{equation*}
so $\Maslov(\gamma) \in 2\Z$. The result will depend on $[\gamma] \in H_{1}(\Leg, \Z)$ and vanishes on torsion elements.

\begin{defn}
The map $\Maslov: H_{1}(\Leg, \Z) \rightarrow \Z$ is the Maslov homomorphism and $\rho = \rho(\Leg)$ is defined by $\im(\Maslov) = 2\rho\Z \subset \Z$. For elements $h \in H_{1}(\Leg, \Z)$ we define $\Z$ gradings by
\begin{equation*}
|h| = \Maslov(h).
\end{equation*}
For a word $\word$ of length $\filtLevel$, define $\Z$-valued Maslov numbers and gradings
\begin{equation*}
\Maslov(\word) = \Maslov(\gamma(\word)), \quad |\word| = n - 3 + \filtLevel + \Maslov(\word)
\end{equation*}
where $\gamma(\word)$ is as defined in Equation \eqref{Eq:WordLoop}.
\end{defn}

The following lemma is evident from tracking the orientation changes of $T\Leg$ as $\gamma(\word)$ traverses chords, where $CW$ rotations are applied.

\begin{lemma}
$\Maslov(\word) \bmod_{2} = \Maslov_{2}(\word)$.
\end{lemma}

A few more details on Maslov homomorphisms: Suppose that $\Lag: \LegUp \rightarrow \LegDown$ is a homology cobordism. Since the $i_{\pm}: H_{1}(\LegUpDown) \rightarrow H_{1}(\Lag)$ are isomorphisms, so is the composition
\begin{equation}\label{Eq:HomologyCobIso}
(i_{-})^{-1}i_{+}: H_{1}(\LegUp) \rightarrow H_{1}(\LegDown).
\end{equation}

\begin{lemma}\label{Lemma:MaslovInvariance}
For $h \in H_{1}(\LegUp)$, $\Maslov(h) = \Maslov\left((i_{-})^{-1}i_{+}h\right)$.
\end{lemma}

\begin{proof}
Consider the Lagrangian Grassmannian bundles $\LagGras_{\C \times \SympBase} \rightarrow \C \times \SympBase$ over the symplectization with Gauss map section $\widetilde{G}: \Lag \rightarrow \LagGras_{\C \times \SympBase}$ using which we may define and compute $\Maslov: H_{1}(\Lag) \rightarrow \Z$.

Along $[C, \infty) \times \LegUp$ and $(-\infty, -C]\times \LegDown$, the Gauss map is $\widetilde{G} = \R_{s} \oplus G$. Therefore $\widetilde{G}$ over the collars of the $\LegUpDown$ computes $\Maslov: H_{1}(\LegUpDown, \Z) \rightarrow \Z$ and $\Maslov(h) = \Maslov(i_{\pm}h)$ for $h \in H_{1}(\LegUpDown)$ by the additivity of $\Maslov$ under direct sum \cite{MS:SymplecticIntro}. This is a restatement of what we wanted to prove.
\end{proof}

In the case $\dim \Leg = 1$, this is one of Chantraine's results \cite{Chantraine:Concordance} on preservation of classical invariants ($\tb ,\rot$) of Legendrian knots under Lagrangian concordance.

\section{Boundary words and their moduli spaces}\label{Sec:BoundaryWords}

In this section we further develop the language of planar diagrams and inscriptions introduced in the introduction. We also incorporate $H_{1}(\Leg)$ and $H_{1}(\Lag)$ data into the story described there. Then the moduli spaces relevant to $\newRSFTA$ are described. To simplify matters, readers may want to ignore $H_{1}$ data upon a first reading.

Throughout $\LagGrouped: \LegGroupedUp \rightarrow \LegGroupedDown$ will be an exact Lagrangian cobordism with a partition $\Partition$ of its connected components inducing partitions $\Partition^{\pm}$ of its ends. We assume that connecting paths and preferred capping paths have been chosen for all pairs of $\Partition^{\pm}$ composable chords on the $\LegUpDown$ so that each word $\word$ has a preferred capping path. Recall that we use $\Lag_{\LegUpDown}$ to denote Lagrangian cylinders over the $\LegUpDown$.

As described in the introduction, we have vector spaces $\orbitVS = \orbitVS(\Leg)$ spanned by the admissible $\Partition$ cyclic words associated to a $\LegGrouped$. We say that
\begin{equation*}
\tensorAlg(\orbitVS), \quad \field[H_{1}(\Leg)]\otimes \tensorAlg(\orbitVS)
\end{equation*}
are the $\newRSFTA$ algebras for $\Leg$ with $\field$ and $\field[H_{1}(\Leg)]$ coefficients, respectively. Either will denoted by $\Algebra = \Algebra(\Leg)$. When working with homology cobordisms, we implicitly identify the $\field[H_{1}(\LegUpDown)]$ using Equation \eqref{Eq:HomologyCobIso}.

\subsection{Boundary words and homological decoration}

A \emph{$H_{1}$ decorated boundary word for $\Lag$} is a sequence
\begin{equation}\label{Eq:ExBoundaryWord}
\boundaryWord = (\chord^{a_{1}}_{i_{1}}\eta_{1}\chord^{a_{2}}_{i_{2}}\cdots\chord^{a_{m}}_{i_{m}}\eta_{m})
\end{equation}
where $a_{i}$ is a $\pm$ sign, each $\chord^{\pm}_{i_{j}}$ is a chord of $\LegUpDown$ (with the $\pm$ signs matching), and each $\eta_{j}$ is an endpoint-relative homotopy class of paths in $\Lag$. We require that at least one of the $a_{i}$ is positive. The endpoints $\eta_{j}$ are constrained as follows: If $a_{j} = +$ ($a_{j} = -$) then $\eta_{j}$ begins at $q^{+}(\chord_{i_{j}})$ (respectively, $q^{-}(\chord_{i_{j}})$). If $a_{j+1} = +$ ($a_{j+1} = -$) then $\eta_{j}$ ends at $q^{-}(\chord_{i_{j + 1}})$ (respectively, $q^{+}(\chord_{i_{j+1}})$). A \emph{boundary word for $\Lag$} is obtained by deleting the $\eta_{j}$.

Unlike $\Partition$ cyclic words of chords, we consider two $\boundaryWord$ to be equivalent if they can be obtained from one another by cyclic rotation. A $H_{1}$ decorated boundary word $\boundaryWord$ has a planar diagram $\planarDiagram\left(\boundaryWord\right) \subset \C$, which is as described in the introduction but with additional decorations of the $\eta_{j}$ along the non-dashed boundary arcs.

When $\Lag = \Lag_{\Leg}$ is a cylinder over a Legendrian and $\chord$ is a chord of $\Leg$, we say that $(\chord^{+}q^{+}\chord^{-}q^{-})$ is the \emph{trivial strip over $\chord$}, where the $q^{\pm}$ are constant paths at the ends of $\chord$.

Maslov numbers of boundary words for $\Lag$ are defined using the notation of the proof of Lemma \ref{Lemma:MaslovInvariance} as follows: Let $\boundaryWord$ be as in Equation \eqref{Eq:ExBoundaryWord}. Each $\eta_{i}$ is lifted to $\LagGras_{\C \times \SympBase}$ by apply the Lagrangian Gauss map and the end of each $\eta_{i}$ is joined to the beginning of the next using a CW rotation along a chord. Applying the determinant, we have a map $\Circle \rightarrow \Circlepi$ whose degree is
\begin{equation*}
\Maslov\left(\boundaryWord\right) \in \Z.
\end{equation*}
We then define the \emph{index of a boundary word} as
\begin{equation*}
\ind\left(\boundaryWord\right) = n-3  + m + \Maslov\left(\boundaryWord\right).
\end{equation*}
Observe that when all of the $a_{i}$ are $+$ signs and the $\eta_{i}$ are preferred capping paths, then $\boundaryWord$ corresponds to a generator $\word \in \orbitVS(\Leg^{+})$ for which
\begin{equation*}
\ind\left(\boundaryWord\right) = |\word|.
\end{equation*}

\subsection{Operators $\mu_{\boundaryWordThicc}$}\label{Sec:MuOperators}

Let $\orbitVS^{\pm}$ be the vector spaces of words associated to the $\LegUpDown$ as described in the introduction. We will now formally define the operators sketched there. To do so we'll need a modified notion of inscription which we'll call \emph{full inscription} requiring consideration of multiple boundary words simultaneously. This is necessary to define cobordism maps and will help us to track markers needed to define our differential $\partial$.

\subsubsection{Disconnected boundary words and full inscriptions}

We use
\begin{equation*}
\boundaryWordThicc = \{ \boundaryWord_{i} \}
\end{equation*}
to denote an unordered, finite collection of boundary words. The planar diagram $\planarDiagram(\boundaryWordThicc)$ is the disjoint union of the $\planarDiagram(\boundaryWord_{i})$ and an inscription $\planarDiagram(\boundaryWordThicc) \subset \planarDiagram(\word)$ is a collection of inscriptions $\planarDiagram(\boundaryWord_{i}) \subset \planarDiagram(\word)$ whose images are disjoint in $\planarDiagram(\word)$.

\begin{defn}
A \emph{full inscription} $\planarDiagram(\boundaryWordThicc) \subset \planarDiagram(\word)$ is an inscription for which every $\chord_{i} \in \word$ appears as a positive puncture of some $\boundaryWord_{j}$.
\end{defn}

In the case $\Lag = \Lag_{\Leg}$, a word $\word$, and a single $\boundaryWord$, an inscription $\planarDiagram\left(\boundaryWord\right) \subset \planarDiagram(\word)$ can be extended to a full inscription by adding trivial strips over each chord in $\word$ which is not a positive chord of $\boundaryWord$. Compare Figures \ref{Fig:Differential} and \ref{Fig:DifferentialFullInscription}.

\begin{figure}[h]
\begin{overpic}[scale=.5]{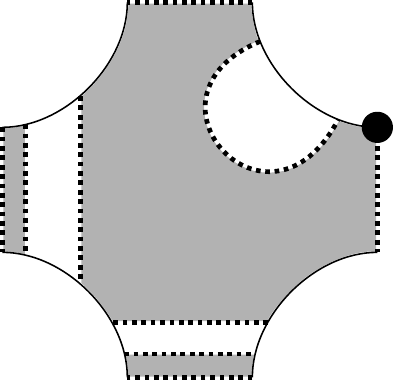}
\end{overpic}
\caption{The inscription of Figure \ref{Fig:Differential} is extended to a full inscription by adding trivial strips.}
\label{Fig:DifferentialFullInscription}
\end{figure}

\begin{lemma}\label{Lemma:AdmissibleComplement}
Let $\word$ be a word of $\LegUp$ chords. Provided a full inscription $\planarDiagram(\boundaryWordThicc) \subset \planarDiagram(\word)$, the complement $\planarDiagram(\word) \setminus \planarDiagram(\boundaryWordThicc)$ is a collection of planar diagrams $\planarDiagram(\word_{i})$ of admissible words $\word_{i}$ for $\LegDown$.
\end{lemma}

\begin{proof}
The full inscription condition ensures that all of the chords in the boundary of $\planarDiagram(\word) \setminus \planarDiagram(\boundaryWordThicc)$ are chords of $\LegDown$. We can view all of the negative chords of $\boundaryWordThicc$ as properly embedded arcs in $\planarDiagram(\word)$.

If some connected component of $\planarDiagram(\word) \setminus \planarDiagram(\boundaryWordThicc)$ is a bigon, then its boundary must consist of a single non-dashed arc corresponding to some piece $\Lag^{\Partition}_{j}$ of $\Lag$ and a single negative chord, $\chord^{-}_{i}$. Because $\chord^{-}_{i}$ connects $\Lag^{\Partition}_{j}$ to itself and $\Partition$ induces the partition $\Partition^{-}$ of $\LegDown$, $\chord^{-}_{i}$ must be $\Partition^{-}$ pure. Therefore our bigon is the planar diagram for $(\chord^{-}_{i}) \in \orbitVS^{-}$.

Now suppose that a connected component of $\planarDiagram(\word) \setminus \planarDiagram(\boundaryWordThicc)$ is a $2m$-gon for $m > 1$. Say that $\chord^{-}_{i}, i=1, \dots, m$ are the corresponding negative chords, ordered counter clockwise. At the positive end of each $\chord^{-}_{i}$, traverse the boundary of $\planarDiagram(\word)$ clockwise until you hit the endpoint of some $\chord^{+}_{i} \in \word$. The $\chord^{+}_{i}$ must be distinct and lie on distinct pieces $\Lag^{\Partition}_{j}$ of $\Lag$. Since the piece of $\Lag$ touched by the positive endpoint of $\chord^{+}_{i}$ is the same as the piece of $L$ that is touched by the positive endpoint of $\chord^{-}_{i}$ and the endpoints of the $\chord^{+}_{i}$ lie on distinct pieces of $\Lag$ (by admissibility), it follows that the $\chord^{-}_{i}$ all end on distinct pieces of $\Lag$. Since $\Partition$ induces the partition $\Partition^{-}$ of $\Leg^{-}$, $\word^{-} = (\chord^{-}_{1}\dots\chord^{-}_{m})$ is an admissible generator of $\orbitVS^{-}$. The connected component of $\planarDiagram(\word) \setminus \planarDiagram(\boundaryWordThicc)$ which we are studying is a planar diagram for this $\word^{-}$.
\end{proof}

\begin{lemma}\label{Lemma:FiltrationPreservation}
Let $\planarDiagram(\boundaryWordThicc) \subset \planarDiagram(\word)$ and let $\word_{i}$ be as in the previous lemma so that $\planarDiagram(\word_{i})$ corresponds to some connected component of $\planarDiagram(\word) \setminus \planarDiagram(\boundaryWordThicc)$. Then $\filtLevel(\word_{i}) \leq \filtLevel(\word)$.
\end{lemma}

\begin{proof}
Because we can insert each of the $\planarDiagram\left(\boundaryWord\right)$ into $\planarDiagram(\word)$ one at a time for $\boundaryWord \in \boundaryWordThicc$, it suffices to work out the case when there is a single $\boundaryWord$. In this case, there is one connected region $\region$ in $\planarDiagram(\word) \setminus \planarDiagram\left(\boundaryWord\right)$ for each negative chord of $\boundaryWord$. There is one negative chord of $\boundaryWord$ and some number of positive chords of $\word$ in $\partial \region$. The result then follows from the fact that $\boundaryWord$ has at least one positive puncture.
\end{proof}

\subsubsection{Markers and orderings of output words}

Now we define an operator
\begin{equation*}
\mu_{\boundaryWordThicc}: \orbitVS^{+} \rightarrow \AlgebraDown
\end{equation*}
for each finite collection $\boundaryWordThicc = \{ \boundaryWord_{i} \}$ of boundary words. Here $\Algebra^{\pm} = \tensorAlg(\orbitVS^{\pm})$ are the algebras associated to the positive and negative ends of a cobordism.

For each $\word$ and full inscription $\planarDiagram(\boundaryWordThicc) \subset \planarDiagram(\word)$, the proof of Lemma \ref{Lemma:AdmissibleComplement} provides us words $\word^{-}_{i}$ as possible outputs for $\mu_{\boundaryWordThicc}(\word)$. However the $\word^{-}_{i}$ are only (so far) defined up to cyclic rotation. To provide a formal description of $\mu_{\boundaryWordThicc}(\word)$, we need to address this ambiguity, as well as how the $\word_{i}^{-}$ should be ordered in a tensor product.

Suppose that we have a full inscription with input $\word = (\chord_{i_{1}}\cdots \chord_{i_{\filtLevel}})$. Index the negative chords $\chord^{-}_{i_{j}} \in \Chords(\LegDown)$ of $\boundaryWordThicc$ (ordering the $j$s) so that when we traverse the boundary of $\planarDiagram(\word)$ counterclockwise starting at the endpoint of $\chord_{i_{1}}$, we encounter the endpoints of the $\chord^{-}_{i_{j}}$ in order.

\begin{defn}\label{Def:FullInscriptionOperator}
For each connected component of $\planarDiagram(\word) \setminus \planarDiagram(\boundaryWordThicc)$ we get a word $\word_{j}^{-} = (\chord^{-}_{i_{1}}\cdots \chord^{-}_{i_{m}})$ by requiring that $i_{1}$ is the least index with respect to the aforementioned ordering. The indices $j$ of the $\word^{-}_{j}$,
\begin{equation*}
\word^{-}_{1} = \left(\chord^{-}_{i_{1,1}}\cdots\right),\dots,\word^{-}_{m} = \left(\chord^{-}_{i_{m,1}}\cdots\right).
\end{equation*}
are determined by the $\chord^{-}_{i_{j, 1}}$ as follows: When we traverse the boundary of $\planarDiagram(\word)$ counterclockwise starting at the endpoint of $\chord_{i_{i}}$, we first encounter the endpoint of $\chord^{-}_{i_{1, 1}}$, then the positive endpoint of $\chord^{-}_{i_{2, 1}}$, and so on.

For each $\boundaryWordThicc$ define
\begin{equation*}
\mu_{\boundaryWordThicc}: \orbitVS^{+} \rightarrow \AlgebraDown, \quad \mu_{\boundaryWordThicc}(\word) = \begin{cases}
\word^{-}_{1}\cdots \word^{-}_{m} & \planarDiagram(\boundaryWordThicc) \subset \planarDiagram(\word) \\
0 & \text{otherwise}.
\end{cases}
\end{equation*}
\end{defn}

The definitions above are easier to see in pictures than in formulas. The reader may want to refer back to Figure \ref{Fig:Differential}, add in trivial strips as shown in Figure \ref{Fig:DifferentialFullInscription}, and ensure that the input and output words for the $\mu_{\boundaryWordThicc}$ are as described in the caption of Figure \ref{Fig:Differential}.

\subsubsection{$H_{1}$ twisted coefficients}\label{Sec:TwistedCoeffs}

When $\Lag: \LegUp \rightarrow \LegDown$ is a homology cobordism (eg. a cylinder $\Lag_{\Leg}$) we associate an element of $H_{1} = \field[H_{1}(\Lag)] \simeq \field[H_{1}(\LegUpDown)]$ to each $H_{1}$ decorated $\boundaryWord =  (\chord^{a_{1}}_{i_{1}}\eta_{1}\chord^{a_{2}}_{i_{2}}\cdots\chord^{a_{m}}_{i_{m}}\eta_{m})$. Each $\eta_{j} \in \boundaryWord$ is an oriented path between starting or ending points of chords. Following preferred basepaths (using forward or backwards orientations depending on the $a_{j}, a_{j+1}$), start at a basepoint, travel to the start of $\eta_{j}$, traverse $\eta_{j}$, then travel back to a basepoint. If necessary, follow a concordance path to obtain a closed loop in $\Lag$. This yields an element $h\left(\boundaryWord\right) \in H_{1}(\Lag)$. Applying exponential notation, write
\begin{equation*}
e^{h(\boundaryWordThicc)} = e^{\sum_{\boundaryWord \in \boundaryWordThicc} h\left(\boundaryWord\right)}= \prod_{\boundaryWord \in \boundaryWordThicc} e^{h\left(\boundaryWord\right)} \in \field[H_{1}(\Leg)].
\end{equation*}

When using $\field[H_{1}(\Leg)]$ coefficients we upgrade the definition of $\mu_{\boundaryWordThicc}$ to
\begin{equation}\label{Eq:H1Operator}
\mu_{\boundaryWordThicc}: \orbitVS^{+} \rightarrow \AlgebraDown, \quad \mu_{\boundaryWordThicc}(\word) = \begin{cases}
e^{h(\boundaryWordThicc)}\word^{-}_{1}\cdots \word^{-}_{m} & \planarDiagram(\boundaryWordThicc) \subset \planarDiagram(\word) \\
0 & \text{otherwise}.
\end{cases}
\end{equation}

\subsection{Moduli spaces}

In this article we will mostly need two types of moduli spaces:
\be
\item Those associated to Lagrangian projections $\pi_{\SympBase}\Leg \subset \SympBase$.
\item Those associated to general Lagrangian cobordisms $\Lag \subset \R_{s} \times M$.
\ee
Moduli spaces associated to $1$-parameter families of Legendrians and Lagrangians will also be considered in the invariance proofs of Appendix \ref{App:Invariance}.

Let $\orderedBoundaryPunc = (\boundaryPunc_{1},\dots, \boundaryPunc_{m})$ be a tuple for which the $\boundaryPunc_{k} \in \partial \disk$ are pairwise distinct, ordered counterclockwise, and have $\orderedBoundaryPunc_{1} = 1$. The moduli space of such $\orderedBoundaryPunc$ has dimension $m - 3$ for $m \geq 3$ and $0$ otherwise. A fixed $\orderedBoundaryPunc$ has automorphism group of dimension $3 - m$ for $m \leq 3$ and otherwise consists only of the identity map. For such $\orderedBoundaryPunc$, write $\holoDom = \disk \setminus \orderedBoundaryPunc$. Let for $k = 1, \dots, m$, let $I_{k} \subset \partial \holoDom$ be the open interval whose closure $\overline{I}_{k}$ has oriented boundary $\partial I_{k} = \boundaryPunc_{k+1} - \boundaryPunc_{k}$.

Let $\boundaryWord$ be a boundary word as described in Equation \eqref{Eq:ExBoundaryWord}. From the data of $\boundaryWord$, we declare that $\boundaryPunc_{k} \in \orderedBoundaryPunc$ is \emph{positive} (negative) if $a_{k} = +$ (respectively, $a_{k} = -$). Likewise, we declare that a $\boundaryPunc_{k} \in \orderedBoundaryPunc$ is pure, mixed, $\Partition$ pure, or $\Partition$ mixed if the corresponding chord $\chord_{i_{k}}\in \boundaryWord$ has the corresponding property. We label each $I_{k} \subset \partial\holoDom$ with a $j_{k}$ index of some $\Lag^{\Partition}_{j_{k}}$ as follows. If $a_{k} = +$ and $\chord_{i_{k}}$ ends on $\Lag^{\Partition}_{j}$, then we set $j_{k} = j$. If $a_{k} = -$ and $\chord_{i_{k}}$ begins on $\Lag^{\Partition}_{j}$, then we set $j_{k} = j$. We say that the $j_{k}$ associated to the $I_{k}$ are \emph{Lagrangian boundary labels}. The paths $\eta_{k}$ associated to each $I_{k}$ is the \emph{relative path label}. Note that these $\eta_{k}$ can not jump across connected components of $\Lambda$ and will not necessarily coincide with preferred capping paths. When $\Lag = \Lag_{\Leg}$ we can view the $\eta_{k}$ as paths in $\Leg$ via the projection $\R_{s} \times \Leg \rightarrow \Leg$.

\begin{defn}
Given a boundary word $\boundaryWord$ for $\Leg$, define $\ModSpace_{\Leg}\left(\boundaryWord\right)$ to be the moduli space of pairs $(\orderedBoundaryPunc, u)$ with $u: \holoDom \rightarrow \SympBase$ satisfying
\be
\item $u$ is holomorphic with respect to $J_{\SympBase}$,
\item $\lim_{z \rightarrow z_{k}}u(z) = \pi_{\SympBase}\chord_{i_{k}}$,
\item $u(I_{k}) \subset \pi_{\SympBase}\Leg$ for all $k$,
\item $u|_{I_{k}}$ is boundary-relative homotopic to $\pi_{\SympBase}\eta_{k}$ for each $k$, and
\item $u$ admits a continuous lift $\widetilde{u}$ to $M$ so that $\pi_{\SympBase}\widetilde{u} = u$
\ee
modulo the relation $(\orderedBoundaryPunc, u) \sim (\orderedBoundaryPunc', u')$ if there is some $\phi \in \Aut(\disk)$ for which $\orderedBoundaryPunc' = \phi\orderedBoundaryPunc, u' = u\phi$.
\end{defn}

The continuous lift condition ensures that $u$ cannot jump across double points along the interiors of the $I_{k}$. We define the \emph{energy} of such a $u$ as $\action(u) = \int_{\holoDom}u^{\ast}d\beta$. It follows from Stokes' theorem that
\begin{equation*}
\action(u) = \sum_{k} a_{k}\action(\chord_{i_{k}}) < \infty.
\end{equation*}

Now we modify the above notation to deal with Lagrangian cobordisms $\Lag: \Leg^{+} \rightarrow \Leg^{-}$. Let $\ChordsUpDown$ be the chords of the $\LegUpDown$.

\begin{defn}
Given a boundary word $\boundaryWord$ for $\Lag$, define $\ModSpace_{\Lag}\left(\boundaryWord\right)$ to be the moduli space of pairs $(\orderedBoundaryPunc, u)$ with $u: \holoDom \rightarrow \R_{s} \times M$ satisfying
\be
\item $u$ is holomorphic with respect to $J$,
\item each $z_{k}$ is positively (negatively) asymptotic to $\chord_{i_{k}}$ if $a_{k} = +$ (respectively, $a_{k} = -$), and
\item $u(I_{k}) \subset \Lag$ for all $k$.
\ee
modulo the relation $(\orderedBoundaryPunc, u) \sim (\orderedBoundaryPunc', u')$ if there is some $\phi \in \Aut(\disk)$ for which $\orderedBoundaryPunc' = \phi\orderedBoundaryPunc, u' = u\phi$.
\end{defn}

For a $\boundaryWordThicc = \{ \boundaryWord_{i}\}$, we set $\ModSpace_{\Lag}(\boundaryWordThicc) = \prod \ModSpace_{\Lag}\left(\boundaryWord\right)$. For more details on convergence to chords near boundary punctures, see \cite{DR:Lifting, Ekholm:Z2RSFT, RS:Strips, EES:LegendriansInR2nPlus1}. For the following, see \cite[\S 3.1]{Ekholm:Z2RSFT}:

\begin{lemma}
The expected dimension of $\ModSpace_{\Lag}(\boundaryWordThicc)$ is $\sum \ind(\boundaryWord_{i})$. If $\word \in \orbitVS^{+}$ bounds a disk $u_{\word}$, then
\begin{equation*}
|\word| = \ind(u_{\word}).
\end{equation*}
\end{lemma}

Our choice of gradings of words is designed so that the degree of the operator $\mu_{\boundaryWordThicc}$ computes the Fredholm index. The following is an immediate consequence of additivity of the Fredholm index under gluing along pairs of positive and negative punctures:

\begin{lemma}\label{Lemma:MuDegree}
The operator $\mu_{\boundaryWordThicc}$ of Equation \eqref{Eq:H1Operator} satisfies $\deg \mu_{\boundaryWordThicc} = \sum \ind(\boundaryWord_{i})$.
\end{lemma}

The moduli space $\ModSpace_{\Lag_{\Leg}}\left(\boundaryWord\right)$ admits an $\R_{s}$ action by shifting in the symplectization coordinate. The action is free away from trivial strips whose corresponding $\ModSpace_{\Leg}\left(\boundaryWord\right)$ elements are constant maps to $\pi_{\SympBase}\chord_{i}$.

\begin{lemma}\label{Lemma:Lifting}
The map $u \mapsto \pi_{\SympBase}u$ induces a homeomorphism
\begin{equation*}
\ModSpace_{\Lag_{\Leg}}\left(\boundaryWord\right)/\R_{s} \rightarrow \ModSpace_{\Leg}\left(\boundaryWord\right).
\end{equation*}
\end{lemma}

The proof is identical to \cite[Theorem 2.1]{DR:Lifting}. See also \cite[Theorem 7.7]{ENS:Orientations} for the case $\SympBase = \C$. We have the immediate corollary:

\begin{lemma}\label{Lemma:SympBaseIndex}
The expected dimension of $\ModSpace_{\Leg}(\boundaryWordThicc)$ is the sum of $\ind(\boundaryWord_{i}) - 1$ over all $\boundaryWord_{i} \in \boundaryWordThicc$ which are not trivial strips.
\end{lemma}

\begin{notation}
For $u \in \ModSpace_{\Leg}(\boundaryWordThicc)$ we use $\ind(u)$ for the index of the corresponding element of $\ModSpace_{\Lag_{\Leg}}(\boundaryWordThicc)$.
\end{notation}

\subsection{Admissibility of disks}

A \emph{splitting arc} in $\holoDom$ is a compact, embedded, oriented arc $\splittingArc \subset \holoDom$ with boundary on $\partial \holoDom$ with the interior of $\splittingArc$ contained in $\B$. Splitting arcs are oriented as follows: $\holoDom \setminus \splittingArc$ has two connected components $D_{0}$ and $D_{1}$ with $1 \in \partial D_{1}$. We declare that the orientation of $\splittingArc$ agrees with the boundary orientation of the closure $\overline{D_{0}}$ of $D_{0}$.

The following definition and lemma are from \cite{Ekholm:Z2RSFT}.

\begin{defn}\label{Def:AdmissibleDisk}
We say that $\boundaryWord$ is an \emph{admissible boundary word} if for every splitting arc $\splittingArc$ for which the Lagrangian boundary labels $\Lag^{\Partition}_{j}$ at the endpoints of $\splittingArc$ are the same, one component of the complement of $\splittingArc$ has either
\be
\item no boundary punctures or
\item only negative, $\Partition$ pure chords.
\ee
A holomorphic disk $u \in \ModSpace_{\Leg}\left(\boundaryWord\right)$ or $u \in \ModSpace_{\Lag}\left(\boundaryWord\right)$ is \emph{admissible} if $\boundaryWord$ is an admissible boundary word.
\end{defn}

\begin{lemma}\label{Lemma:NoPurePosForAdmissible}
Suppose that $\boundaryWord$ is an admissible boundary word. If $\boundaryWord$ has a $\Partition$ pure positive boundary puncture, then all of its other boundary punctures must be $\Partition$ pure negative punctures. A holomorphic map $u \in \ModSpace_{\Lag}\left(\boundaryWord\right)$ is somewhere injective.
\end{lemma}

\begin{lemma}\label{Lemma:AdmissabilityEquiv}
Let $\boundaryWord$ be a boundary word with all $a_{k}$ positive. Then the word of chords $\boundaryWord$ is admissible if and only if it is an admissible boundary word.
\end{lemma}

\begin{proof}
Suppose $\vec{a} \subset \planarDiagram(\word)$ is a splitting arc with both endpoints on the same $\Lag^{\Partition}_{j}$. If $\boundaryWord$ is an admissible $\Partition$ cyclic word, then $\vec{a}$ must be boundary parallel. Therefore $\boundaryWord$ is an admissible boundary word.

If $\boundaryWord$ -- as a word of chords -- is not admissible, then two distinct chords in $\boundaryWord$ end on the same $\Lag^{\Partition}_{j}$. If one of the two chords is $\Partition$ pure, make a $\vec{a} \subset \planarDiagram\left(\boundaryWord\right)$ which is parallel to this chord. Then both components of $\planarDiagram\left(\boundaryWord\right)$ will have positive chords. If both chords are $\Partition$ mixed, we get two non-dashed boundary arcs for $\planarDiagram\left(\boundaryWord\right)$ with the same Lagrangian label and we can choose an arc $\vec{a}$ connecting them to split $\planarDiagram\left(\boundaryWord\right)$ into two components, each containing a positive chord. Thus $\boundaryWord$ is not admissible as a boundary word.
\end{proof}

\begin{lemma}\label{Lemma:InscriptionAdmiss}
Let $\word$ be admissible $\Partition$ cyclic word for $\Leg^{+}$ and let $\boundaryWordThicc$ be a collection of boundary words for $\Lag$. If there exists an inscription $\planarDiagram(\boundaryWordThicc) \subset \planarDiagram(\word)$, then each $\boundaryWord_{i} \in \boundaryWordThicc$ is admissible in the sense of Definition \ref{Def:AdmissibleDisk}.
\end{lemma}

\begin{proof}
Let $\vec{a}$ be a splitting arc in some $\boundaryWord_{i}$ with both of its endpoints associated to the same piece $\Lag^{\Partition}_{j}$ of $\Lag^{\Partition}$. Using the inscription we can view $a$ as being contained in $\planarDiagram(\word)$. By the fact that $\word$ is admissible, there is a bigon $b \subset \planarDiagram(\word)$ whose (unoriented) boundary is the union of $a$ with an in arc in $\Lag_{j}^{\Partition}$ connecting the endpoints of $a$. Then $b \cap \planarDiagram(\boundaryWord_{i})$ gives us one of the connected components of $\word \setminus a$. This subset can only contain $\Partition$ pure negative chords both of whose endpoints live in $\Lag^{\Partition}_{j}$. Therefore $\boundaryWord$ is admissible.
\end{proof}

\subsection{Bubble trees and compactness}

We use language similar to that of \cite{BH:ContactDefinition, Pardon:Contact} to describe gluing operations on boundary words.

\begin{defn}
A \emph{bubble tree of boundary words for $\Lag$} is a connected, directed acyclic graph
\begin{equation*}
\tree = (\{\vertex_{i}\}, \{\edge_{j}\})
\end{equation*}
with the following labeling data and conditions:
\be
\item An assignment $\Lag(\vertex_{i}) \in \{\Lag, \Lag_{\LegUp}, \Lag_{\LegDown}\}$ to each vertex $\vertex_{i}$.
\item An assignment of a boundary word $\boundaryWord(\vertex_{i})$ for $\Lag(\vertex_{i})$ to each vertex.
\item A chord $\chord(\edge_{j})$ is assigned to each edge $\edge_{j}: \vertex_{j^{-}} \rightarrow \vertex_{j^{+}}$ which is a positive chord of $\boundaryWord(\vertex_{j^{-}})$ and a negative chord of $\boundaryWord(\vertex_{j^{+}})$.
\item For all $\boundaryWord(\vertex_{i})$, each $\chord_{j} \in \boundaryWord(\vertex_{i})$ is associated to at most one edge.
\item If $\chord_{i_{k}} \in \boundaryWord(\vertex_{i})$ is not touched by any edge, then it is either a positive chord of $\Leg^{+}$ or a negative chord of $\Leg^{-}$.
\ee
When $\Leg = \Leg^{-} = \Leg^{+}$, we simply say that $\tree$ is a bubble tree for $\Leg$.
\end{defn}

To simplify notation, when $\Lag$ is understood we'll simply call a bubble tree of boundary words for $\Lag$ a \emph{bubble tree}. We define and draw a \emph{planar diagram of a bubble trees of boundary words $\tree$} by adjoining the planar diagrams of the $\boundaryWord(\vertex_{i})$ in $\C$ according to the matchings determined by the $\edge_{j}$. See Figure \ref{Fig:BubbleTree}. We denote the planar diagram $\planarDiagram(\tree)$. Clearly all of the data of $\tree$ can be recovered from the picture $\planarDiagram(\tree)$.

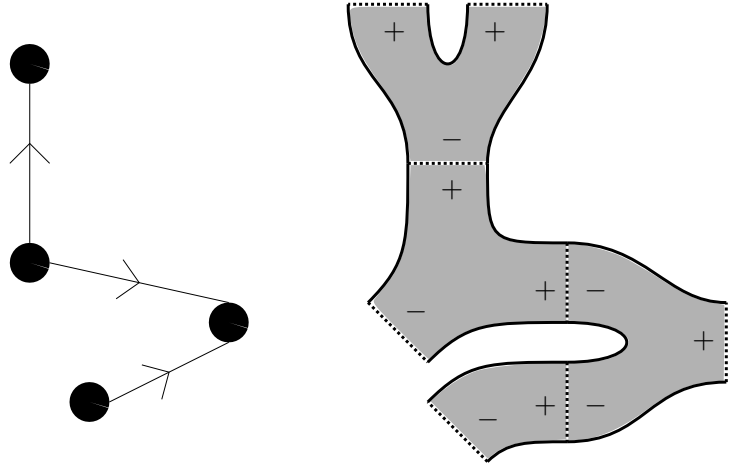
\begin{figure}[h]
\begin{overpic}[scale=.5]{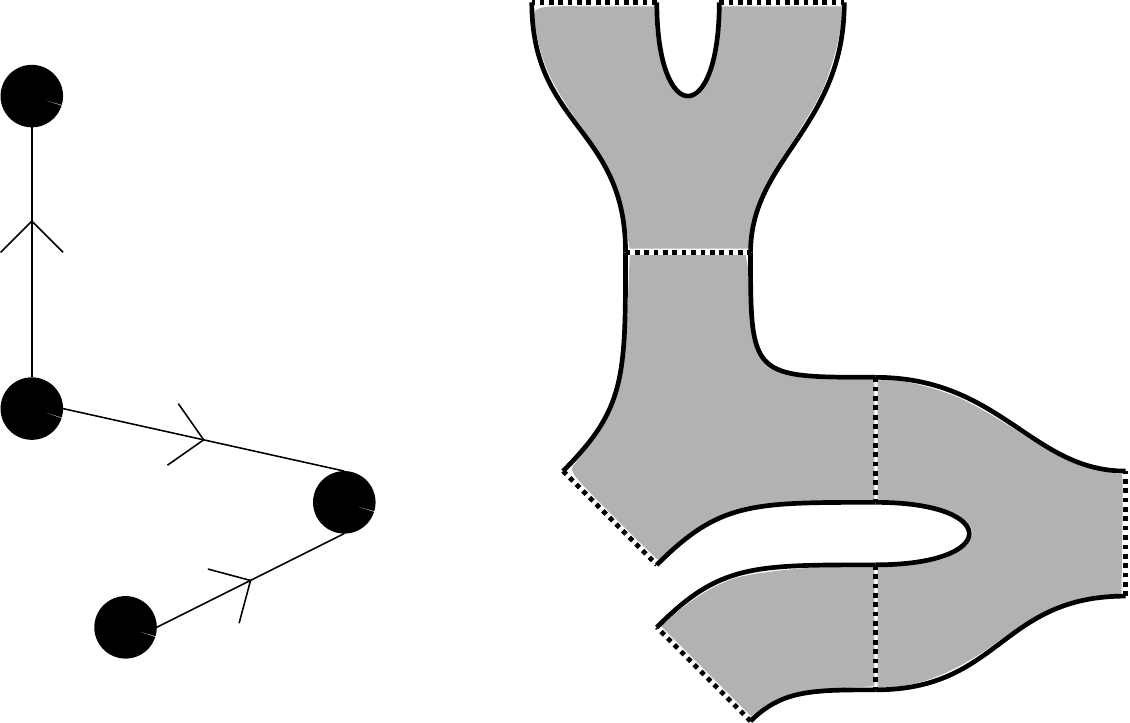}
\put(52, 59){$+$}
\put(66, 59){$+$}
\put(60, 44){$-$}
\put(60, 37){$+$}
\put(73, 23){$+$}
\put(80, 23){$-$}
\put(73, 7){$+$}
\put(80, 7){$-$}
\put(95, 16){$+$}
\put(55, 20){$-$}
\put(65, 5){$-$}
\end{overpic}
\caption{A bubble tree of boundary words represented as a graph (left) and a planar diagram (right). For simplicity, only signs of chords are indicated.}
\label{Fig:BubbleTree}
\end{figure}

For an $\edge_{j}$ of a $\tree$, define a new bubble tree $\#_{\edge_{j}}\tree$ whose planar diagram is obtained by deleting $\chord(\edge_{j})$ from $\planarDiagram(\tree)$ and concatenating capping paths according to the boundary orientation of non-dashed arcs of the planar diagram. The operation $\#_{\edge_{j}}$ is associative. By applying $\#_{\edge_{j}}$ to each $\edge_{j}$ we obtain an $H_{1}$ decorated boundary word, we'll call $\#(\tree)$.

\begin{defn}
A \emph{$\ModSpace_{\Lag}$ bubble tree} for $\Lag: \Leg^{+} \rightarrow \Leg^{-}$ is a bubble tree $\tree$ of boundary words for $\Lag$ with an assignment of an unparameterized holomorphic maps
\begin{equation*}
u(\vertex_{i}) \in \ModSpace_{\Lag(\vertex_{i})}(\boundaryWord(\vertex_{i}))
\end{equation*}
to each vertex. A \emph{$\ModSpace_{\Leg}$ bubble tree} is a bubble tree $\tree$ for $\Leg$ with an assignment of a holomorphic map
\begin{equation*}
u(\vertex_{i}) \in \ModSpace_{\Leg}(\boundaryWord(\vertex_{i}))
\end{equation*}
to each vertex.
\end{defn}

\begin{thm}
Let $u_{n}, n\in \Z_{\geq 0}$ be a sequence of holomorphic disks in $\ModSpace_{\Lag}\left(\boundaryWord\right)$ (respectively, $\ModSpace_{\Leg}\left(\boundaryWord\right)$). Then there is a subsequence $u_{k}$ which converges in the sense of \cite{SFTCompactness} to a $\ModSpace_{\Lag}\left(\boundaryWord\right)$ bubble tree (respectively, a $\ModSpace_{\Leg}\left(\boundaryWord\right)$ bubble tree).
\end{thm}

Since we are using the same holomorphic disks as \cite{Ekholm:Z2RSFT}, the case $\ModSpace_{\Lag}\left(\boundaryWord\right)$ follows immediately from \cite[Appendix B]{Ekholm:Z2RSFT}. The case for $\ModSpace_{\Leg}\left(\boundaryWord\right)$ follows from the results of the next section, combined with the corresponding results for $\CE$ moduli spaces in $\R_{t} \times \SympBase$ (using the finite geometry at infinity hypothesis) described in \cite{EES:PtimesR}, with details appearing in \cite{EES:LegendriansInR2nPlus1}. Observe that in either case the fact that all elements of our sequences are contained in a single moduli space entails that the energy bounds required for compactness are automatically satisfies. See \cite[Appendix B]{Ekholm:Z2RSFT} for details.

\subsection{$\CE$ type boundary words}

Say that $\boundaryWord$ is of \emph{$\CE$ type} if it has exactly one positive chord.  The following lemma states that each moduli space of admissible disks is homeomorphic to a moduli space of $\CE$ type curves. The proof applies shifts to the connected components of $\Leg$ in the $t$-coordinate of $M = \R_{t} \times W$, as indicated in Figure \ref{Fig:LCHShift}. This shifting preserves the Lagrangian projection. The particular type of $\CE$ moduli space produced by the lemma is not relevant: The lemma will be used to manage analytical aspects of moduli spaces of admissible holomorphic disks, since $\CE$ moduli spaces well understood by the results of \cite{EES:LegendriansInR2nPlus1, EES:PtimesR}.

\begin{figure}[h]
\begin{overpic}[scale=.4]{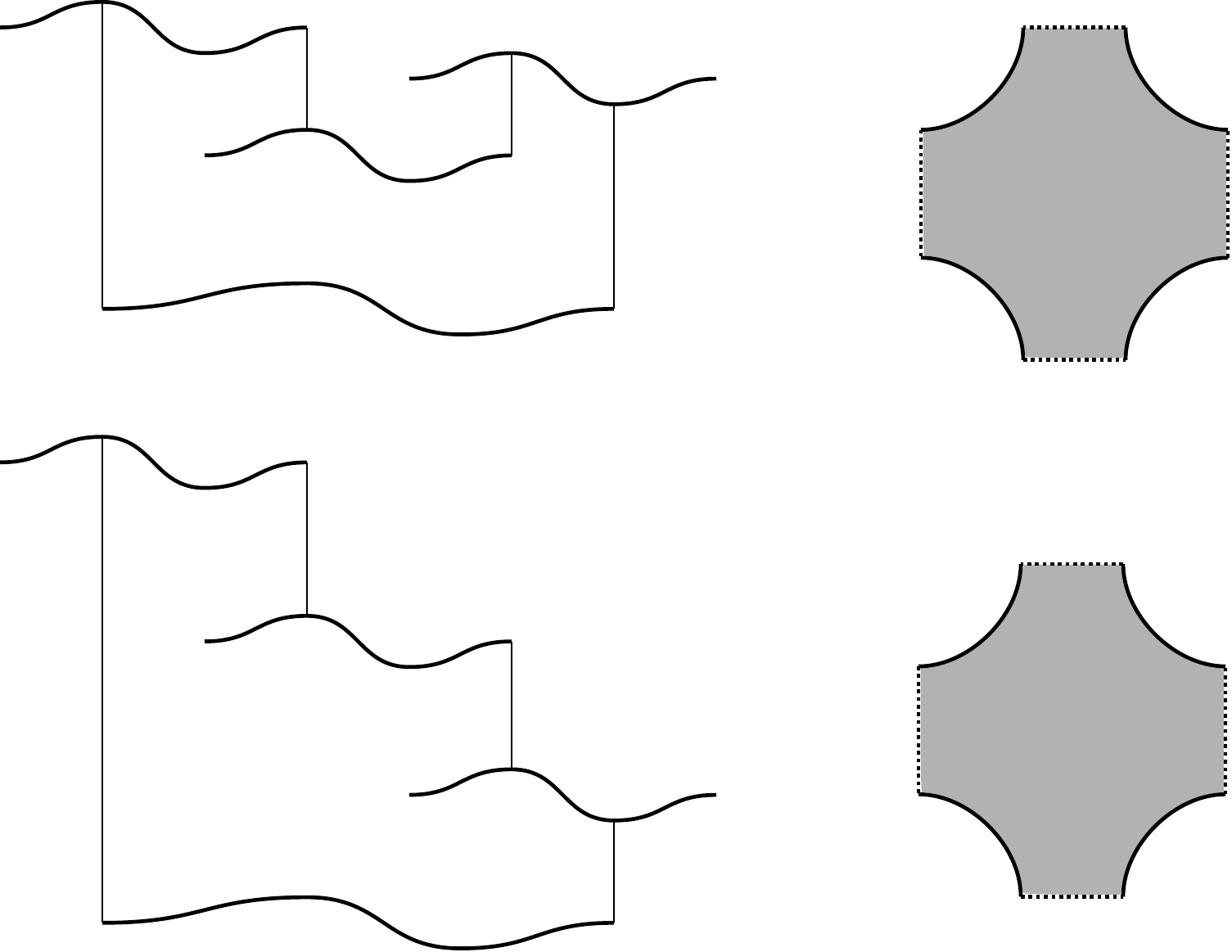}
\put(-5, 75){$\Leg^{\Partition}_{1}$}
\put(11, 63){$\Leg^{\Partition}_{2}$}
\put(28, 70){$\Leg^{\Partition}_{3}$}
\put(3, 51){$\Leg^{\Partition}_{4}$}

\put(85, 43){$\chord_{i_{4}}^{-}$}
\put(101, 60){$\chord_{i_{1}}^{+}$}
\put(85, 77){$\chord_{i_{2}}^{-}$}
\put(68, 60){$\chord_{i_{3}}^{+}$}

\put(19, 70){$\chord_{i_{2}}$}
\put(2, 63){$\chord_{i_{1}}$}
\put(36, 67){$\chord_{i_{3}}$}
\put(44, 60){$\chord_{i_{4}}$}

\put(-5, 38){$\widetilde{\Leg}^{\Partition}_{1}$}
\put(11, 23){$\widetilde{\Leg}^{\Partition}_{2}$}
\put(28, 12){$\widetilde{\Leg}^{\Partition}_{3}$}
\put(3, 1){$\widetilde{\Leg}^{\Partition}_{4}$}

\put(85, 0){$\widetilde{\chord}_{i_{4}}^{-}$}
\put(101, 17){$\widetilde{\chord}_{i_{1}}^{+}$}
\put(85, 34){$\widetilde{\chord}_{i_{2}}^{-}$}
\put(68, 17){$\widetilde{\chord}_{i_{3}}^{-}$}

\put(19, 30){$\widetilde{\chord}_{i_{2}}$}
\put(2, 12){$\widetilde{\chord}_{i_{1}}$}
\put(36, 18){$\widetilde{\chord}_{i_{3}}$}
\put(44, 6){$\widetilde{\chord}_{i_{4}}$}

\put(81, 60){$\planarDiagram(u)$}
\put(81, 16){$\planarDiagram(\widetilde{u})$}
\end{overpic}
\caption{A schematic for the proof of Lemma \ref{Lemma:LCHShift} for $m = 4$ and all $\chord_{i_{k}}$ $\Partition$ mixed. Relative heights of the $\Leg^{\Partition}_{j}$ and $\widetilde{\Leg}^{\Partition}_{j}$ in neighborhoods of chords are indicated on the left. Planar diagrams for holomorphic maps are shown on the right.}
\label{Fig:LCHShift}
\end{figure}

\begin{lemma}\label{Lemma:LCHShift}
For every admissible boundary word $\boundaryWord = (\chord_{i_{1}}^{+}\chord_{i_{2}}^{a_{2}}\cdots\chord_{i_{m}}^{a_{m}})$ there is a Legendrian $\widetilde{\Leg} \subset M$ and an admissible boundary word $\widetilde{\boundaryWord} = (\widetilde{\chord}_{i_{1}}^{+}\widetilde{\chord}_{i_{2}}^{-}\cdots\widetilde{\chord}_{i_{m}}^{-})$ for $\widetilde{\Leg}$ of $\CE$ type such that
\begin{equation*}
\pi_{\SympBase}\widetilde{\Leg} = \pi_{\SympBase}\Leg, \quad \pi_{\SympBase}\widetilde{\chord}_{i_{k}} = \pi_{\SympBase}\chord_{i_{k}}
\end{equation*}
for all $k=1, \dots, m$ and we have homeomorphisms of moduli spaces
\begin{equation*}
\ModSpace_{\Leg}\left(\boundaryWord\right) \simeq \ModSpace_{\widetilde{\Leg}}\left(\widetilde{\boundaryWord}\right),\quad \ModSpace_{\Lag_{\Leg}}\left(\boundaryWord\right) \simeq \ModSpace_{\Lag_{\widetilde{\Leg}}}\left(\widetilde{\boundaryWord}\right).
\end{equation*}
\end{lemma}

\begin{proof}
If all of the $\chord_{i_{k}}$ are $\Partition$ pure there is nothing to prove, since in this case the admissibility criterion would imply that $\boundaryWord$ is of $\CE$ type. Likewise, admissibility implies that if there are any $\Partition$ mixed chords, then at least one of them must be positive. So we assume that $\chord_{i_{1}}$ is $\Partition$ mixed. Traverse the boundary of $\planarDiagram\left(\boundaryWord\right)$ starting at the endpoint of $\chord_{i_{1}}$ while tracking Lagrangian boundary labels. This yields a length $l$ tuple $(\Leg^{\Partition}_{j_{1}}, \cdots, \Leg^{\Partition}_{j_{l}})$ of pairwise distinct Legendrians. For simplicity, let's assume that their union is all of $\Leg$ so that $l = N^{\Partition}$ and that $j_{k} = k$ for $k=1, \dots, N^{\Partition}$.

The remainder of the proof is summarized in Figure \ref{Fig:LCHShift}. For $C \gg 0$ consider
\begin{equation*}
\widetilde{\Leg}^{\Partition}_{j} = \Flow^{(1- j)C}_{\partial_{t}}\Leg^{\Partition}_{j} \quad \widetilde{\Leg} = \sqcup \widetilde{\Leg}^{\Partition}_{j}.
\end{equation*}
Choose $C$ large enough so that $\min t|_{\widetilde{\Leg}^{\Partition}_{j}} > \max t_{\widetilde{\Leg}^{\Partition}_{j+1} }$ for $j=1, \dots, N^{\Partition} -1$. Then $\widetilde{\Leg}$ is  embedded and there will be no chords from $\widetilde{\Leg}^{\Partition}_{i}$ to $\widetilde{\Leg}^{\Partition}_{j}$ if $i < j$ where the $i, j$ are \emph{not} taken modulo $N^{\Partition}$.

Clearly $\pi_{\SympBase}\widetilde{\Leg} = \pi_{\SympBase}\Leg$ and as chords are in bijective correspondence with double points in $\SympBase$, the chords of $\Leg$ are in bijective correspondence with those of $\widetilde{\Leg}$. For each $\chord_{i}$ of $\Leg$ write $\widetilde{\chord}_{i}$ for the corresponding chord of $\widetilde{\Leg}$.

For a holomorphic $u \in \ModSpace_{\Leg}\left(\boundaryWord\right)$, we get a holomorphic $\widetilde{u}$ associated to $\widetilde{\Leg}$ with associated boundary word $\widetilde{\boundaryWord}$ as described in the statement of the lemma. The planar diagram for $\widetilde{u}$ is the same as for $u$ except that all chords in $\planarDiagram(\widetilde{u})$ except for $\chord_{i_{1}}$ are negative. The identification of $\ModSpace_{\Leg}$ moduli spaces is evident from $\pi_{\SympBase}\widetilde{\Leg} = \pi_{\SympBase}\Leg$. The corresponding statement relating the $\ModSpace_{\Lag_{\Leg}}$ and $\ModSpace_{\Lag_{\widetilde{\Leg}}}$ then follows from Lemma \ref{Lemma:Lifting}.
\end{proof}

\subsection{Transversality}

Now we address transversality for generic adapted almost complex structures:

\begin{lemma}\label{Lemma:SympBaseTransversality}
For a fixed $\Leg$ and generic choice of $J_{\SympBase}$, the compactified moduli space $\overline{\ModSpace_{\Leg}\left(\boundaryWord\right)}$ is a smooth manifolds of its expected dimension $\ind\left(\boundaryWord\right) - 1$ with corners, provided $\ind\left(\boundaryWord\right) \leq 2$. With such $J_{\SympBase}$ all admissible $\ind(u) = 0$ holomorphic disks are trivial trips.
\end{lemma}

Provided Lemma \ref{Lemma:LCHShift}, this is a consequence of \cite[Proposition 2.3]{EES:PtimesR}. We can similarly use a fixed $J_{\SympBase}$ and generic $\Cinfty$ small perturbations of $\Leg$ through Legendrian isotopy. The following is the corresponding statement for holomorphic curves in symplectizations following \cite[Lemma B.6]{Ekholm:Z2RSFT}:

\begin{lemma}\label{Lemma:CobTransversality}
For a fixed $\Lag$ and generic choice of $J$, the compactified moduli space $\overline{\ModSpace_{\Lag}\left(\boundaryWord\right)}$ is a smooth manifolds of its expected dimension $\ind\left(\boundaryWord\right)$ with corners, provided $\ind\left(\boundaryWord\right) \leq 1$.
\end{lemma}

\section{Differentials, cobordism morphisms, and basic structural features}\label{Sec:DiffsAlgStructures}

Now that we've defined the $\mu_{\boundaryWordThicc}$ operators and the moduli spaces $\ModSpace_{\Leg}\left(\boundaryWord\right)$ and $\ModSpace_{\Lag}\left(\boundaryWord\right)$, we are ready to state the definitions of $\partial$ and cobordism maps. We use $\field[H_{1}]$ coefficient systems unless otherwise indicated.

\subsection{Definition of $\partial$}

Here we work with a fixed chord generic $\Leg$ with associated $\field$ module $\orbitVS$ spanned by the admissible $\Partition$ cyclic words with word length filtration and algebra $\Algebra = \field[H_{1}(\Leg)]\otimes \tensorAlg(\orbitVS)$ with filtrations
\begin{equation*}
\field[H_{1}] = \filtration^{0}\Algebra \subset \filtration^{1}\Algebra \cdots \subset \filtration^{N^{\Partition}}\Algebra = \filtration^{N^{\Partition}+1}\Algebra = \cdots = \Algebra, \quad \filtration^{\filtLevel}\Algebra = \field[H_{1}]\otimes \tensorAlg(\filtration^{\filtLevel}\orbitVS).
\end{equation*}

Suppose that $J_{\SympBase}$ is chosen so that the conclusions of Lemma \ref{Lemma:SympBaseTransversality} are satisfied. According to Lemma \ref{Lemma:SympBaseIndex}, if $\ModSpace_{\Leg}(\boundaryWordThicc)$ is nonempty and $\ind(\boundaryWordThicc) = 1$, then $\boundaryWordThicc = \{\boundaryWord_{i} \}$ consists of one $\boundaryWord_{i} \in \boundaryWordThicc$ having $\ind(\boundaryWord_{i}) = 1$ with all other elements of $\boundaryWordThicc$ having $\ind = 0$, so that they are trivial strips. In this case the moduli space must consist of a finite number of points by compactness. We define
\begin{equation*}
\partial: \orbitVS \rightarrow \Algebra, \quad \partial = \sum_{\ind(\boundaryWordThicc) = 1} \#\ModSpace_{\Leg}(\boundaryWordThicc)\mu_{\boundaryWordThicc}
\end{equation*}
To extend $\partial$ to all of $\Algebra$, we set $\partial 1 = 0$, apply the Leibniz rule
\begin{equation*}
\partial(x \cdot y) = (\partial x)\cdot y + x \cdot (\partial y),
\end{equation*}
and extend linearly over $\field[H_{1}]$.

\begin{thm}
$(\Algebra, \partial, \filtration)$ is a max-filtered DGA.
\end{thm}

\begin{proof}
Clearly $\Algebra$ is a max-filtered graded algebra. Because $\partial$ is defined using $\mu_{\boundaryWordThicc}$ operators, it preserves the word length filtration by Lemma \ref{Lemma:FiltrationPreservation}. That $\partial$ has degree $-1$ follows from Lemma \ref{Lemma:MuDegree}. Indeed $|\word| - |\partial \word|$ is computed as the index of the holomorphic disks whose count contributes to $\partial$, which is always $1$. To establish that $\partial$ squares to zero, it suffices to analyze $\partial|_{\orbitVS}$ as we're using the Leibniz rule.

For a word $\word$, an element $h \in H_{1}(\Leg)$, and an ordered collection $\word^{-}_{1}, \dots , \word^{-}_{m}$ of boundary words for which $|\word| - 1 = \sum |\word^{-}_{i}|$ write $\thicc{\word}$ for the product $\thicc{\word} = \word^{-}_{1}\cdots \word^{-}_{m}$ and $\langle \partial^{2}\word, e^{h}\thicc{\word} \rangle \in \field$ for the coefficient of $e^{h}\thicc{\word}$ in $\partial^{2}\word$. Ignoring trivial strips, contributions to $\langle \partial^{2}\word, \thicc{\word}\rangle$ are ordered pairs $u_{1}, u_{2}$ of holomorphic disks with
\begin{equation*}
\begin{gathered}
\ind(u_{1}) = \ind(u_{2}) = 1,\quad h(u_{1}) + h(u_{2}) = h,\\
\planarDiagram(u_{1}) \subset \planarDiagram(\word), \quad \planarDiagram(u_{2}) \subset \left(\planarDiagram(\word) \setminus \planarDiagram(u_{1})\right).
\end{gathered}
\end{equation*}
Recall that the $h(u_{k})$ are defined in \S \ref{Sec:TwistedCoeffs}. The inclusions in the second row indicate of the above equation indicate inscriptions of planar diagrams. There are two cases to consider.

First, if some positive puncture of $u_{2}$ connects to a negative puncture of $u_{1}$, then $\planarDiagram(u_{2}) \cup \planarDiagram(u_{1})$ forms a connected region in $\planarDiagram(\word)$ whose complement is the union of the $\planarDiagram(\word_{i}^{-})$. By gluing, a bubble tree with vertices assigned to the $u_{k}$ and a single edge determines a point in the boundary of the moduli space of $\ind(u) = 2$ admissible curves $\ModSpace_{\Leg}\left(\boundaryWord\right)$ with $\mu_{\boundaryWord}\word = e^{h}\thicc{\word}$. It follows that $\langle \partial^{2}\word, e^{h}\thicc{\word} \rangle$ counts points in $\partial \overline{\ModSpace_{\Leg}\left(\boundaryWord\right)}$. By the compactness results of Lemma \ref{Lemma:SympBaseTransversality}, this is a count of points in the boundary of a compact $1$-manifold and so is zero.

Second, all positive punctures of $u_{2}$ may be sent to positive punctures of $\word$ in which case the $\planarDiagram(u_{k}) \subset \planarDiagram(\word)$ for $k=1, 2$ are disjoint and the positive punctures of the $u_{k}$ must all be distinct. Tracking the sequential orderings of the endpoints of negative chords around the boundary of $\planarDiagram(\word)$, it follows that there are $m_{l}, m_{r}, m'_{l}, m'_{r} \in \Z_{\geq 0}$ for which
\begin{equation*}
(1^{\otimes m_{l}} \otimes \mu_{\boundaryWord(u_{1})}\circ 1^{\otimes m_{r}})\circ \mu_{\boundaryWord(u_{2})}\word =(1^{\otimes m'_{l}} \otimes \mu_{\boundaryWord(u_{2})}\circ 1^{\otimes m'_{r}})\mu_{\boundaryWord(u_{1})}\word.
\end{equation*}
So there is a $\Z/2\Z$ action on the pairs $(u_{1}, u_{2})$ of index $1$ disks contributing to $\langle \partial^{2}\word, \thicc{\word} \rangle$ given by swapping the order of the pair. Since this action is free and we are working over $\field = \Z/2\Z$, $\langle \partial^{2}\word, \thicc{\word} \rangle = 0$.
\end{proof}

\subsection{Definition of cobordism morphisms}

Now we work with a partitioned exact Lagrangian cobordism $\LagGrouped: \LegGroupedUp \rightarrow \LegGroupedDown$ with associated filtered $\field$ modules and algebras $\Algebra^{\pm}$. When $\Lag$ is a homology cobordism, we work with $\field[H_{1}(\Lag)]$. Otherwise we use $\field$ coefficients.

Assume that an almost complex structure $J$ on $\R_{s} \times M$ has been chosen which satisfies the conclusions of Lemma \ref{Lemma:CobTransversality}. Then if $\ModSpace_{\Lag}(\boundaryWordThicc)$ is non-empty and satisfies $\ind(\boundaryWordThicc) = 0$, then each $\boundaryWord \in \boundaryWordThicc$ has $\ind\left(\boundaryWord\right) = 0$. Again, the moduli space must consist of a number of points by compactness. We define
\begin{equation*}
\newRSFTA\LagGrouped: \orbitVS^{+} \rightarrow \AlgebraDown, \quad \newRSFTA\LagGrouped = \sum_{\ind(\boundaryWordThicc)=0} \#\ModSpace(\boundaryWordThicc)\mu_{\boundaryWordThicc}
\end{equation*}
and extend to a unital algebra morphism,
\begin{equation*}
\newRSFTA\LagGrouped 1_{\AlgebraUp} = 1_{\AlgebraDown}, \quad \newRSFTA\LagGrouped(x \cdot y) = \newRSFTA\LagGrouped x \cdot \newRSFTA\LagGrouped y.
\end{equation*}

\begin{thm}
$\newRSFTA\LagGrouped$ is a filtration preserving degree $0$ chain map,
\begin{equation}\label{Eq:ThmCobChain}
\partial^{-}\circ \newRSFTA\LagGrouped - \newRSFTA\LagGrouped \circ \partial^{+} = 0.
\end{equation}
\end{thm}

\begin{proof}
Filtration preservation and the degree $0$ condition again follow from Lemmas \ref{Lemma:FiltrationPreservation} and \ref{Lemma:MuDegree}. The chain map condition can again be checked on individual elements of $\orbitVS^{+}$ using the compactness and transversality results of Lemma \ref{Lemma:CobTransversality} and will follow from standard compactness and gluing arguments.

For $\ind(\boundaryWordThicc) = 1$, the boundary of the compactification of the moduli space $\ModSpace_{\Lag}(\boundaryWordThicc)$ will consist of collections of bubble trees $\{ \tree_{i} \}$ having $1$ vertex of one $\tree_{i}$ assigned an $\ind(u) = 1$ holomorphic map and the remaining vertices assigned $\ind(u) = 0$ holomorphic maps.

The $\ind = 1$ vertex must correspond to a holomorphic map in one of the $\Lag_{\Leg^{\pm}}$. If not, then it corresponds to an element in the interior of $\ModSpace_{\Lag}$, which is impossible since we are at the boundary of the moduli space.

The $\ind = 0$ vertices must all either correspond to trivial strips or $\ModSpace_{\Lag}$ disks. Otherwise, we could take the corresponding disk, project it down to $\SympBase$ and obtain a disk living in a $\ModSpace_{\Leg^{\pm}}$ of expected dimension $-1$ in violation of our transversality assumptions. Therefore all of the holomorphic disks associated to the $\ind = 0$ vertices give us a $\newRSFTA\LagGrouped$ contribution.

If the $\ind = 1$ holomorphic disk is a $\Lag_{\Leg^{+}}$ disk, then by looking at the planar diagram for the bubble tree it contributes to $\partial^{+}\word$ (after applying Lemma \ref{Lemma:Lifting}) and the $\ind = 0$ disks contribute to $\newRSFTA\LagGrouped \partial^{+}\word$. Otherwise the $\ind = 0$ disks contribute to $\newRSFTA\LagGrouped\word$ and the $\ind = 1$ disk contribute to $\partial^{-}\newRSFTA\LagGrouped\word$ (again via Lemma \ref{Lemma:Lifting}).

So points in the boundary of our compactified $\ind = 1$ moduli spaces contribute to $\partial^{-}\newRSFTA\LagGrouped - \newRSFTA\LagGrouped\partial^{+}$. By gluing, this is exactly the count of points in the boundary of the $\ModSpace_{\Lag}(\boundaryWordThicc)$ spaces of $\ind = 1$. Since these are compact $1$-manifolds by Lemma \ref{Lemma:CobTransversality}, the proof is complete.
\end{proof}

\section{max-filtered DGAs and their invariants}\label{Sec:mfDGA}

In this section we describe basic properties of max-filtered DGAs and their invariants. Throughout $\rho$ is a non-negative integer and $\ring$ is a unital $\field$ algebra.

\subsection{Basic definitions}

\begin{defn}
A $\Z/2\rho\Z$-graded \emph{differential graded algebra} (DGA) over a unital ring $\ring$ is an algebra $\Algebra$ over $\ring$ with graded components and differential $\partial$
\begin{equation*}
\Algebra = \oplus_{i \in \Z/2\rho\Z}\Algebra_{i}, \quad \partial: \Algebra_{i} \rightarrow \Algebra_{i-1}
\end{equation*}
satisfying the relations
\begin{equation*}
\Algebra_{i} \cdot \Algebra_{i'} \subset \Algebra_{i + i'}, \quad \partial (xy) = (\partial x)y + (-1)^{|x|}x(\partial y), \quad \partial^{2} = \partial 1 = 0.
\end{equation*}
\end{defn}

We also consider DGAs whose differentials have degree $1$ in which case we use $d$ in place of $\partial$. max-filtered graded algebras (mfGAs) and max-filtered DGAs (mfDGAs) have already been defined in the introduction. The following example demonstrates the abundance of mfDGAs in algebraic topology.

\begin{ex}
Set $\field = \R$ or $\C$, $\rho=0$, and consider a compact manifold $L$ equipped with an exhaustion by a countable collection of nested open sets
\begin{equation*}
U^{0} \subset U^{1} \subset \cdots L.
\end{equation*}
The open cover induces a filtration $\filtration$ on the de Rham complex $(\Omega^{\ast} = \Omega^{\ast}(L), d)$ with $\filtration^{\filtLevel}\Omega^{\ast}$ being the space of forms with compact support contained in $U^{\filtLevel}$. With the wedge product $\wedge$, $(\Omega^{\ast}, d, \filtration)$ is a mfDGA as
\begin{equation*}
\filtration^{\filtLevel}\Omega^{\ast} \wedge \filtration^{\filtLevel'}\Omega^{\ast} \subset \filtration^{\min(\filtLevel, \filtLevel')}\Omega^{\ast} \subset \filtration^{\max(\filtLevel, \filtLevel')}\Omega^{\ast}.
\end{equation*}
This set up can be used to derive classical cohomological spectral sequences such as the Leray-Serre sequence for a fibration. See for example \cite{BottTu}.
\end{ex}

A \emph{morphism of mfGAs}
\begin{equation*}
\phi: (A, \filtration_{A}) \rightarrow (B, \filtration_{B})
\end{equation*}
is a degree zero algebra morphism which preserves the unit and filtration structures:
\begin{equation*}
\phi(a_{1}a_{2}) = \phi(a_{1})\phi(a_{2}),\quad \phi(1_{A}) = 1_{B}, \quad \phi (\filtration^{\filtLevel}_{A}A) \subset \filtration^{\filtLevel}_{B}B.
\end{equation*}
A \emph{morphism} of mfDGAs
\begin{equation*}
\phi: (\Algebra, \partial_{\Algebra}, \filtration_{\Algebra}) \rightarrow (\AlgebraOther, \partial_{\AlgebraOther}, \filtration_{\AlgebraOther})
\end{equation*}
is a mfGA morphism which is also a chain map, $\partial_{\AlgebraOther}\phi = \phi \partial_{\Algebra}$.

A \emph{chain homotopy} of DGA morphisms $\phi, \psi: \Algebra \rightarrow \AlgebraOther$ is a degree $1$ map $h: \Algebra \rightarrow \AlgebraOther$ satisfying
\begin{equation*}
\phi - \psi = \partial_{\AlgebraOther}\circ h + h \circ \partial_{\Algebra}.
\end{equation*}
A \emph{chain homotopy of mfDGA morphisms} is a chain homotopy of DGA morphisms which preserves the filtration structure,
\begin{equation*}
h (\filtration^{\filtLevel}_{\Algebra}\Algebra ) \subset \filtration^{\filtLevel}_{\AlgebraOther}\AlgebraOther.
\end{equation*}
A \emph{chain homotopy equivalence} between (DGAs or mfDGAs) $\Algebra$ and $\AlgebraOther$ is a pair of morphism
\begin{equation*}
\begin{tikzcd}
\Algebra \arrow[r, bend left=50, "\phi"] & \AlgebraOther \arrow[l, bend left=50, "\psi"]
\end{tikzcd}
\end{equation*}
such that $\psi \phi$ is chain homotopic to $\Id_{\Algebra}$ and $\phi\psi$ is chain homotopic to $\Id_{\AlgebraOther}$.

Each filtered piece $\filtration^{\filtLevel}\Algebra$ of $\Algebra$ is itself an mfDGA and the canonical examples of mfDGA morphisms are inclusion maps
\begin{equation}\label{Eq:FiltrationSequence}
\filtration^{0}\Algebra \rightarrow \filtration^{1}\Algebra \rightarrow \cdots \rightarrow \Algebra.
\end{equation}
Provided a morphism $\phi$ as above, inclusions of the filtered pieces of $\Algebra$ and $\AlgebraOther$ fit together in a commutative diagram of mfDGA morphisms
\begin{equation}\label{Eq:FiltrationLadder}
\begin{tikzcd}
\filtration^{0}_{\Algebra}\Algebra \arrow[r]\arrow[d] & \filtration^{1}_{\Algebra}\Algebra \arrow[d] \arrow[r] & \cdots \arrow[r] & \Algebra \arrow[d]\\
\filtration^{0}_{\AlgebraDown}\AlgebraOther \arrow[r] & \filtration^{1}_{\AlgebraOther}\AlgebraOther \arrow[r] & \cdots \arrow[r] & \AlgebraOther
\end{tikzcd}
\end{equation}

Because mfDGA chain homotopies are filtration preserving, a chain homotopy $h$ between pairs of mfDGA morphisms $\phi, \psi: \Algebra \rightarrow \AlgebraOther$ will induce chain mfDGA chain homotopies between filtered pieces
\begin{equation*}
\phi, \psi: \filtration^{\filtLevel}_{\Algebra}\Algebra \rightarrow \filtration^{\filtLevel}_{\AlgebraOther}\AlgebraOther.
\end{equation*}
Consequently if $\Algebra$ is mfDGA homotopic to $\AlgebraOther$, then each $\filtration^{\filtLevel}_{\Algebra}\Algebra$ is homotopic to $\filtration^{\filtLevel}_{\AlgebraOther}\AlgebraOther$ for each $\filtLevel$.

\subsection{Homology and torsion}

The homology of a mfDGA is a mfGA which is an invariant of the mfDGA chain homotopy class of $\Algebra$. The inclusion morphisms of Equation \eqref{Eq:FiltrationSequence} induce mfGA morphisms
\begin{equation*}
H(\filtration^{0}\Algebra) \rightarrow H(\filtration^{1}\Algebra) \rightarrow \cdots \rightarrow H(\Algebra).
\end{equation*}
Applying homology to Equation \eqref{Eq:FiltrationLadder} produces a commutative diagram of mfGA morphisms
\begin{equation*}
\begin{tikzcd}
H(\filtration^{0}\Algebra) \arrow[r]\arrow[d] & H(\filtration^{1}\Algebra) \arrow[d] \arrow[r] & \cdots \arrow[r] & H(\Algebra) \arrow[d]\\
H(\filtration^{0}\AlgebraOther) \arrow[r] & H(\filtration^{1}\AlgebraOther) \arrow[r] & \cdots \arrow[r] & H(\AlgebraOther).
\end{tikzcd}
\end{equation*}

The homology $H(\Algebra)$ of a DGA vanishes if and only if $1 \in \Algebra$ is exact, $1 \in \im \partial$. So if we have a mfDGA morphism $\phi$ as above, then
\begin{equation}\label{Eq:VanishingImpliesVanishing}
H(\Algebra) = 0 \implies H(\AlgebraOther) = 0.
\end{equation}
So if $H(\filtration^{\filtLevel}\Algebra) = 0$ for some $\filtLevel$, then $H(\filtration^{\filtLevel'}\Algebra) = 0$ for all $\filtLevel' \geq \filtLevel$.

\begin{defn}
The \emph{H-torsion of an mfDGA},
\begin{equation*}
\tau_{H} = \tau_{H}(\Algebra) \in \Z_{\geq 0} \cup \{\infty\},
\end{equation*}
is the greatest $\filtLevel$ for which $H(\filtration^{\filtLevel}\Algebra) \neq 0$.
\end{defn}

Clearly $\tau_{H}$ is a quasi-isomorphism invariant of an mfDGA and the existence of a morphism $\phi$ implies
\begin{equation*}
\tau_{H}(\AlgebraOther) \leq \tau_{H}(\Algebra),
\end{equation*}
refining Equation \eqref{Eq:VanishingImpliesVanishing}.

\subsection{Augmentations and A-torsion}

\begin{defn}
Let $\Algebra$ be a $\Z/2\rho$-graded DGA. An \emph{augmentation of $\Algebra$} is a DGA morphism
\begin{equation*}
\aug: \Algebra \rightarrow \field,
\end{equation*}
where $\field$ is viewed as a $\Z/2\rho$-graded DGA with trivial differential.
\end{defn}

\begin{rmk}
While we only discuss augmentations with target $\field$, representations to other rings are interesting and useful in applications \cite{NR:Helix, Sivek:NoAugs}. Interesting augmentations also come from reducing gradings of some $\Z/2\rho'$ graded DGA to $\Z/\rho$ with $\rho$ not necessarily even. See \cite{EtnyreNg:LCHSurvey, NS:AugRuling} for further discussion and results.
\end{rmk}

The primary utility of augmentations is to establish the non-vanishing of DGAs \cite{Chekanov:LCH}:

\begin{lemma}\label{Lemma:AugImpliesHomologyNonzero}
The existence of an $\aug$ implies that $H(\Algebra) \neq 0$.
\end{lemma}

When $\Algebra$ is an mfDGA, the filtration structure allows us to pull back augmentations along inclusion morphisms: If $\aug$ is an augmentation of $\filtration^{\filtLevel}\Algebra$, then there are induced augmentations
\begin{equation}\label{Eq:AugPullback}
\begin{tikzcd}
\filtration^{0}\Algebra \arrow[r] \arrow[d, dashed] & \filtration^{1}\Algebra \arrow[r]\arrow[d, dashed] & \cdots \arrow[r] & \filtration^{\filtLevel}\Algebra \arrow[d, "\aug"]\\
\field \arrow[r, "\Id"] & \field \arrow[r, "\Id"] & \cdots \arrow[r, "\Id"] & \field
\end{tikzcd}
\end{equation}
For a given $\Algebra$, we can ask which $\filtration^{\filtLevel}\Algebra$ admit augmentations and package the answer with the following:

\begin{defn}
The \emph{A-torsion}
\begin{equation*}
\tau_{A} = \tau_{A}(\Algebra, \partial, \filtration) \in \Z_{\geq 0} \cup \{\infty\}
\end{equation*}
is the greatest $\filtLevel$ for which $\filtration^{\filtLevel}\Algebra$ admits an augmentation.
\end{defn}

So $\tau_{A} = \infty$ if and only if $\Algebra$ admits an augmentation. According to Lemma \ref{Lemma:AugImpliesHomologyNonzero},
\begin{equation*}
\tau_{A}(\Algebra) \leq \tau_{H}(\Algebra).
\end{equation*}
By our ability to pull back augmentations along $\phi$, the existence of a mfDGA morphism $\phi: \Algebra \rightarrow \AlgebraOther$ implies
\begin{equation*}
\tau_{A}(\Algebra) \geq \tau_{A}(\AlgebraOther).
\end{equation*}
The same reasoning implies that the A-torsion is a mfDGA chain homotopy invariant.

For a given $\aug: \filtration^{\filtLevel}\Algebra \rightarrow \field$ it is natural to ask if an extension exists
\begin{equation}\label{Eq:AugExtension}
\begin{tikzcd}
\filtration^{\filtLevel}\Algebra \arrow[r]\arrow[dr, "\aug"] & \filtration^{\filtLevel + 1}\Algebra \arrow[d, dashed, "\exists ?"]\\
& \field,
\end{tikzcd}
\end{equation}
and if so, when is such an extension unique. The \emph{augmentation tree}, defined in \S \ref{Sec:AugTrees} below, will package the answer to this question on a homotopical level.

\subsection{Free mfDGAs}

Let $\ring$ be a graded $\field$ algebra. A mfDGA $(\Algebra, \partial, \filtration)$ is \emph{free} if there is a collection of generators $a_{i}^{\filtLevel}, \filtLevel\geq 0$ with degree gradings $|a_{i}^{\filtLevel}| \in \Z/2\rho\Z$ from which we can define a filtered, graded $\ring$ module
\begin{equation*}
\filtration^{0}\orbitVS \subset \filtration^{1}\orbitVS \subset \cdots \subset \orbitVS, \quad \filtration^{\filtLevel}\orbitVS_{\deg} = \bigoplus_{\filtLevel' \leq \filtLevel,\  |a^{\filtLevel'}_{i}|=\deg}\ring a_{i}^{\filtLevel'}, \quad \filtration^{\filtLevel}\orbitVS = \bigoplus_{\deg \in \Z/2\rho\Z} \filtration^{\filtLevel}V_{\deg}.
\end{equation*}
so that $(\Algebra, \filtration)$ is the tensor algebra with the induced filtration
\begin{equation*}
\filtration^{0}\Algebra \subset \filtration^{1}\Algebra \subset \cdots \subset \Algebra, \quad \filtration^{\filtLevel}\Algebra = \tensorAlg_{\ring}(\filtration^{\filtLevel}\orbitVS)
\end{equation*}
where $\tensorAlg_{\ring}$ is the tensor algebra over $\ring$. We require that the differential $\partial$ satisfies the Leibniz rule with respect to tensor multiplication. A free mfDGA is \emph{finitely generated} if $\#(a_{i}^{\filtLevel}) \leq \infty$. If $(\Algebra, \partial, \filtration)$ is free, then $\partial$ is entirely determined by $\partial|_{\orbitVS}$. Free DGAs are defined similarly with filtrations ignored \cite{Chekanov:LCH}. Free commutative mfDGAs $\Algebra^{\com}$ are defined similarly, using the exterior algebra $\filtration^{\filtLevel}\Algebra^{\com} = \tensorAlgCom(\filtration^{\filtLevel}\orbitVS)$ instead of the full tensor algebra.

\begin{rmk}
For $\newRSFTA$ we can take $\ring$ to be $\field$ or $\field[H_{1}(\Leg)]$. In the latter case we can identify
\begin{equation*}
\tensorAlg_{\ring}(\oplus \ring \word_{i}) = \ring \otimes_{\field} \left(\tensorAlg_{\field}(\oplus \field \word_{i})\right)
\end{equation*}
so that the notation of this section matches that of the rest of the paper.
\end{rmk}

\begin{ex}
Any free DGA becomes a free mfDGA by declaring that $\orbitVS = \filtration^{0}\orbitVS$ or $\filtration^{1}\orbitVS$. The Chekanov-Eliashberg algebra of a Legendrian $\Leg$ in a contact manifold of the form $\R \times \SympBase$ is the canonical example of a free DGA. The contact homology algebra $CH\Mxi$ of a contact manifold $\Mxi$ is the canonical example of a free commutative DGA.
\end{ex}

\begin{ex}
Let $\Leg$ be a Legendrian submanifold whose chords $\chord_{i}$ have action $\action_{i} = \action(\chord_{i})$ with respect to some contact form. Assume that the $\chord_{i}$ are indexed so that the $\action_{i}$ are monotonically ordered $\action_{i} \leq \action_{i+1}$. The we can label $a^{i}_{i} = \chord_{i}$ to obtain a filtered $\filtration$ vector space spanned by the $\chord_{i}$ as above. This induces an mfDGA structure on $\CE(\Leg)$ by the fact that $\partial_{\CE}$ is action decreasing. Such action filtrations are of frequent use in quantitative applications \cite{DRS:Persistence} or direct limit construction in symplectic field theory, cf. \cite{BH:ContactDefinition, Siegel:RSFT}. Using alternative language, $\tau_{A}$ appears in \cite[Theorem 1.1]{DRS:Persistence}.
\end{ex}

\begin{assump}
Throughout the remainder of the article, mfDGAs are free and finitely generated.
\end{assump}

Let $\Algebra$ have generating set $\{ a^{\filtLevel}_{i} \}$ and $\AlgebraOther$ have generating set $\{b^{\filtLevel}_{i}\}$. A mfDGA isomorphism $\phi: \Algebra \rightarrow \AlgebraOther$ is \emph{elementary} if there is some $j$ for which
\begin{equation*}
\phi(a^{\filtLevel}_{i}) = \begin{cases}
b^{\filtLevel}_{i} & i \neq j\\
\const b^{\filtLevel}_{j} + v_{j} & i = j
\end{cases}
\end{equation*}
for some unit $\const \in \ring$ and element $v_{j} \in \AlgebraOther$ which does not have $b^{\filtLevel}_{i}$ as a factor. A mfDGA isomorphism is \emph{tame} if it is a composition of elementary isomorphisms.

\begin{defn}
For a free $(\Algebra, \partial, \filtration)$ with generating set $\{ a_{i}^{\filtLevel}\}$, define the \emph{$(f, \deg)$-stabilization $(\Algebra, \partial, \filtration)$} to be the mfDGA $(\Stab_{f,\deg}\Algebra, \partial, \filtration)$ with
\be
\item underlying graded vector space $\Stab_{f,\deg}\orbitVS = \orbitVS \oplus \ring e^{f}_{\deg} \oplus \ring e^{f}_{\deg-1}$ where $|e^{f}_{j}| = j$,
\item filtration $\filtration^{\filtLevel}S_{f,\deg}\orbitVS = \filtration^{\filtLevel}\orbitVS$ for $\filtLevel < f$ and $\filtration^{\filtLevel}S_{f,\deg}\orbitVS = \filtration^{\filtLevel}\orbitVS \oplus  \ring e^{f}_{\deg} \oplus \ring e^{f}_{\deg-1}$ for $\filtLevel \geq f$,
\item differential extended by $\partial e_{\deg}^{f} = e_{\deg-1}^{f}$ and $\partial e_{\deg-1}^{f} = 0$.
\ee
\end{defn}

\begin{defn}\label{Def:FilteredStableTameIso}
A pair of mfDGAs $(\Algebra, \partial_{\Algebra}, \filtration_{\Algebra})$, $(\AlgebraOther, \partial_{\AlgebraOther}, \filtration_{\AlgebraOther})$ are \emph{filtered stable tame isomorphic} if there are finite sequences
\begin{equation*}
(f_{a, k}, \deg_{a, k}), k=1, \dots, n_{a}, \quad (f_{b, k}, \deg_{b, k}), k=1, \dots, n_{b}
\end{equation*}
and tame isomorphisms
\begin{equation*}
\phi: (\Stab_{f_{a, n_{a}},\deg_{a, n_{a}}}\cdots \Stab_{f_{a, 1},\deg_{a, 1}}\Algebra, \partial_{\Algebra}, \filtration_{\Algebra}) \rightarrow (\Stab_{f_{b, n_{b}},\deg_{b, n_{b}}}\cdots \Stab_{f_{b, 1},\deg_{b, 1}}\AlgebraOther, \partial_{\AlgebraOther}, \filtration_{\AlgebraOther}).
\end{equation*}
\end{defn}

Provided a stabilization $\Stab_{f, \deg}\Algebra$ of a mfDGA, we have a naturally defined pair of mfDGA morphisms
\begin{equation*}
\begin{tikzcd}
\Stab_{f, \deg}\Algebra \arrow[r, bend left=50, "\pi"] & \Algebra \arrow[l, bend left=50, "\inc"]
\end{tikzcd},
\quad \pi(v) = \begin{cases}
0 & v = e^{f}_{\deg}, e^{f}_{\deg-1}, \\
v & \text{otherwise}.
\end{cases}
\end{equation*}
Clearly $\pi \inc = \Id_{\Algebra}$. According to \cite[Corollary 3.11]{ENS:Orientations}, $\inc \pi$ is chain homotopic to $\Id_{\Stab_{f, \deg}\Algebra}$.\footnote{The proof in \cite{ENS:Orientations} for DGAs extends to mfDGAs without modification. The homotopy is given below in Equation \eqref{Eq:StabChainHomotopy} when $\field = \Z/2\Z$.} This implies the following basic result.

\begin{lemma}
A free mfDGA is homotopic to its stabilization. Therefore the homology of an mfDGA is a stable tame isomorphism invariant.
\end{lemma}

As chain homotopies of mfDGAs preserve filtrations, we have the following obvious corollary:

\begin{cor}
The mfGA isomorphism class of each $H(\filtration^{\filtLevel}\Algebra)$ is a filtered stable tame isomorphism invariant of $\Algebra$.
\end{cor}

\subsection{Augmentations of free mfDGAs}

An augmentation of a free DGA $\Algebra$ is completely determined by its values on $\orbitVS_{0}$, the $\deg = 0$ summand of $\orbitVS$.

\subsubsection{Augmentation varieties}

For each generator $a \in \orbitVS_{1}$ define a polynomial $D_{a} \in \field[\orbitVS_{0}]$ (the ring of polynomial functions on $V_{0}$) by taking $\partial a$, eliminating all terms which are not in $\orbitVS_{0}$, and making all variables commute. The algebraic variety
\begin{equation*}
\Variety_{\Aug} = \Variety_{\Aug}\Algebra= \{ v \in \orbitVS_{0}\ :\ D_{a}(v) = 0\ \forall a \in \orbitVS_{1} \} \subset \orbitVS_{0}
\end{equation*}
is the \emph{augmentation variety of $\Algebra$} \cite{Ng:ComputableInvariants}. By construction the points in $\Variety_{\Aug}\Algebra$ are exactly the augmentations of $\Algebra$. Indeed, suppose that $V_{0}$ is generated by some $a_{0, i}$. Given $v = \sum \epsilon_{i} a_{0, i}$ we have an augmentation
\begin{equation*}
\aug_{v}: \Algebra \rightarrow \field, \quad \aug_{v}(a) = \begin{cases}
\epsilon_{i} & a = a_{0, i} \\
0 & |a| \neq 0
\end{cases}
\end{equation*}
and the condition $\aug \partial = 0$ determines a point of $\Variety_{\Aug}$ from each $\aug$. These correspondences are bijective.

If $\Algebra$ is a mfDGA, then each $\filtration^{\filtLevel}\Algebra$ has an augmentation variety $\Variety_{\Aug}^{\filtLevel}$. Each inclusion morphism $\filtration^{\filtLevel}\Algebra \rightarrow \filtration^{\filtLevel + 1}\Algebra$ induces a morphism of varieties via Equation \eqref{Eq:AugPullback}
\begin{equation*}
\pi_{\filtLevel + 1}: \Variety_{\Aug}^{\filtLevel + 1} \rightarrow \Variety_{\Aug}^{\filtLevel}
\end{equation*}
so that $\im \pi_{\filtLevel} \subset \Variety_{\Aug}^{\filtLevel}$ is exactly the set of $\aug: \filtration^{\filtLevel}\Algebra \rightarrow \field$ which admit extensions solving Equation \eqref{Eq:AugExtension}.

Augmentation varieties are non-trivially modified by stabilization. It is not difficult to check that
\begin{equation*}
\Variety_{\Aug}\Stab_{f, \deg}\Algebra \simeq \begin{cases}
\Variety_{\Aug}\Algebra & \deg  \neq 0 \\
\field \times \Variety_{\Aug}\Algebra & \deg = 0
\end{cases}
\end{equation*}
where the new $\field$ factor in the case $\deg = 0$ comes from possible assignments of the stabilized variable $e^{f}_{0}$ to elements of $\field$. For an $\aug \in \Variety_{\Aug}(\Algebra)$ and $c \in \field$ write $(c, \aug)$ for the corresponding element of $\Variety_{\Aug}(\Stab_{f, 0})$. The proof of the following is easy using Equation \eqref{Eq:StabChainHomotopy}. 

\begin{lemma}
Each of the $(c, \aug)$ are chain homotopic as DGA morphisms. Therefore for each $\filtLevel$, the set augmentations of $\filtration^{\filtLevel}\Algebra \rightarrow \field$ modulo DGA chain homotopy is a stable tame isomorphism invariant of $\Algebra$.
\end{lemma}

\begin{rmk}\label{Rmk:AugHomotopy}
Observe that if the degree $-1$ summand $\orbitVS_{-1}$ of $\orbitVS$ is trivial, $\orbitVS_{-1} = 0$, then chain homotopy classes of augmentations are exactly the augmentations themselves.
\end{rmk}

\subsubsection{Augmentation trees}\label{Sec:AugTrees}

Results of \cite{BG:Bilinearized} discussed in the next subsection provide us a means of distinguishing homotopy classes of augmentations. The chain homotopy classes of augmentations of the $\filtration^{\filtLevel}\Algebra \subset \Algebra$ can be organized as a rooted tree
\begin{equation*}
\tree_{\Aug} = \tree_{\Aug}(\Algebra) = (\{\vertex_{i}\}, \{\edge_{j} \})
\end{equation*}
which encodes all restriction and extensions of augmentations described in Equations  \eqref{Eq:AugPullback} and \eqref{Eq:AugExtension} (modulo chain homotopy) as follows:
\be
\item Each vertex $\vertex_{i}$ has a depth $\filtLevel(\vertex_{i}) \in \Z_{\geq 0}$.
\item There is exactly one $\vertex_{i}$ for each DGA homotopy class $\Aug(\vertex_{i}) = [\aug]$ of $\aug: \filtration^{\filtLevel(\vertex_{i})}\Algebra \rightarrow \field$.
\item The root vertex is the trivial augmentation $\Id_{\field}: (\filtration^{0}\Algebra = \field) \rightarrow \field$ with depth $0$.
\item We add an edge $\edge_{j}: \vertex_{i} \rightarrow \vertex_{i'}$ if $\filtLevel(\vertex_{i}) = \filtLevel(\vertex_{i'}) + 1$ and $\Aug(\vertex_{i})$ induces $\Aug(\vertex_{i'})$ via the inclusion $\filtration^{\filtLevel(\vertex_{i'})}\Algebra \rightarrow \filtration^{\filtLevel(\vertex_{i})}\Algebra$.
\ee

\begin{defn}
$\tree_{\Aug}(\Algebra)$ is the \emph{augmentation tree of $\Algebra$}.
\end{defn}

By composing stable tame isomorphisms with chain homotopies, is it clear that $\tree_{\Aug}$ is a filtered stable tame isomorphism invariant of a free mfDGA $\Algebra$. Therefore the following is a consequence of Theorem \ref{Thm:Main}.

\begin{thm}
$\tree_{\Aug}(\newRSFTA)$ is a homotopy invariant of $\LegGrouped$ in the sense of Theorem \ref{Thm:Main} and so we can write $\tree_{\Aug}\LegGrouped$. A partitioned exact Lagrangian cobordism $\LagGrouped: \LegGroupedUp \rightarrow \LegGroupedDown$ induces a morphism of trees
\begin{equation*}
\tree_{\Aug}\LagGrouped: \tree_{\Aug}\LegGroupedDown \rightarrow \tree_{\Aug}\LegGroupedUp
\end{equation*}
which is a homotopy invariant of $\LagGrouped$.
\end{thm}

By a morphism of trees, we mean maps between vertex sets and edge sets satisfying the obvious compatibility conditions and preserving the depths of vertices. When $\Lambda^{-} = \emptyset$ so that $\LagGrouped$ is a Lagrangian filling, $\tree_{\Aug}(\Lambda^{-}, \Partition^{-})$ is just a single root vertex and $\tree_{\Aug}\LagGrouped$ sends this vertex to the vertex of $\tree_{\Aug}(\Leg^{+}, \Partition)$ corresponding to the augmentation of $PDA(\Leg^{+}, \Partition^{+})$ induced by $\LagGrouped$.

\subsubsection{Comparison with $\CE$ augmentations}

We'll see in \S \ref{Sec:Computations} that $\newRSFTA\LegGrouped$ often does not have augmentations even when $\Leg$ can not be destabilized. On the other hand, the following lemma shows how partitions can be used to help find augmentations of $\CE$ algebras of disconnected Legendrians. The following is a restatement of \cite[Proposition 2.1]{BC:Bilinearized} whose proof we include for completeness.

\begin{lemma}\label{Lemma:CEAugCombo}
Let $\Partition$ be a partition of $\Leg \subset \Mxi$ such that each $\CE(\Leg^{\Partition}_{j})$ has an augmentation $\aug^{\Partition}_{j}$. Then the $\aug^{\Partition}_{j}$ can naturally be combined to define an augmentation
\begin{equation*}
\aug: \CE(\Leg) \rightarrow \field, \quad \aug(\chord) = \begin{cases}
\aug^{\Partition}_{j}(\chord) & \chord \in \Chords^{\Partition}_{j, j} \text{ for some } j,\\
0 & \text{otherwise}.
\end{cases}
\end{equation*}
\end{lemma}

\begin{proof}
Using the Leibniz rule, it suffices to establish $\aug\partial_{\CE} = 0$ when restricted to the vector space of individual chords. Restricting to this space, write $\partial_{\CE} = \partial_{m} + \partial_{p}$ where $\partial_{p}$ counts holomorphic disks all of whose negative chords $\Partition$ pure and $\partial_{m}$ counts disks with at least one $\Partition$ mixed negative chord. By definition, $\aug \partial_{m} = 0$ as $\aug$ annihilates $\Partition$ mixed chords. Likewise, $\partial_{p}$ preserves the subalgebra $\CE(\Leg_{j}) \subset \CE(\Leg)$ of $\Partition$ pure chords which begin and end on a given $\Leg^{\Partition}_{j}$. Therefore
\begin{equation*}
\aug\partial_{p}|_{\CE_{j}} = \aug^{\Partition}_{j}\partial_{\CE_{j}} = 0.
\end{equation*}
\end{proof}

\subsection{Bilinearization and spectral sequences}

We briefly review the bilinearized homology and augmentation category constructions of \cite{BC:Bilinearized, CEKSW:AugCat} which enhance the linearization construction of \cite{Chekanov:LCH}. There are a few ways to carry out the constructions and we do so in a way which we feel is most straightforward computationally. To simplify we'll set $\ring = \Z/2\Z$ so that we don't have to worry about signs.

\subsubsection{Review of the constructions}

For a free DGA $(\Algebra = \tensorAlg(\orbitVS), \partial)$ we can break up the differential $\partial|_{\orbitVS}$ into a collection of $\field$-linear maps
\begin{equation*}
\partial|_{\orbitVS} = \sum_{0}^{\infty} \partial_{k}, \quad \partial_{k}: \orbitVS \rightarrow \orbitVS^{\otimes k}.
\end{equation*}
Then $\partial^{2}=0$ is equivalent to $(\orbitVS, \partial_{k})$ being a curved $A_{\infty}$ coalgebra structure. Here ``curved'' means $\partial_{0}$ is not necessarily $0$ and ``$A_{\infty}$ coalgebra'' means that we have the $A_{\infty}$ algebra relations ``upside down'',
\begin{equation*}
0 = \sum_{i + j - 1 = k,\ a + b +1 = i} \left(1^{\otimes b} \otimes \partial_{j}\otimes 1^{\otimes a} \right)\circ  \partial_{i}: \orbitVS \rightarrow \orbitVS^{\otimes k}.
\end{equation*}
When $\Algebra = \CE$ or $\filtration^{\filtLevel}\newRSFTA$, energy bounds imply that only some finite number of the $\partial_{k}$ are non-zero.

Let $\aug_{l}, \aug_{r}$ be an ordered pair of augmentations of $\Algebra$. Define $\field$-linear maps
\begin{equation*}
\pi^{\aug_{l}, \aug_{r}}_{k}: \orbitVS^{\otimes k} \rightarrow \orbitVS, \quad \pi_{0}^{\aug_{l}, \aug_{r}} = 0, \quad \pi_{k > 0 }^{\aug_{l}, \aug_{r}} = \sum_{i + j = k - 1} \aug_{l}^{\otimes i} \otimes 1 \otimes \aug_{r}^{\otimes j}.
\end{equation*}
Composing these maps we get a $\field$-linear differential operator
\begin{equation*}
\partial^{\aug_{l}, \aug_{r}} = \sum \pi_{k}^{\aug_{l}, \aug_{r}}\partial_{k}: \orbitVS_{\ast} \rightarrow \orbitVS_{\ast-1}, \quad (\partial^{\aug_{l}, \aug_{r}})^{2} = 0.
\end{equation*}

\begin{defn}
$(\orbitVS, \partial^{\aug_{l}, \aug_{r}})$ is the \emph{bilinearized chain complex} whose homology $H^{\aug_{l}, \aug_{r}}_{\ast} = H_{\ast}(\orbitVS, \partial^{\aug_{l}, \aug_{r}})$ is the \emph{bilinearized homology}.
\end{defn}

Define a chain map $\counit^{\aug_{l}, \aug_{r}}$ from the bilinearized chain complex to $\field$,
\begin{equation*}
\begin{gathered}
\counit^{\aug_{l}, \aug_{r}} = \aug_{l} - \aug_{r}: \orbitVS \rightarrow \field,\\
\counit^{\aug_{l}, \aug_{r}} \partial^{\aug_{l}, \aug_{r}} = \sum_{k}\left(\sum_{i + j = k - 1} \aug_{l}^{\otimes i+1} \otimes \aug_{r}^{\otimes j} - \aug_{l}^{\otimes i} \otimes \aug_{r}^{\otimes j + 1} \right)\partial_{k} = 0,
\end{gathered}
\end{equation*}
yielding a linear map on homology which vanishes in non-zero homological degree
\begin{equation*}
\Counit^{\aug_{l}, \aug_{r}} = [\counit^{\aug_{l}, \aug_{r}}]: H^{\aug_{l}, \aug_{r}}_{\ast} \rightarrow \field.
\end{equation*}

\begin{defn}
$\Counit^{\aug_{l}, \aug_{r}}$ is the \emph{fundamental class on bilinearized homology}.
\end{defn}

\begin{rmk}
The name ``fundamental class'' is justified by the fact that when we are working with the bilinearized complex associated to the $\CE$ algebra of a connected $\Leg$, $\Counit^{\aug_{l}, \aug_{r}}$ is the map
\begin{equation*}
LCH^{\aug_{l}, \aug_{r}}_{0} \rightarrow H_{0}(\Leg) \simeq \field
\end{equation*}
in the duality exact sequence of \cite{EES:Duality}. See the proof of \cite[Proposition 3.2]{BG:Bilinearized}.
\end{rmk}

The pair $(H^{\aug_{l}, \aug_{r}}_{\ast}, \Counit^{\aug_{l}, \aug_{r}})$ depends only on the homotopy classes of $\aug_{l}, \aug_{r}$ \cite{BC:Bilinearized}. The following is simply a restatement of \cite[Lemma 3.1]{BG:Bilinearized} and gives a practical tool for distinguishing homotopy classes of augmentations. Its proof is purely algebraic, and so is applicable to any free DGA, such as the $\filtration^{\filtLevel}\newRSFTA$.

\begin{thm}
The augmentations $\aug_{l}, \aug_{r}$ are homotopic if and only if $\Counit^{\aug_{l}, \aug_{r}} = 0$.
\end{thm}

With a single augmentation $\aug = \aug_{l} = \aug_{r}$, the bilinearized homology construction gives rise to the \emph{linearized chain complex} and \emph{linearized homology},
\begin{equation*}
(\orbitVS, \partial^{\aug} = \partial^{\aug, \aug}), \quad H^{\aug}_{\ast} = H_{\ast}(\orbitVS, \partial^{\aug}).
\end{equation*}
The linearized homology can also be defined for a commutative free DGA of the form $(\tensorAlgCom(\orbitVS), \partial)$ in which case we replace the $\pi^{\aug_{0}, \aug_{1}}_{k > 0}$ with
\begin{equation*}
\pi^{\aug}_{k}(v_{1} \wedge \cdots \wedge v_{k}) = \sum_{i = 1}^{k} \left(\prod_{j \neq i} \aug(v_{j})\right)v_{i}.
\end{equation*}

To define bilinearized cochain complexes from a pair of augmentations $\aug_{l}, \aug_{r}$, define a differential $d_{\aug_{l}, \aug_{r}}$ on $\orbitVS$ whose structure coefficients $\langle d_{\aug_{l}, \aug_{r}}v, w \rangle$ are defined by adjunction with $\partial^{\aug_{l}, \aug_{r}}$,
\begin{equation*}
d_{\aug_{l}, \aug_{r}}: \orbitVS_{\ast} \rightarrow \orbitVS_{\ast + 1}, \quad \langle d_{\aug_{l}, \aug_{r}}v, w \rangle = \langle v, \partial^{\aug_{l}, \aug_{r}}w \rangle.
\end{equation*}

\begin{defn}
$(\orbitVS, d_{\aug_{l}, \aug_{r}})$ is the \emph{bilinearized cochain complex} whose cohomology $H_{\aug_{l}, \aug_{r}}^{\ast} = H^{\ast}(\orbitVS, d_{\aug_{l}, \aug_{r}})$ is the \emph{bilinearized cohomology}.
\end{defn}

When applied to $\CE$, the (bi)linearized homology and cohomology groups are called the \emph{(bi)linearized Legendrian contact homology} and \emph{(bi)linearized Legendrian contact cohomology}, denoted $LCH^{\aug_{l}, \aug_{r}}_{\ast}(\Leg)$ and $LCH_{\aug_{l}, \aug_{r}}^{\ast}(\Leg)$, respectively.

\subsubsection{Enhancements for mfDGAs and $\filtration^{\filtLevel}\newRSFTA$}

For mfDGAs, the bilinearized chain complex is a filtered chain complex: As $\partial: V \rightarrow \Algebra$ is filtration preserving, then so is $\partial^{\aug_{l}, \aug_{r}}$. Of course there are a few options to organize the filtration data homologically.

First, filtration inclusions induce chain maps determining morphisms on bilinearized homologies,
\begin{equation*}
\begin{gathered}
\cdots \rightarrow (\filtration^{\filtLevel}\orbitVS, \partial^{\aug_{l}, \aug_{r}}) \rightarrow (\filtration^{\filtLevel + 1}\orbitVS, \partial^{\aug_{l}, \aug_{r}}) \rightarrow \cdots,\\
\cdots \rightarrow H_{\ast}(\filtration^{\filtLevel}\orbitVS, \partial^{\aug_{l}, \aug_{r}}) \rightarrow H_{\ast}(\filtration^{\filtLevel + 1}\orbitVS, \partial^{\aug_{l}, \aug_{r}}) \rightarrow \cdots.
\end{gathered}
\end{equation*}
As is evident from the chain level, fundamental classes are pulled back from inclusion morphisms
\begin{equation*}
\begin{tikzcd}
H_{\ast}(\filtration^{\filtLevel}\orbitVS, \partial^{\aug_{l}, \aug_{r}}) \arrow[r]\arrow[dr, "\Counit^{\aug_{l}, \aug_{r}}"] & H_{\ast}(\filtration^{\filtLevel + 1}\orbitVS, \partial^{\aug_{l}, \aug_{r}}) \arrow[d, "\Counit^{\aug_{l}, \aug_{r}}"]\\
& \field,
\end{tikzcd}
\end{equation*}

Second, we can consider the spectral sequence homology groups associated to the filtered complex,
\begin{equation*}
E^{r}_{p, \deg}(\aug_{l}, \aug_{r}) = E^{r}_{p, \deg}(\orbitVS, \partial^{\aug_{l}, \aug_{r}}, \filtration).
\end{equation*}
Our conventions for homological spectral sequences associated to ascending filtrations are that on the $r$th page $E^{r}$, differentials go
\begin{equation}\label{Eq:SSDiffConvention}
E^{r}_{p, \deg} \rightarrow E^{r}_{p - r, \deg - 1}
\end{equation}
with $p$ denoting filtration level and $\deg$ indicating homological degree. For cohomological spectral sequences associated to descending filtrations, our convention is that differentials on the $r$th page go
\begin{equation*}
E_{r}^{p, \deg} \rightarrow E_{r}^{p + r, \deg + 1}.
\end{equation*}

\begin{lemma}
The spectral sequence homology groups $E^{r}_{p, \deg}(\aug_{l}, \aug_{r}), r \geq 1$ are homotopy invariants of the pair $\aug_{l}, \aug_{r}$. The set of all collections of spectral sequence homology groups is a filtered stable tame isomorphism invariant of $(\Algebra, \partial, \filtration)$.
\end{lemma}

\begin{proof}
These follows from the fact that that if a chain map between filtered complexes induces isomorphisms on $E^{1}$ then it induces isomorphisms on all $E^{r}$. See for example \cite[\S 4]{Baldwin:KHFSS}. If we modify $\aug_{r}$ by a chain homotopy, we have induced chain homotopy equivalence between filtered complexes $(\orbitVS, \partial^{\aug_{l}, \aug_{r}}, \filtration)$ and $(\orbitVS, \partial^{\aug_{l}, \aug_{r}'}, \filtration)$ which induce $E^{1}$ isomorphisms. The same goes for modification of $\aug_{l}$.

As for stable tame isomorphism invariance, we can apply the homotopy equivalence between an algebra and its stabilization pulling back augmentations to $\Algebra$ via the inclusion $\Algebra \rightarrow \Stab_{f, \deg}\Algebra$. The associated map on the linearized complex clearly induces an $E^{1}$ isomorphism as the differential on $\Stab_{f, \deg}\Algebra$ is linear when restricted to new generators.
\end{proof}

Using the $E^{r}_{p, \deg}(\aug_{l}, \aug_{r})$ where the $\aug_{l}, \aug_{r}$ are augmentations of $\filtration^{\filtLevel}\newRSFTA$ we define a polynomial invariant of the homotopy classes of the $\aug_{l}, \aug_{r}$,
\begin{equation}\label{Eq:SpecPolyDef}
P^{\spec}_{\aug_{l}, \aug_{r}}(t, x, y) = \sum_{r = 1}^{\max(\filtLevel, N^{\Partition})} \dim E^{r}_{p, \deg}(\aug_{l}, \aug_{r})t^{\deg}x^{r - 1}y^{p}.
\end{equation}
The ``$r - 1$'' ensures the polynomial isn't guaranteed to be divisible by $x$.

As an immediate consequence of the above lemma and Theorem \ref{Thm:Main} we have the following theorem.

\begin{thm}
Let $\aug_{l}, \aug_{r}$ be augmentations of $\filtration^{\filtLevel}\newRSFTA\LegGrouped$. Then the isomorphism class of the bilinearized spectral sequence homology groups $E^{r}_{p, q}(\aug_{l}, \aug_{r}), r \geq 1$ -- and so the $P^{\spec}_{\aug_{l}, \aug_{r}}$ -- associated to the filtered complex $(\filtration^{\filtLevel}\orbitVS, \partial, \filtration)$ are homotopy invariants of the pair $\aug_{l}, \aug_{r}$. The set of all such collections of spectral sequences -- over varying pairs $\aug_{l},\aug_{r}$ -- is a Legendrian isotopy invariant of the partitioned Legendrian $\LegGrouped$.

\end{thm}

We can continue in the same fashion describing analogues of $\CE$ structures for the $\filtration^{\filtLevel}\newRSFTA$ and stating theorems which come for free as a consequence of Theorem \ref{Thm:Main} combined with the results of \cite{BC:Bilinearized} and standard homological algebra. Here are some examples:
\be
\item Bilinearized cochain groups and spectral sequences are also available with the caveat that the differentials no longer preserve the ascending filtration. We can instead equip $\orbitVS$ with an descending filtration $\filtration_{op}$ with each $\filtration^{\filtLevel}_{op}$ additively generated by words with length $\geq \filtLevel$. Let's call $E_{r}^{p, q}$ the associated cohomological spectral sequence with associated Poincar\'{e} polynomial,
\begin{equation*}
P_{\spec}^{\aug_{l}, \aug_{r}}(t, x, y) = \sum_{r = 1}^{\max(\filtLevel, N^{\Partition})} \dim E_{r}^{p, \deg}(\aug_{l}, \aug_{r})t^{\deg}x^{r - 1}y^{p}.
\end{equation*}
\item The bilinearized cochain complexes give the set of all augmentations of $\filtration^{\filtLevel}\newRSFTA$ the structure of an $A_{\infty}$ category whose \emph{pseudo-equivalence class} \cite{BC:Bilinearized} is a Legendrian isotopy invariant.
\item With a single augmentation of some $\filtration^{\filtLevel}\newRSFTA^{\com}$ or $\filtration^{\filtLevel}\newRSFTA^{\cyc}$ we have linearized chain complexes and cochain complexes with obvious analogues of the above theorem.
\item For a single augmentation $\aug$, the $A_{\infty}$ structure makes the linearized cochain complex $(\orbitVS, d_{\aug}, \filtration_{op})$ an $A_{\infty}$ algebra whose products $\mu_{k}: \orbitVS^{\otimes k} \rightarrow \orbitVS$ for $k \geq 1$ preserve the ascending filtration $\filtration_{op}$.
\item Cochain complexes are -- according to our convention that Lagrangian cobordisms go downward in a symplectization -- contravariantly functorial with respect to cobordism.
\item If we shift the grading by $1$, the linearized cohomology group $H_{\aug}^{\ast}[1]$ becomes a ring and the $E_{r}^{p, q}[1]$ determine a multiplicative spectral sequence using the product $\mu_{2}$.
\ee

In particular, the last item applied to $\newRSFTA^{cyc}$ gives the RSFT spectral sequence of \cite{Ekholm:Z2RSFT} a multiplicative structure after a grading shift has been applied.

\subsubsection{Planar diagrams for bilinearized disk counts}

To gain some intuition as to how these invariants work, let's draw planar diagrams for the disks counted by the differentials for the $\newRSFTA$ (bi)linearized (co)chain complexes and spectral sequences.

\begin{figure}[h]
\begin{overpic}[scale=.5]{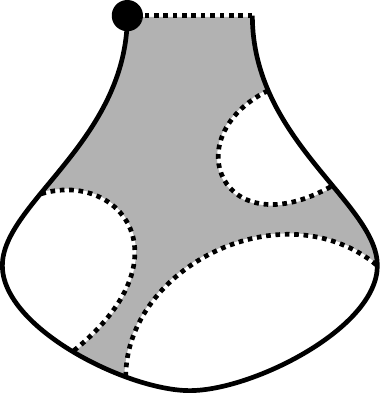}
\put(12, 28){$\aug_{l}$}
\put(62, 59){$\aug_{r}$}
\put(58, 15){$\word^{-}$}
\end{overpic}
\caption{Disks contributing to $\partial^{\aug_{l}, \aug_{r}}$ in linearized $LCH$.}
\label{Fig:LCHOps}
\end{figure}

For (bi)linearized Legendrian contact homology, we count holomorphic disks with one positive puncture and $m_{-} \geq 1$ negative punctures, capping off all but one of them with augmentations. Figure \ref{Fig:LCHOps} displays the usual picture (cf. \cite[Figure 1]{BC:Bilinearized}) in the planar diagram style of this article.

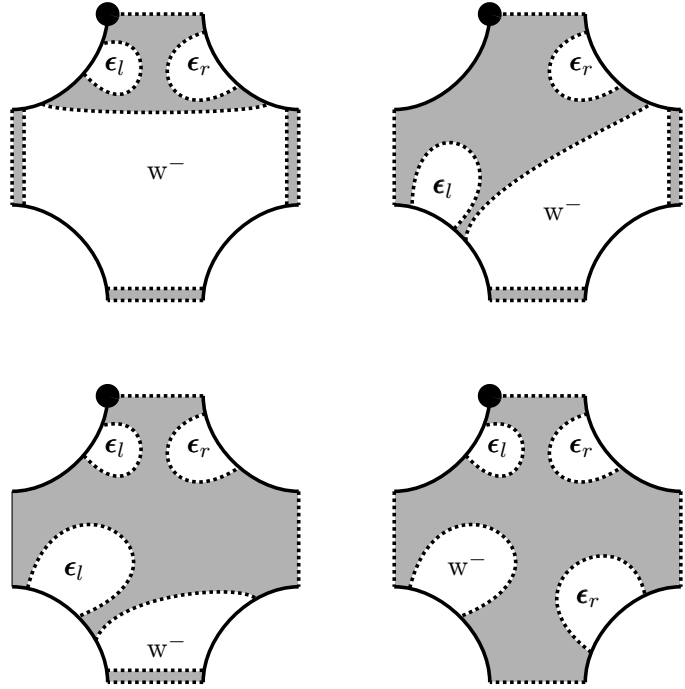
\begin{figure}[h]
\begin{overpic}[scale=.6]{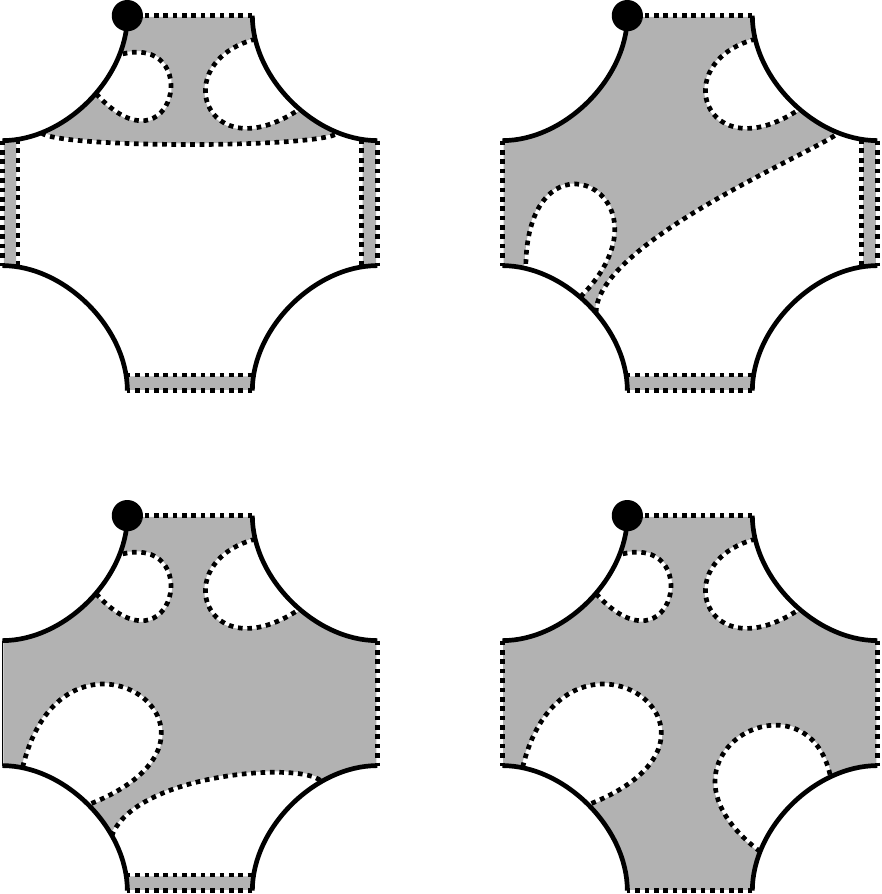}
\put(14, 90){$\aug_{l}$}
\put(26, 90){$\aug_{r}$}
\put(20, 74){$\word^{-}$}

\put(62, 72){$\aug_{l}$}
\put(82, 90){$\aug_{r}$}
\put(78, 68){$\word^{-}$}

\put(14, 34){$\aug_{l}$}
\put(8, 16){$\aug_{l}$}
\put(26, 34){$\aug_{r}$}
\put(20, 4){$\word^{-}$}

\put(70, 34){$\aug_{l}$}
\put(82, 34){$\aug_{r}$}
\put(83, 12){$\aug_{r}$}
\put(64, 16){$\word^{-}$}
\end{overpic}
\caption{Planar diagrams for contributions to the bilinearized homology spectral sequence differentials of a word $\word$ of length $4$. From left-to-right, top-to-bottom, we see contributions to differentials on the $E^{0}, E^{1}, E^{2}$, and  $E^{3}$ pages. Here trivial strips appear near the positive chords in $\word$ which are not touched by the $\ind = 1$ admissible disk.}
\label{Fig:LinearOps}
\end{figure}

For $\filtration^{\filtLevel}\newRSFTA$, the differentials on the $E^{0}$ page of a (bi)linearized spectral sequence count disks for which the inputs and outputs of the operators have the same word length, with some numbers of additional outputs assigned augmentations. For the differentials on the $E^{k > 0}$ page of the homological spectral sequence, the disks we count will have one input $\word$, one output $\word^{-}$ for which
\begin{equation*}
\filtLevel(\word) - \filtLevel(\word^{-}) = k,
\end{equation*}
and some additional $\word^{-}_{i}$ which are augmented. For the cohomological spectral sequence, we view $\word$ as an output and $\word^{-}$ as an input. See Figure \ref{Fig:LinearOps}.

\subsubsection{Convergence of spectral sequences for links in $\Rthree$}

Now that we know what the spectral sequence differentials count, it's easy to see that for links in $\Rthree$ the $\newRSFTA$ bilinearized spectral sequences converge very quickly.

\begin{thm}
Let $\LegGrouped$ be a partitioned Legendrian in $\Rthree$ and let $\aug_{l}, \aug_{r}$ be augmentations of some $\filtration^{\filtLevel}\newRSFTA$. Then the associated (bi)linearized spectral sequences $E^{r}_{p, q}$ and $E^{p, q}_{r}$ converge at their $E^{2}$ and $E_{2}$ pages respectively.
\end{thm}

\begin{proof}
This is an immediate consequence of \cite{Avdek:LSFT}: After applying a Legendrian isotopy we can guarantee that all rigid holomorphic disks contributing to the $\newRSFTA$ differential have at most two positive punctures. This means that for this isotopy representative the differentials on $E^{k}$ and $E_{k}$ pages of the spectral sequences vanish for $k \geq 2$. The result then follows from the fact that the spectral sequence is an invariant of the Legendrian isotopy class of $\LegGrouped$.
\end{proof}

\section{Computations} \label{Sec:Computations}

Here we carry some basic computations. Prior to \S \ref{Sec:Cables}, we focus on low-dimensional examples. We start by describing general techniques for Legendrian links in $\Rthree$ which are implemented in the software package \cite{Avdek:Software}.

\subsection{Computations in $\Rthree$}

We assume the reader is familiar with the basics of combinatorially defined Legendrian contact homology for links in $\Rthree$ as in \cite{Chekanov:LCH, EtnyreNg:LCHSurvey} and will use a slightly modified setup following \cite{Ng:RSFT} (which uses two marked points rather than one marked point on each connected component of $\Leg$).

Everything here works just as well for Legendrians in $\R_{t}$ times a Riemann surface $\Sigma$ with non-empty boundary and Liouville form $\beta$ as in \cite{Bjorklund:LCH}. The only additional data that is needed in this more general case is a framing of $T\Sigma$ which can be determined by an immersion $\Sigma \rightarrow \C$.

Decorate each connected component $\Leg_{i}$ of $\Leg$ with an arrow and two points $\ast_{i}$ and $\bullet_{i}$ which are not the endpoints of any chord. The arrow indicates orientation, the $\ast_{i}$ will be our basepoint, and the $\bullet_{i}$ will help us record homology classes of curves.

If the $\Leg^{\Partition}_{i}$ are connected or if we're satisfied working with $\Z/2\Z$ gradings the basepoints can be ignored. Otherwise we must draw connecting arcs between the $\ast_{i}$ and extend the oriented Lagrangian subspaces $T_{\ast_{i}}\Leg$ over them as described in \S \ref{Sec:Gradings}.

To compute the Maslov number of a word $\word = (\chord_{1}\cdots\chord_{\filtLevel})$, we calculate the rotation angles over preferred capping paths (which miss the $\bullet_{i}$), performing CW rotations of the tangent spaces $T(\pi_{\SympBase}\Leg_{i})$ when traversing chords. In this dimension, CW rotations are simply clockwise rotations in the usual sense. Standard calculations show that the modulo $2$ we get
\begin{equation*}
\begin{aligned}
\Maslov_{2}(\word) &= \#(\text{negative crossing } \chord_{i} \in \word) \in \Z/2\Z,\\
|\word|_{2} &= 1 + \filtLevel + \#(\text{negative crossing } \chord_{i}) \\
&= 1 + \#(\text{positive crossing } \chord_{i}) \in \Z/2\Z.
\end{aligned}
\end{equation*}
Of course, if $\filtLevel = 1$, we get the $\CE$ grading by definition.

We assign a variable $T_{i} = e^{[\Leg_{i}]} \in \field[H_{1}(\Leg)]$ to each $\Leg_{i}$ where $[\Leg_{i}]$ is the fundamental class in homology with respect to the prescribed orientation. The grading of each $T_{i}$ is given by
\begin{equation*}
|T_{i}| = 2\rot(\Leg_{i}), \quad |T_{i}^{-1}| = -2\rot(\Leg_{i}).
\end{equation*}
We recall that it is our convention that the $T_{i}$ commute with the $\word$ generators of our algebra.

Differentials can be computed by counting immersions of polygons into $\C$ with boundary on $\pi_{\SympBase}\Leg$ as in \cite{Chekanov:LCH, EtnyreNg:LCHSurvey, Ng:RSFT}. To compute the $\field[H_{1}(\Leg)]$ element $h(u)$ associated to each polygon $u$ contributing to $\partial$, traverse the boundary of the polygon and multiply by $T_{i}$ (or $T^{-1}_{i}$) every time we cross over a $\bullet_{i}$ according to the positive (or negative) orientation of $\Leg$. To check admissibility of disks and compute their output words, we recommend having plenty of scrap paper on hand.

\subsection{The Hopf link}\label{Sec:HopfLink3d}

Let $\Leg_{1}$ be the standard $\tb = -1$ unknot in $\Rthree$ and let $\Leg_{2}$ be a pushoff with its orientation flipped, $\Leg_{2} = -\Flow^{\epsilon}_{\partial_{t}}(\Leg_{1})$ for $\epsilon$ small. Then $\Leg$ is a Hopf link which bounds an exact, oriented Lagrangian annulus in the symplectization of $\Rthree$ \cite{EHK:LagrangianCobordisms}. The Lagrangian filling implies the existence of an augmentation of $\aug_{\Lag}:LCH(\Leg) \rightarrow \field$ and so $LCH(\Leg)$ is non-zero. Likewise the individual $\Leg_{i}$ bound Lagrangian disks so that $LCH(\Leg_{i}) \neq 0$.

\begin{figure}[h]
\begin{overpic}[scale=.4]{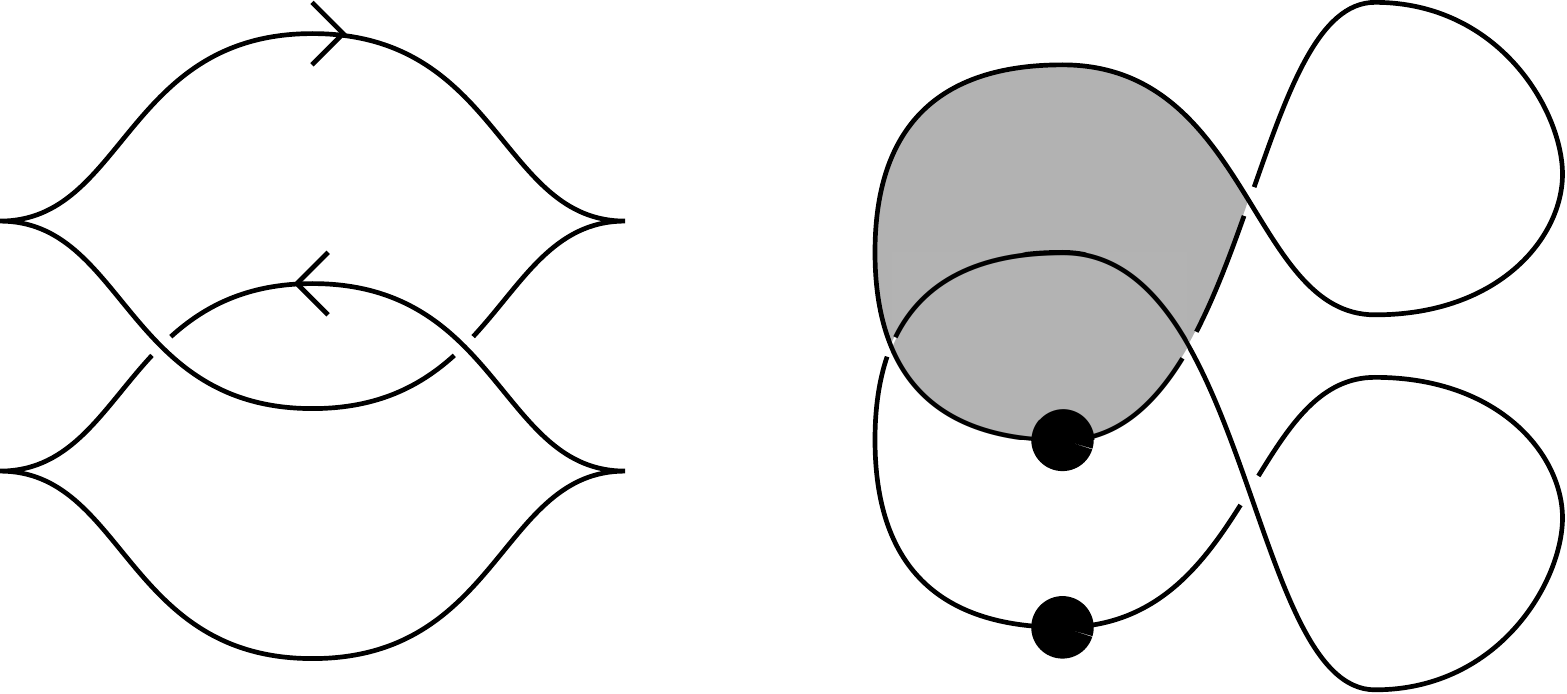}
\put(45, 29){$\Leg_{2}$}
\put(45, 13){$\Leg_{1}$}
\put(51, 20){$\chord_{1}$}
\put(78, 20){$\chord_{2}$}
\put(82, 30){$\chord_{3}$}
\put(82, 10){$\chord_{4}$}
\put(64, 32){$u^{3}_{1, 2}$}
\put(64, 21){$u^{1, 2}$}
\end{overpic}
\caption{A Legendrian Hopf link in the front (left) and Lagrangian projections (right) in the $xz$ and $xy$ planes, respectively. Projections of holomorphic disks $u^{3}_{1, 2}$ and $u^{1, 2}$ to $\R^{2}_{x, y}$ are shaded.}
\label{Fig:HopfDifferential}
\end{figure}
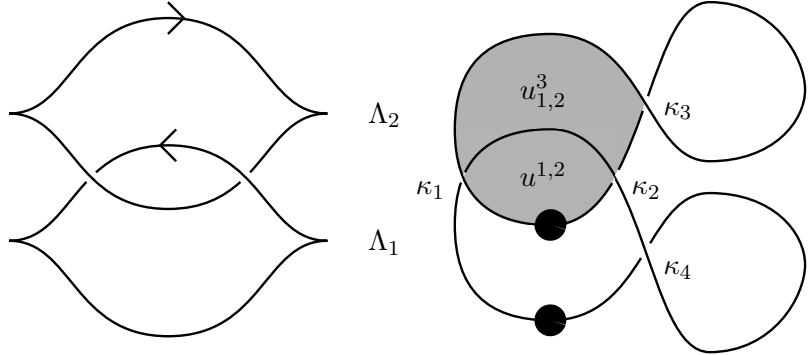

There are two ways to group $\Leg$: Either $N^{\Partition} = 1$ or $N^{\Partition} = 2$. In the first case,
\begin{equation*}
\newRSFT(\Leg, \Partition) = \filtration^{1}\newRSFT(\Leg, \Partition) = LCH(\Leg) \neq 0, \quad \tau_{H}\LegGrouped = \tau_{A}\LegGrouped = \infty.
\end{equation*}
Moving forward we assume $N^{\Partition} = 2$ so that $\Leg^{\Partition}_{i} = \Leg_{i}$.

Figure \ref{Fig:HopfDifferential} shows a Lagrangian resolution \cite{Ng:ComputableInvariants} of $\Leg$ with four chords: $\chord_{1}$ and $\chord_{2}$ are mixed (going from $\Leg_{1}$ to $\Leg_{2}$ and vice-versa, respectively) and $\chord_{3}, \chord_{4}$ are pure. Then
\begin{equation*}
\filtration^{1}\orbitVS = \langle(\chord_{3}), (\chord_{4})\rangle, \quad \filtration^{2}\orbitVS = \filtration^{1}\orbitVS \oplus R \langle(\chord_{1}\chord_{2}), (\chord_{2}\chord_{1})\rangle = \orbitVS.
\end{equation*}
where $R=\field[T^{\pm}_{1}, T^{\pm}_{2}]$. The differentials of the $(\chord_{3}), (\chord_{4})$ count only disks with one positive puncture so that
\begin{equation*}
\partial \chord_{3} = 1 + T_{2}^{-1}, \quad \partial \chord_{4} = 1 + T_{1} \implies T_{1} \simeq T_{2} = 1.
\end{equation*}
As described in Example \ref{Ex:LengthOne}, the disk $u^{3}_{1, 2}$ shown in the figure does not contribute a term to $\partial (\chord_{3})$ as it cannot be inscribed in $\planarDiagram((\chord_{3}))$. (More on this in a moment). With $R$ coefficients, we see that $PDA^{1}$ admits an augmentation $T_{1},T_{2} \mapsto 1$, $\chord_{3}, \chord_{4} \mapsto 0$ and so $PDH^{1}$ is non-zero. With $\field$ coefficients, the differential vanishes so that $PDH^{1}$ is freely generated by the $(\chord_{3}), (\chord_{4})$.

There is a disk $u^{1, 2}$ with two positive puncture asymptotic to $\chord_{1}$ and $\chord_{2}$, yielding
\begin{equation*}
\partial (\chord_{1}\chord_{2}) = \partial (\chord_{2}\chord_{1}) = T^{-1}_{2} \simeq 1.
\end{equation*}
We conclude that for this choice of grouping
\begin{equation*}
\newRSFT = \newRSFT^{2} = 0, \quad \tau_{H}(\Leg, \Partition) = 1.
\end{equation*}
For this $\Partition$, the augmentations of $\newRSFTA^{1}$ are in one-to-one correspondence with tensor products of augmentations for the $\Leg_{i}$. There is only one such $\aug$ sending $1 \mapsto 1$, $T_{i} \mapsto 1$, and $\chord_{3}, \chord_{4} \mapsto 0$.

Observe that if we glue $u^{3}_{1, 2}$ to $u^{1, 2}$ at $\chord_{1}$, we obtain a $\ind = 2$ disk $u = u^{1, 2} \# u^{3}_{1, 2}$ with four boundary punctures. The moduli space $\ModSpace_{\Leg}$ in which $u$ is contained is one dimensional and at one of its (compactified) ends we obtain a $2$-level SFT building with $u^{1, 2}$ on the bottom and $u^{3}_{1, 2}$ on the top. At the other end of $\ModSpace/\R_{s}$, the disk degenerates into a nodal curve consisting of a trivial strip over $\chord_{2}$ conjoined to a disk with one positive boundary puncture at $\chord_{3}$. This nodal degeneration is exactly the string topological breaking which we are seeking to avoid in defining $\newRSFTA$.

So how is it possible that $\newRSFT = 0$ if $\Leg$ has an exact Lagrangian filling? This is because the $N^{\Partition} = 2$ partition obstructs the existence of Lagrangian fillings $\Lag = \Lag_{1} \sqcup \Lag_{2}$ for which $\partial \Lag_{i} = \Leg_{i}$. Such $\Lag_{i}$ would have intersection number $\pm 1 = \lk(\Leg_{1}, \Leg_{2})$ so that $\Lag$ could not be embedded. In \S \ref{Sec:UnknotLinks} we'll generalize this computation to all dimensions.

\subsection{More $2$-component links in $\Rthree$}

Similar calculations show that $\newRSFTA$ obstructs disconnected fillings for two component links with $0$ linking number. For example, consider the two component link of unknots on the left-hand side of Figure \ref{Fig:TwistingExamples} with $N^{\Partition} = 2$. The $LCH$ of this link is non-zero but using this partition $\newRSFT\LegGrouped = 0$. So $\Leg$ has no disconnected filling.

\begin{figure}[h]
\begin{overpic}[scale=.4]{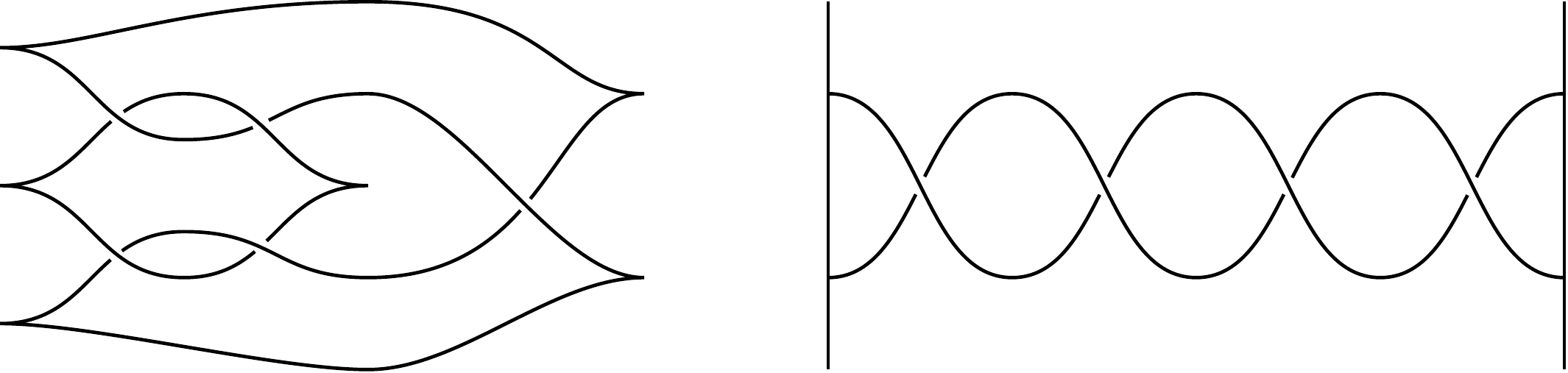}
\end{overpic}
\caption{On the left is a link of $\tb=-1$ unknots with $0$ linking number. On the right, a front satellite pattern for $(2k, 2)$-cabling (with respect to the contact framing) in the case $k=2$ in the style of \cite{NT:Torus}.}
\label{Fig:TwistingExamples}
\end{figure}

Similarly, for any knot $\Leg$ we can consider the $(2k, 2)$ cable, $\Leg(2k, 2)$, shown in the front projection on the right-hand side of the figure with $N^{\Partition}=2$. Our convention is that cables are defined with respect to the contact framing. For any $k > 1$ there is a length $2$ word $\word$ satisfying $\partial \word = 1$ as can be seen by counting disks of low energy in a $1$-jet neighborhood of $\Leg$. We'll see an alternate proof which generalizes to high dimensions in Proposition \ref{Prop:CableFIllingObs}.

The case $k = 1$ is more subtle. When $\Leg$ is the $\tb = -1$ unknot, then the $(2, 2)$ cable is an unlink. For $\Leg$ a Lagrangian slice knot, there is a Lagrangian concordance from the $(2, 2)$ cable to the $(2, 2)$ cable of the unknot. Therefore $\Leg(2, 2)$ is fillable by a pair of disjoint Lagrangian slice disks. See \cite[Proposition 1.8]{CET:PlusOneFilling}.

\subsection{Unbounded torsion for links in $3$-manifolds}

We do not currently have examples of links in $\Rthree$ with torsions other than $1$ or $\infty$. Figure \ref{Fig:HighTorsion} sketches the construction of Legendrians with
\begin{equation*}
\tau_{H}\LegGrouped = \tau_{A}\LegGrouped = k
\end{equation*}
for all $k \geq 2$ in contact $3$-manifolds of the form $\R_{t} \times \Sigma$ for $\Sigma$ a general Riemann surface.

\begin{figure}[h]
\begin{overpic}[scale=.7]{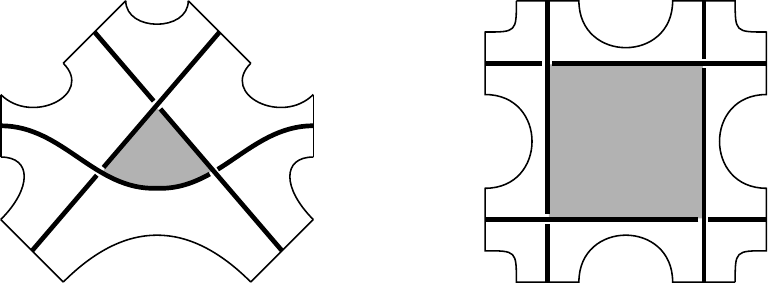}
\end{overpic}
\caption{Links $\Leg$ in $\R_{t}\times \Sigma$ in the cases $k=2, 3$. The $(k+1)$-gon $U$ is shaded.}
\label{Fig:HighTorsion}
\end{figure}

Here are the details: Consider $k + 1 \geq 3$ line segments $\leg_{i}$ in $\C$, with each $\leg_{i}$ intersecting $\leg_{i+1}$ transversely in a single point. Here indices are taken modulo $k+1$ and the $\leg_{i}$ form the smooth boundary arcs of a $(k+1)$-gon, $U$. Create a $4(k + 1)$-gon $\widetilde{\Sigma} \subset \C$ by taking a neighborhood of the union of the $\leg_{i}$ with $U$ and straightening the edges near the boundary points of the $\leg_{i}$. Now make a Riemann surface $\Sigma$ with $\chi(\Sigma) = 1 - 2(k + 1)$ by identifying the pairs of straightened edges of $\partial\widetilde{\Sigma}$ lying at the endpoints of each $\leg_{i}$. Then the $\leg_{i}$ form closed curves in $\Sigma$ with $\leg_{i} \cap \leg_{j}$ being a single point for each pair $i \neq j$. Equip $\Sigma$ with a Liouville form $\beta$ and find $\Leg_{i} \subset \R_{t} \times \Sigma$ whose projections to $\Sigma$ agree with the $\leg_{i}$ along the boundary of $U \subset \Sigma$ and such that the crossings of the $\Leg_{i}$ are alternating along the boundary of $U$ as shown in Figure \ref{Fig:HighTorsion} in the cases $k=2, 3$. We partition $\Leg$ so that all of the $\Leg_{i}$ are distinct, $N^{\Partition} = k+1$.

For $\filtLevel \leq k$, $\filtration^{\filtLevel}\orbitVS = 0$ so that $\filtration^{\filtLevel}\newRSFTA = \field$ and has only the trivial augmentation $1 \mapsto 1$. There is only one holomorphic disk contributing to $\partial$ which has $k + 1$ positive punctures, no negative punctures, and makes the unit in $\filtration^{k+1}\newRSFTA$ exact. This disk projects to the region $U$ in $\Sigma$. It follows that
\begin{equation*}
\newRSFT^{\filtLevel}\LegGrouped \simeq \begin{cases}
\field & \filtLevel \leq k,\\
0 & \filtLevel > k.
\end{cases}
\end{equation*}

It would be interesting to have examples of links for which the H-torsion and A-torsions disagree. It seems that such examples are difficult to produce. In \cite{Sivek:NoAugs}, Sivek describes a Legendrian knot $K_{2}$ whose $\CE$ algebra has no augmentations but has non-vanishing homology. In the language of this article,
\begin{equation*}
\tau_{H}(K_{2}) = \infty, \quad \tau_{A}(K_{2}) = 1.
\end{equation*}

\subsection{A polyfillable link}\label{Sec:Polyfilling}

In previous examples, we only needed to consider the $\newRSFTA$ differentials of a single element of $\Algebra$. We now consider a slightly more complicated example, carrying out a detailed analysis of $\newRSFTA$ invariants of a two-component link $\Leg = \Leg_{1} \sqcup \Leg_{2}$ where $\Leg_{1}$ is $\tb=1$ trefoil and $\Leg_{2}$ is a $\tb = -1$ unknot. See Figure \ref{Fig:Polyfillable}. We can group the $\Leg_{i}$ in two ways, using either
\be
\item $\Partition_{1}$ for which there is a single $\Leg^{\Partition}$
\item $\Partition_{2}$ for which $\Leg^{\Partition}_{i} = \Leg_{i}$ for $i=1, 2$.
\ee

\begin{figure}[h]
\begin{overpic}[scale=.5]{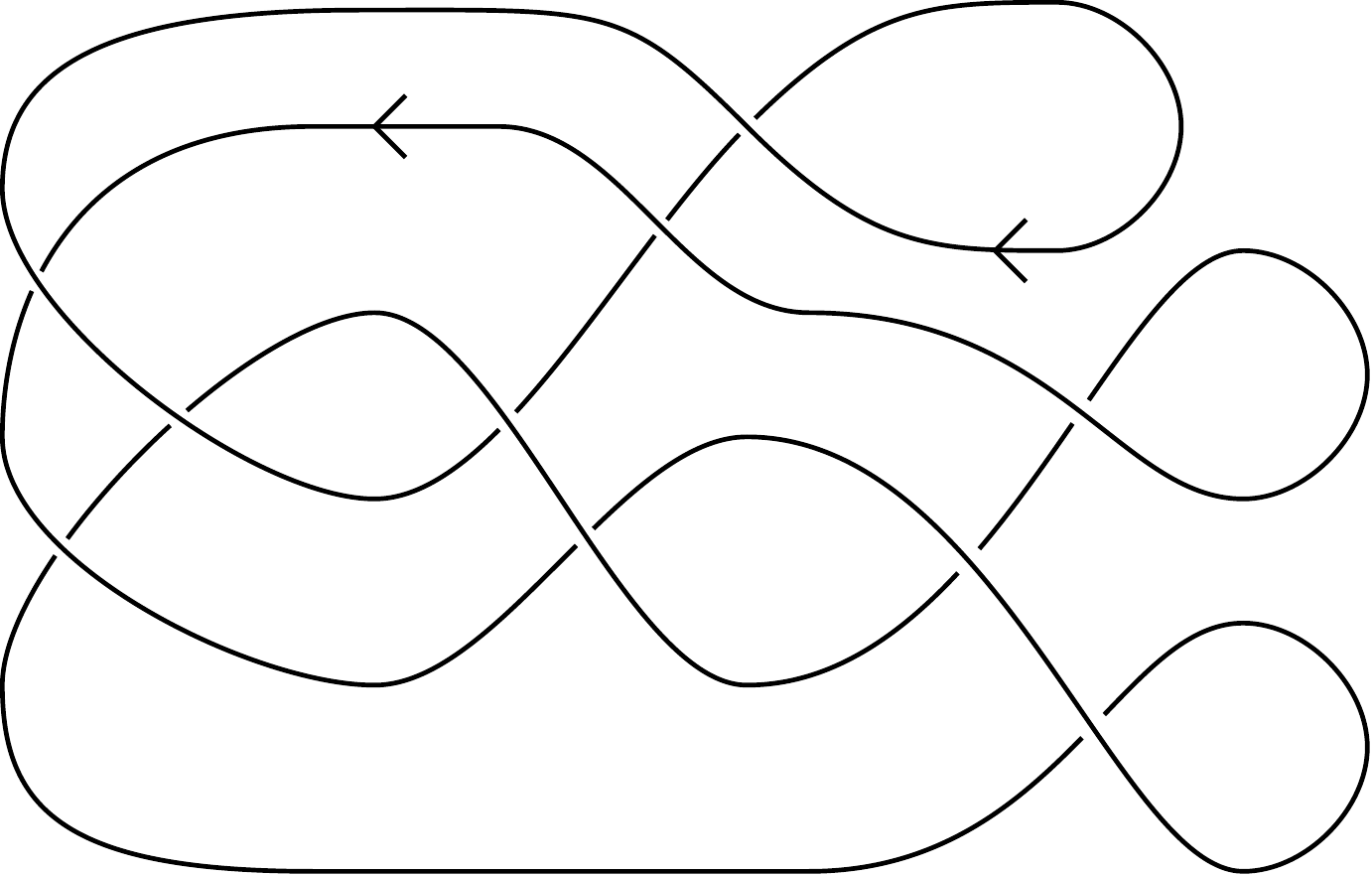}
\put(89, 54){$\Leg_{2}$}
\put(102, 35){$\Leg_{1}$}
\put(-4, 41){$\chord^{1,2}_{1}$}
\put(-3, 22){$\chord^{1, 1}_{1}$}
\put(4, 32){$\chord^{1, 2}_{2}$}
\put(30, 32){$\chord^{2, 1}_{1}$}
\put(35, 22){$\chord^{1, 1}_{2}$}
\put(62, 22){$\chord^{1, 1}_{3}$}
\put(71, 32){$\chord^{1, 1}_{4}$}
\put(71, 9){$\chord^{1, 1}_{5}$}
\put(40, 45){$\chord^{2, 1}_{2}$}
\put(46, 52){$\chord^{2, 2}_{1}$}
\end{overpic}
\caption{The Lagrangian projection of $\Leg$.}
\label{Fig:Polyfillable}
\end{figure}

As described in \cite{CGKS:Polyfilling}, $\Leg$ has topologically distinct Lagrangian fillings. One filling is an annulus. The other is a disjoint union of a genus one surface bounding with one boundary component $\Leg_{1}$ with a disk bounding $\Leg_{2}$. The connected filling can be partitioned only according to $\Partition_{1}$ and the disconnected filling can be partitioned either by $\Partition_{1}$ or $\Partition_{2}$. Since both groupings of $\Leg$ admit Lagrangian fillings, the $\newRSFTA\LegGrouped$ admit augmentations for either choice of $\Partition$. So both of the $\newRSFT(\Leg, \Partition_{i})$ are non-zero and have infinite torsion.

Using $\Partition_{1}$, the algebra $(\Algebra, \partial) = (\filtration^{1}\Algebra, \partial)$ is the Chekanov-Eliashberg algebra of $\Leg$. We now describe the $\newRSFTA$ algebra associated to $\Partition_{2}$. Using $\Partition_{2}$ and the orientation of $\Leg$, $\newRSFTA$ is naturally equipped with a $\Z$-valued degree grading (by the fact that $\rot(\Leg_{1}) = \rot(\Leg_{2}) = 0$) which we use throughout. For simplicity, we will ignore $T$ variables.

\subsubsection{Generators and relations}\label{Sec:SSGensAndDiffs}

Each chord labeled $\chord^{i, j}_{k}$ in Figure \ref{Fig:Polyfillable} will start on $\Leg_{i}$ and end on $\Leg_{j}$. The $10$ chords of $\Leg$ determine $14$ generators of $\Algebra$ for $\Partition_{2}$. Their word-lengths and $\Z$ gradings are as follows:

{\renewcommand{\arraystretch}{2}
\begin{center}
\setlength{\tabcolsep}{5pt}
\begin{tabular}{ |c||c|c|c| } 
 \hline
& $|\word| = 0$ & $|\word| = 1$ & $|\word| = 2$\\ \hline \hline
$\filtLevel(\word) = 1$ & $(\chord^{1, 1}_{1}), (\chord^{1, 1}_{2}), (\chord^{1, 1}_{3})$ & $(\chord^{1, 1}_{4}), (\chord^{1, 1}_{5}), (\chord^{2, 2}_{1})$ & \\ \hline
$\filtLevel(\word) = 2$ & $(\chord^{1, 2}_{1}\chord^{2, 1}_{1}), (\chord^{2, 1}_{1}\chord^{1, 2}_{1})$ & $(\chord^{1, 2}_{1}\chord^{2, 1}_{2}), (\chord^{2, 1}_{2}\chord^{1, 2}_{1}), (\chord^{1, 2}_{2}\chord^{2, 1}_{1}), (\chord^{2, 1}_{1}\chord^{1, 2}_{2})$ & $(\chord^{1, 2}_{2}\chord^{2, 1}_{2}), (\chord^{2, 1}_{2}\chord^{1, 2}_{2})$ \\ 
 \hline
\end{tabular}
\end{center}

For the $\filtLevel(\word) = 1$ generators, differentials are computed using $\CE$ disks, yielding
\begin{equation*}
\begin{gathered}
\partial (\chord^{1, 1}_{1}) = \partial (\chord^{1, 1}_{2}) = \partial (\chord^{1, 1}_{3}) = \partial (\chord^{2, 2}_{1}) = 0,\\
\partial (\chord^{1, 1}_{4}) = 1 + (\chord^{1, 1}_{1}) +  (\chord^{1, 1}_{3}) + (\chord^{1, 1}_{1})(\chord^{1, 1}_{2}) (\chord^{1, 1}_{3}),\\
\partial (\chord^{1, 1}_{5}) = 1 + (\chord^{1, 1}_{1}) +  (\chord^{1, 1}_{3}) + (\chord^{1, 1}_{3})(\chord^{1, 1}_{2})(\chord^{1, 1}_{1}).
\end{gathered}
\end{equation*}

For $\filtLevel(\word) = 2$ generators, differentials are computed using disks with $1$ or $2$ positive punctures:
\begin{equation*}
\begin{gathered}
\partial (\chord^{1, 2}_{1}\chord^{2, 1}_{1}) = \partial (\chord^{2, 1}_{1}\chord^{1, 2}_{1}) = 0,\\
\partial(\chord^{1, 2}_{1}\chord^{2, 1}_{2}) = 1 + (\chord^{1, 2}_{1}\chord^{2, 1}_{1})(\chord^{1, 1}_{1}), \quad \partial(\chord^{2, 1}_{2}\chord^{1, 2}_{1}) = 1 + (\chord^{1, 1}_{1})(\chord^{2, 1}_{1}\chord^{1, 2}_{1}),\\
\partial (\chord^{1, 2}_{2}\chord^{2, 1}_{1}) = 1 + (\chord^{1, 2}_{1}\chord^{2, 1}_{1})(\chord^{1, 1}_{1}), \quad \partial (\chord^{2, 1}_{1}\chord^{1, 2}_{2}) = 1 + (\chord^{2, 1}_{1}\chord^{1, 2}_{1})(\chord^{1, 1}_{1}),\\
\partial (\chord^{1, 2}_{2}\chord^{2, 1}_{2}) = (\chord^{1, 2}_{1}\chord^{2, 1}_{2})(\chord^{1, 1}_{1}) + (\chord^{1, 2}_{2}\chord^{2, 1}_{1})(\chord^{1, 1}_{1}),\quad
\partial (\chord^{2, 1}_{2}\chord^{1, 2}_{2}) = (\chord^{2, 1}_{2}\chord^{1, 2}_{1})(\chord^{1, 1}_{1}) + (\chord^{1, 1}_{1})(\chord^{2, 1}_{1}\chord^{1, 2}_{2}).
\end{gathered}
\end{equation*}
The above equation is organized so that cyclic rotations of words appear in the same row. We see by inspection that $\partial$ is not invariant under cyclic rotation.

\subsubsection{Augmentations}

The DGA $(\filtration^{1}\Algebra, \partial)$ is generated by the words with $\filtLevel(\word) = 1$. This generating set is the union of the $LCH$ generating sets for $\Leg_{1}$ and $\Leg_{2}$. All of the degree $0$ generators are pure chords on $\Leg_{1}$. A standard computation shows that the Chekanov-Eliashberg algebra for the trefoil $\Leg_{1}$ has $5$ augmentations, which are known to come from distinct Lagrangian fillings \cite{BC:Bilinearized, EHK:LagrangianCobordisms}. They are given by the following table:
\begin{center}
\setlength{\tabcolsep}{5pt}
\begin{tabular}{ |c||c|c|c| } 
 \hline
& $(\chord^{1, 1}_{1})$ & $(\chord^{1, 1}_{2})$ & $(\chord^{1, 1}_{3})$ \\
 \hline \hline
$\aug^{1}_{1}$ & $0$ & $0$ & $1$\\
\hline
$\aug^{1}_{2}$ & $0$ & $1$ & $1$\\
\hline
$\aug^{1}_{3}$ & $1$ & $0$ & $0$\\
\hline
$\aug^{1}_{4}$ & $1$ & $1$ & $0$\\
\hline
$\aug^{1}_{5}$ & $1$ & $1$ & $1$\\
\hline
\end{tabular}
\end{center}
These augmentations yield exactly the augmentations for $(\filtration^{1}\Algebra, \partial)$ as $\Leg_{2}$ has no degree $0$ generators. Moreover since there are no chords of negative degree, all of the augmentations are homotopy-inequivalent. See Remark \ref{Rmk:AugHomotopy}.

From the above computations, we see that only the augmentations of $(\filtration^{1}\Algebra, \partial)$ satisfying $(\chord^{1, 1}_{1}) \mapsto 1$ can be upgraded to augmentations of $(\filtration^{2}\Algebra, \partial)$. These are $\aug^{1}_{3}$, $\aug^{1}_{4}$, and $\aug^{1}_{5}$. For each such $\aug^{1}_{i}$ there is exactly one corresponding augmentation $\aug^{2}_{i}$ of $(\filtration^{2}\Algebra, \partial)$ as described in below:
\begin{center}
\setlength{\tabcolsep}{5pt}
\begin{tabular}{ |c||c|c|c|c|c| } 
 \hline
& $(\chord^{1, 1}_{1})$ & $(\chord^{1, 1}_{2})$ & $(\chord^{1, 1}_{3})$ & $(\chord^{1, 2}_{1} \chord^{2, 1}_{1})$ & $(\chord^{2, 1}_{1}\chord^{1, 2}_{1})$\\
 \hline \hline
$\aug^{2}_{3}$ & $1$ & $0$ & $0$ & $1$ & $1$\\
\hline
$\aug^{2}_{4}$ & $1$ & $1$ & $0$ & $1$ & $1$\\
\hline
$\aug^{2}_{5}$ & $1$ & $1$ & $1$ & $1$ & $1$\\
\hline
\end{tabular}
\end{center}

Because the disconnected filling $\Lag$ of $\Leg$ induces an augmentation $\aug_{\Lag}$, at least one of the above augmentations must correspond to $\aug_{\Lag}$. Since each $\aug^{2}_{i}$ sends the $\filtLevel(\word) = 2, |\word| = 0$ generators to $1$, $\aug_{\Lag}$ must include non-zero counts of $\ind = 0$ holomorphic disks with two positive punctures (also known as bananas).

Because the generators of $\newRSFTA$ all have non-negative degree, distinct augmentations cannot be chain homotopic by Remark \ref{Rmk:AugHomotopy}. By comparing the values of the above augmentations on generators, we can compute the entire augmentation tree $\tree_{\Aug}$ for $(\Leg, \Partition)$ as follows:
\begin{equation*}
\begin{tikzcd}
& & 1 & & \\
\aug^{1}_{1}\arrow[rru] & \aug^{1}_{2} \arrow[ru] & \aug^{1}_{3} \arrow[u] & \aug^{1}_{4} \arrow[lu] & \aug^{1}_{5} \arrow[llu] \\
& & \aug^{2}_{3} \arrow[u] & \aug^{2}_{4}\arrow[u] & \aug^{2}_{5}\arrow[u]
\end{tikzcd}
\end{equation*}

Because all the $\aug^{1}_{i}$ are determined by the Lagrangian fillings of \cite{EHK:LagrangianCobordisms}, the following is immediate:

\begin{cor}
At most $3$ of the $5$ Lagrangian fillings of the trefoil described in \cite{EHK:LagrangianCobordisms} appear as components of disconnected fillings of $\Leg$.
\end{cor}

\subsubsection{Bilinearized homologies}

We can compute the bilinearized homologies associated to the pairs of augmentations $\aug^{1}_{i}, \aug^{1}_{j}$ of $(\filtration^{1}\Algebra, \partial)$ directly, yielding Poincar\'{e} polynomials $P_{\aug^{1}_{i}, \aug^{1}_{j}}(t)$ for each bilinearization,
\begin{equation}\label{Eq:BilinPoly}
P_{\aug^{1}_{i}, \aug^{1}_{j}}(t) = \begin{cases}
1 + t & i \neq j\\
2 + 2t & i = j.
\end{cases}
\end{equation}
These are the sums of the Poincar\'{e} polynomials of the bilinearized homologies of the $\Leg_{i}$ with the induced augmentations. This is expected as the bilinearized complex is a direct sum of the bilinearized complexes of the $\Leg_{i}$.

Now we work out a full computation of the bilinearized $PDH^{2}$ spectral sequence homology groups. For each of the $\aug^{2}_{i}, \aug^{2}_{j}$, we compute our bilinearized differentials as
\begin{equation*}
\begin{gathered}
\partial^{\aug^{2}_{i}, \aug^{2}_{j}} (\chord^{1, 1}_{k}) = \partial^{\aug^{1}_{i}, \aug^{1}_{j}} (\chord^{1, 1}_{k}), \quad \partial^{\aug^{2}_{i}, \aug^{2}_{j}} (\chord^{2, 2}_{1}) = \partial^{\aug^{1}_{i}, \aug^{1}_{j}} (\chord^{2, 2}_{1}) = 0\\
\partial^{\aug^{2}_{i}, \aug^{2}_{j}} (\chord^{1, 2}_{1}\chord^{2, 1}_{1}) = \partial^{\aug^{2}_{i}, \aug^{2}_{j}} (\chord^{2, 1}_{1}\chord^{1, 2}_{1}) = 0,\\
\partial^{\aug^{2}_{i}, \aug^{2}_{j}}(\chord^{1, 2}_{1}\chord^{2, 1}_{2}) = (\chord^{1, 2}_{1}\chord^{2, 1}_{1}) + (\chord^{1, 1}_{1}), \quad \partial^{\aug^{2}_{i}, \aug^{2}_{j}}(\chord^{2, 1}_{2}\chord^{1, 2}_{1}) = (\chord^{1, 1}_{1}) + (\chord^{2, 1}_{1}\chord^{1, 2}_{1}),\\
\partial^{\aug^{2}_{i}, \aug^{2}_{j}} (\chord^{1, 2}_{2}\chord^{2, 1}_{1}) = (\chord^{1, 2}_{1}\chord^{2, 1}_{1}) + (\chord^{1, 1}_{1}), \quad \partial^{\aug^{2}_{i}, \aug^{2}_{j}} (\chord^{2, 1}_{1}\chord^{1, 2}_{2}) = (\chord^{2, 1}_{1}\chord^{1, 2}_{1}) + (\chord^{1, 1}_{1}),\\
\partial^{\aug^{2}_{i}, \aug^{2}_{j}} (\chord^{1, 2}_{2}\chord^{2, 1}_{2}) = (\chord^{1, 2}_{1}\chord^{2, 1}_{2}) + (\chord^{1, 2}_{2}\chord^{2, 1}_{1}),\quad
\partial^{\aug^{2}_{i}, \aug^{2}_{j}} (\chord^{2, 1}_{2}\chord^{1, 2}_{2}) = (\chord^{2, 1}_{2}\chord^{1, 2}_{1}) + (\chord^{2, 1}_{1}\chord^{1, 2}_{2}).
\end{gathered}
\end{equation*}
The computation for words of length $2$ quickly follows from $\aug^{2}_{i}(\chord^{1,1}_{1}) = 1$ for all $i$.

According to the convention of Equation \eqref{Eq:SSDiffConvention}, the $E^{0}_{p,\deg}$ groups are additively generated by words of length $p$ having grading $\deg$ and the differential on this $0$th page preserves the word length filtration. The generators for these tables can be read off of the table at the start of \S \ref{Sec:SSGensAndDiffs}, yielding
\begin{equation*}
\begin{tikzcd}
E^{0}_{2,0} = \field^{2}  & \arrow[l] E^{0}_{2,1}=\field^{4} & \arrow[l] E^{0}_{2,2} = \field^{2}\\
E^{0}_{1,0} = \field^{3} & \arrow[l] E^{0}_{1,1} = \field^{3}
\end{tikzcd}
\end{equation*}
From the above computation of the $\partial^{\aug^{2}_{i}, \aug^{2}_{j}}$ we see that the differentials on the $0$th page are the usual bilinearized differentials along the $E^{0}_{1, \ast}$ row. Along the $E^{0}_{2, \ast}$ row we have
\begin{equation*}
\begin{gathered}
\partial^{\aug^{2}_{i}, \aug^{2}_{j}}(\chord^{1, 2}_{1}\chord^{2, 1}_{2}) = (\chord^{1, 2}_{1}\chord^{2, 1}_{1}), \quad \partial^{\aug^{2}_{i}, \aug^{2}_{j}}(\chord^{2, 1}_{2}\chord^{1, 2}_{1}) = (\chord^{2, 1}_{1}\chord^{1, 2}_{1}),\\
\partial^{\aug^{2}_{i}, \aug^{2}_{j}} (\chord^{1, 2}_{2}\chord^{2, 1}_{1}) = (\chord^{1, 2}_{1}\chord^{2, 1}_{1}), \quad \partial^{\aug^{2}_{i}, \aug^{2}_{j}} (\chord^{2, 1}_{1}\chord^{1, 2}_{2}) = (\chord^{2, 1}_{1}\chord^{1, 2}_{1}),\\
\partial^{\aug^{2}_{i}, \aug^{2}_{j}} (\chord^{1, 2}_{2}\chord^{2, 1}_{2}) = (\chord^{1, 2}_{1}\chord^{2, 1}_{2}) + (\chord^{1, 2}_{2}\chord^{2, 1}_{1}),\quad
\partial^{\aug^{2}_{i}, \aug^{2}_{j}} (\chord^{2, 1}_{2}\chord^{1, 2}_{2}) = (\chord^{2, 1}_{2}\chord^{1, 2}_{1}) + (\chord^{2, 1}_{1}\chord^{1, 2}_{2}).
\end{gathered}
\end{equation*}
By inspection this row is exact. So the $E^{1}$ page of our spectral sequence is
\begin{equation*}
\begin{tikzcd}
E^{1}_{1,0} = LCH^{\aug^{1}_{i}, \aug^{1}_{j}}_{0} & E^{1}_{1,1} = LCH^{\aug^{1}_{i}, \aug^{1}_{j}}_{1}
\end{tikzcd}
\end{equation*}
with vanishing differentials $E^{1}_{p, \deg} \xrightarrow{0} E^{1}_{p-1, \deg-1}$. So we see that the spectral sequence converges at the $r=1$ page yielding
\begin{equation*}
E^{r\geq 1}_{p, \deg} = \begin{cases}
LCH^{\aug^{1}_{i}, \aug^{1}_{j}}_{0} & p = 1\\
0 & p \neq 1.
\end{cases}
\end{equation*}
Packaging the information as a Poincaré polynomial as defined in Equation \eqref{Eq:SpecPolyDef}, we get
\begin{equation*}
P_{\aug^{2}_{i}, \aug^{2}_{j}}^{\spec}(t, x, y) = (1+x)y P_{\aug^{1}_{i}, \aug^{1}_{j}}(t)
\end{equation*}
as determined by the bilinearized Legendrian contact homology polynomials of Equation \eqref{Eq:BilinPoly}.

\subsection{$\newRSFT$ of $2$-cables}\label{Sec:Cables}

Now we investigate the $\newRSFT$ of two component links, using $\Z/2\Z$ gradings throughout the remainder of this section.

Let $B$ be a closed, oriented manifold so that $T^{\ast}B$ is an exact symplectic manifold when equipped with the canonical $1$-form, $pdq$. Here and throughout the remainder of the subsection, $\R_{t} \times T^{\ast}B$ is equipped with the canonical contact structure $\ker(dt + pdq)$.

\begin{defn}
A \emph{Legendrian cabling pattern} for $B$ is a Legendrian $\Leg \subset \R_{t} \times T^{\ast}B$ such that the composition of the inclusion with the projection
\begin{equation*}
\Leg \rightarrow \R_{t} \times T^{\ast}B \rightarrow B
\end{equation*}
is a covering map. We say that $\Leg$ has \emph{$d$ strands} where $d$ is the degree of $\Lambda \to B$.
\end{defn}

The Legendrian cabling pattern generalizes the construction of Legendrian cables which are solid torus links as described in \cite{NT:Torus}, frequently studied in the literature. Given a cabling pattern $\Leg$ for some $B$ and a Legendrian embedding of $B$ into some $\Mxi$, we can apply the Weinstein neighborhood theorem to realize $\Leg$ as  a Legendrian in $\Mxi$ as well. We then say that $\Leg$ is a \emph{$d$-cable of $B$}.

\begin{lemma}\label{Lemma:LocalVanishing}
Suppose that $B \subset \Mxi$ is a connected Legendrian in a contactization $\Mxi$ and let $\Leg \subset \Mxi$ be a cable, contained in a Weinstein neighborhood $N = [-\epsilon, \epsilon]_{t} \times \disk^{\ast}B$ of $B \subset \Mxi$. If $PDH(\Leg \subset N, \Partition) = 0$ for some partition, then $PDH(\Leg \subset M, \Partition)=0$ as well.
\end{lemma}

\begin{proof}
The size of the neighborhood $N$ can be shrunk by applying the time $T \ll 0$ flow of the contact vector field $t\partial_{t} + p\partial_{p}$, where $p\partial_{p}$ is the Liouville vector field on $T^{\ast}B$. This flow induces a Legendrian isotopy on $\Leg \subset B$ and so on $\Leg \subset \Mxi$.

Say that a word of chords on $\Leg \subset \Mxi$ is \emph{short} if all of its chords are completely contained in $N$. Shrinking the neighborhood sufficiently we see that, without loss of generality, actions of short words can be assumed arbitrarily small. Since the $\newRSFTA$ differential is action decreasing, it must send short words to short words. Since the short words are exactly the generators of $\newRSFTA(\Leg \subset N, \Partition)$, we see that the inclusion $N \subset \Mxi$ induces a DGA morphism $\newRSFTA(\Leg \subset N, \Partition) \to \newRSFTA(\Leg \subset M, \Partition)$ and hence a unital algebra morphism $\newRSFT(\Leg \subset N, \Partition) \to \newRSFT(\Leg \subset M, \Partition)$. Since the domain of the latter is zero by assumption, the target must be zero as well and the lemma is established.
\end{proof}

This lemma provides a means of producing $(\Leg, \Partition)$ for which $LCH(\Leg) \neq 0$ and $\newRSFT(\Leg, \Partition) = 0$.

Consider the case in which there are $d=2$ strands, $\Leg$ is disconnected, and $N^{\Partition} = 2$. We'll see that Floer's identification of holomorphic strips with Morse flow lines \cite{Floer:Morse} is sufficient to compute $\partial$. The generalized techniques of \cite{FO:FlowTree} can be applied to computations with more strands and \cite{Ekholm:FlowTrees} can be applied to compute $\partial$ for general Legendrians in $\R_{t} \times T^{\ast}B$ or to cables of Legendrians $B$ inside of $1$-jet spaces $\Mxi$. In this setup we can describe $\Leg$ as consisting of the $0$-section $\Leg_{1} = B$ and the $1$-jet of a smooth function $f: B \rightarrow \R$ having $0$ as a regular value,
\begin{equation*}
\Leg_{2} = \{  (f(q), -df(q))\ : q \in B\} \subset \R_{t} \times T^{\ast}B.
\end{equation*}
The function $f$ splits $B$ in two pieces,
\begin{equation*}
B = B^{-} \cup B^{+}, \quad B^{-} = f^{-1}((-\infty, 0])\quad B^{+} = f^{-1}([0, \infty))
\end{equation*}
whose intersection is $f^{-1}(0)$. Since $[\Leg_{1}] = \pm [\Leg_{2}] \neq 0 \in H_{n-1}(N)$ where $N$ is our Weinstein neighborhood of $B$, $\Leg \subset N$ cannot have a disconnected Lagrangian filling. We can ask when $\Leg$ cabled about a $B \subset \Mxi$ has a disconnected filling for more general $\Mxi$. The following result provides a local obstruction using $\newRSFT$.

\begin{prop}\label{Prop:CableFIllingObs}
For a $2$-cable $\Leg \subset N= \R_{t} \times T^{\ast}B$ equipped with the $N^{\Partition}=2$ partition, $\newRSFT\LegGrouped = 0$ if and only if the Bockstein morphism $\delta$ is non-zero,
\begin{equation*}
\delta: H_{\ast}(B, B^{-}) \rightarrow H_{\ast - 1}(B^{-}).
\end{equation*}
If $\delta = 0$, then $\newRSFT \simeq \tensorAlg(H)$ is a tensor algebra where $H$ is as in Equation \eqref{Eq:MorseHomology} below. Therefore if $\delta \neq 0$, then $\Leg \subset \Mxi$ has no disconnected filling.
\end{prop}

Consequently, only the ``N''-shaped sphere in Figure \ref{Fig:MorseFun} corresponds to a $\Leg$ having vanishing $\newRSFT$.

\begin{figure}[h]
\begin{overpic}[scale=.4]{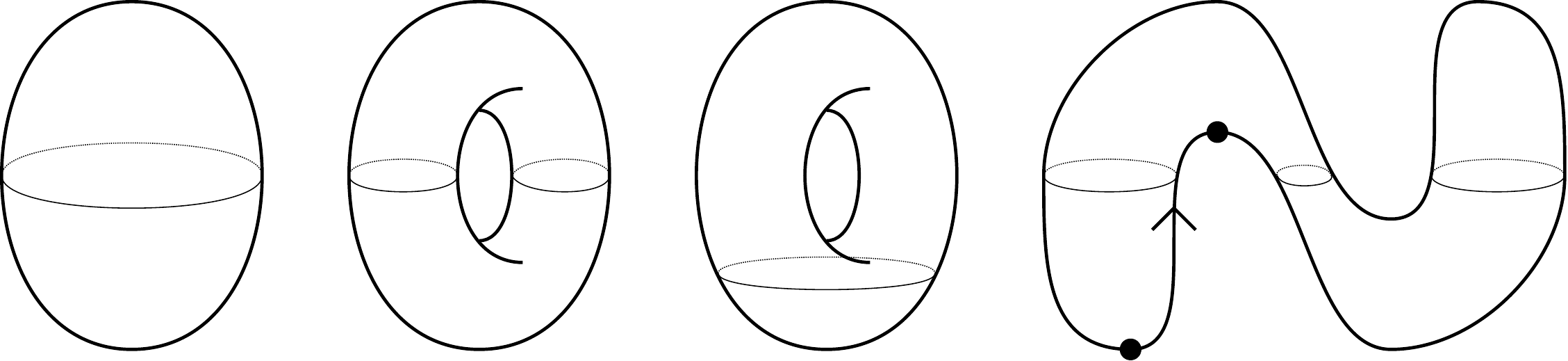}
\put(71, 3){$q^{-}$}
\put(77, 17){$q^{+}$}
\end{overpic}
\caption{Morse functions on spheres and tori with circles indicating the $f^{-1}(0)$. The regions above the circles are the $B^{+}$ and the regions below are the $B^{-}$. The $\grad f$ flow line indicated determines a holomorphic disk $u$ contributing to $\partial(\chord_{q^{+}}\chord_{q^{-}}) = 1$.}
\label{Fig:MorseFun}
\end{figure}

\begin{proof}
The chords are in one-to-one correspondence with critical points $q \in \Crit(f)$ of $f$ and we will write the corresponding chord as $\chord_{q}$. Looking at the relative values of $f$ we see that
\begin{equation*}
\chord_{q} \in \begin{cases}
\Chords_{1, 2} & f(q) > 0 \\
\Chords_{2, 1} & f(q) < 0.
\end{cases}
\end{equation*}
There are no pure chords and every generator $\word \in \orbitVS$ of $\Algebra$ has word length $2$, being of the form $(\chord_{q^{-}}\chord_{q^{+}})$ or $(\chord_{q^{+}}\chord_{q^{-}})$ for critical points $q^{\pm} \in B^{\pm}$. Therefore we can identify the vector space of words as
\begin{equation*}
V = \left( C^{\Morse}_{\ast}(f|_{B^{-}}) \otimes C^{\ast}_{\Morse}(f|_{B^{+}})\right) \oplus \left( C^{\ast}_{\Morse}(f|_{B^{+}}) \otimes C^{\Morse}_{\ast}(f|_{B^{-}}) \right) .
\end{equation*}
Note that we are viewing the critical points of $B^{-}$ as homological generators and the critical points of $B^{+}$ as cohomological generators.

The $\newRSFTA$ differential of such a word has two possible contributions. There could be
\be
\item strips positively asymptotic to one of the $\chord_{q^{\pm}}$ and having one negative asymptotic and
\item strips with two positive punctures, asymptotic to both of the $\chord_{q^{\pm}}$ and with no negative punctures.
\ee
Let's call these $\partial_{1}$ and $\partial_{2}$ contributions so that $\partial = \partial_{1} + \partial_{2}$ with the subscript indicating the numbers of positive punctures involved in curve counts. Writing $\filtration_{\otimes}$ for the ascending tensor length filtration on $\Algebra = \tensorAlg(\orbitVS)$, defined
\begin{equation*}
\filtration_{\otimes}^{k}\Algebra = \bigoplus_{k' \leq k} \orbitVS^{\otimes k'},
\end{equation*}
we observe that
\begin{equation*}
\partial_{1}\filtration_{\otimes}^{k}\Algebra \subset \filtration_{\otimes}^{k}\Algebra, \quad \partial_{2}\filtration_{\otimes}^{k}\Algebra \subset \filtration_{\otimes}^{k - 2}\Algebra.
\end{equation*}
Therefore treating $\Algebra$ as a $\field$ vector space so that $(\Algebra, \partial)$ is an additive chain complex, we can compute $\newRSFT = H(\Algebra, \partial)$ with a spectral sequence associated to the filtered complex $\filtration_{\otimes}$. Writing $D_{r}$ for the differential on the $E^{r}$ page, it's clear that
\begin{equation*}
D_{0} = \partial_{1}, \quad D_{1} = 0, \quad D_{2} = \partial_{2}, \quad D_{k > 2} = 0
\end{equation*}
so that $\newRSFT$ is a direct sum of subgroups of $E^{3}$.

According to \cite{Floer:Morse}, the $\partial_{1}$ contributions are counted as $-\grad f$ flow lines between critical points in $B^{-}$ which start at some $q^{-}$ and $\grad f$ flow lines between critical points in $B^{+}$ which start at some $q^{+}$. Indeed, such flow lines correspond to holomorphic strips in the Lagrangian projection which lift to strips in the symplectization of the $1$-jet space of $B$. The strips will have one positive and one negative end iff both of the chords correspond to critical points in the same $B^{\pm}$. As we are working over a field, $\field$, the Kunneth formula gives us isomorphisms between tensor products of homologies and homologies of tensor products, so that
\begin{equation*}
\begin{gathered}
E^{1} = \ker \partial_{1} / \im \partial_{1} = \tensorAlg(H),\\
H = \left( H_{\ast}^{\Morse}(B^{-}) \otimes H^{\ast}_{\Morse}(B^{+}) \right) \oplus \left( H^{\ast}_{\Morse}(B^{+}) \otimes H_{\ast}^{\Morse}(B^{-}) \right).
\end{gathered}
\end{equation*}
Using the facts that $H_{\ast}^{\Morse}(B^{-})$ computes homology and $H_{\Morse}^{\ast}(B^{+})$ computes cohomology \cite[Chapter 4]{Schwarz:Morse}, applying Poincar\'{e} duality, and changing notation for indices, we can rewrite
\begin{equation}\label{Eq:MorseHomology}
\begin{gathered}
H^{\ast}_{\Morse}(B^{+}) = H_{\ast+n-1}(B, B^{-})\\
H = \left( H_{\ast}(B^{-}) \otimes H_{\ast}(B, B^{-}) \right) \oplus \left( H_{\ast}(B, B^{-}) \otimes H_{\ast}(B^{-}) \right).
\end{gathered}
\end{equation}

Again by \cite{Floer:Morse}, the $\partial_{2}$ contributions are counted as $\grad f$ flow lines which begin at a $q^{-} \in B^{-}$ and end at some $q^{+} \in B^{+}$. On each summand of $H$, this is exactly the Bockstein homomorphism $\delta$. If $\delta$ vanishes, then $D_{2} = \partial_{2}$ vanishes and $\newRSFTA \simeq E_{1}$. Otherwise we can take an element of the form $x \otimes \delta(x)$ or $\delta(x)\otimes x \in E^{2}$ with $x \neq 0 \in H^{\ast}(B^{-})$ and see that $\partial_{2}\word = 1$, meaning that $\newRSFT$ vanishes. A specific example of a non-zero $\delta$ count is indicated in the right-most subfigure of Figure \ref{Fig:MorseFun}. The last statement of the proposition regarding filling obstructions then follows from Lemma \ref{Lemma:LocalVanishing} and functoriality of $PDH$.
\end{proof}

\subsection{Links in $\Rtwonminusone$ with vanishing $\newRSFT$}\label{Sec:UnknotLinks}

We continue to use the data of $f, B$ to determine a $2$-cable inside of $\R_{t}\times T^{\ast}B$ with $f$ determining $B^{\pm} \subset B$. Let $B \subset \Rtwonminusone$ be a Legendrian so that we view the cable $\Leg = \Leg_{1} \sqcup \Leg_{2}$ as being contained in $\Rtwonminusone$ as well. Suppose that the orientations of $\Leg_{1}$ and $\Leg_{2}$ differ by a sign $\sigma$.

The linking number $\lk(\Leg_{2}, \Leg_{1})$ can be computed as the signed count of chords in $\Chords_{2, 1}$. Each pure chord in $\Chords_{1, 1}$ determines one chord in $\Chords_{2, 1}$ having sign given by $\sigma$ times the sign of the corresponding element of $\Chords_{1, 1}$. See the two copy construction of \cite{EES:Duality}. All other chords in $\Chords_{1, 2}$ correspond to critical points of $B^{-}$ and their signed count is computed as $\sigma(-1)^{\ind_{\Morse}}$. Therefore
\begin{equation}\label{Eq:ChiLInking}
\lk(\Leg_{2}, \Leg_{1}) = \sigma(\tb(B) + \chi(B^{-})).
\end{equation}

It is not difficult to produce examples of $2$-cables in $\Rtwonminusone$ which have
\begin{equation*}
LCH(\Leg) \neq 0, \quad \newRSFTA\LegGrouped = 0
\end{equation*}
with the $N^{\Partition} = 2$ partition, and arbitrary linking number: By Lemma \ref{Lemma:CEAugCombo}, if $\CE(B)$ admits an augmentation, then so does $\CE(\Leg)$, implying $LCH(\Leg) \neq 0$. Meanwhile Proposition \ref{Prop:CableFIllingObs} and Equation \eqref{Eq:ChiLInking} can be applied to ensure that $\newRSFT\LegGrouped = 0$ while controlling the linking number.

For some specific examples, take $B \subset \Rtwonminusone$ to be the standard Legendrian unknot. Choose $f$ so that $B^{-}$ consists of a disk together with some copies of $S^{1} \times \disk^{n-2}$ implying that the Bockstein $\delta$ of Proposition \ref{Prop:CableFIllingObs} is non-zero. By choosing orientations appropriately ($\sigma = (-1)^{n-1}$ in the above notation) we can guarantee that the linking number vanishes. For $n=3$ with a single $S^{1} \times \disk^{1}$, this corresponds exactly to the ``N''-shaped sphere of Figure \ref{Fig:MorseFun} which we draw as a spinning \cite{EES:LegendriansInR2nPlus1} in Figure \ref{Fig:FrontSpinning}.

\begin{figure}[h]
\begin{overpic}[scale=.5]{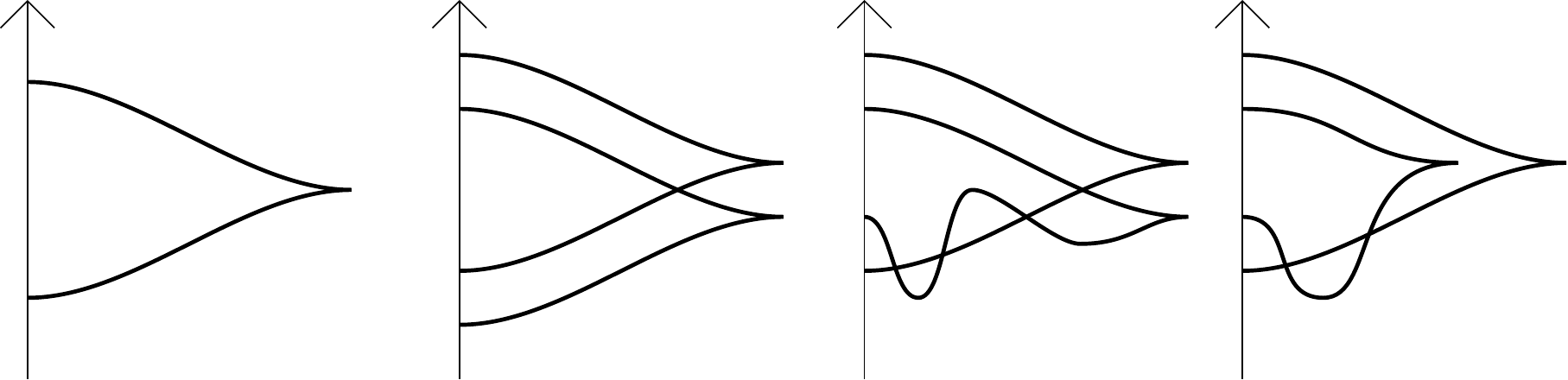}
\end{overpic}
\caption{Front spinnings \cite{EES:LegendriansInR2nPlus1} of Legendrian arcs about the $t$-axis determining Legendrian spheres in $(\R^{5}, \xi_{std})$. On the left, the spinning produces the ``flying saucer'' picture of the unknot. The Hopf link given by the $2$-cable corresponding to a constant function is center-left. Spinning the center-right figure and then breaking the $S^{1}$ symmetry will give a two-component link with zero linking number corresponding to the ``N'' shaped sphere of Figure \ref{Fig:MorseFun}. On the right, the center-right subfigure is simplified by a Reidemeister move.}
\label{Fig:FrontSpinning}
\end{figure}

Now we compute the $PDH$ of a Hopf link $\Leg$ with $n$ arbitrarily, generalizing the $n=2$ case of \S \ref{Sec:HopfLink3d}. We define $\Leg \subset \Rtwonminusone, n > 2$ as the union of an unknot $\Leg_{1}$ with a push-off $\Leg_{2} = (-1)^{n-1}\Flow^{\epsilon}_{\partial_{t}}\Leg_{1}$. Here the sign indicates orientation and $\epsilon > 0$ small. We can view $\Leg$ as a $2$-cable of $\Leg_{1}$ with associated Morse function $f$ a constant. Therefore one of the $B^{\pm}$ is empty and Proposition \ref{Prop:CableFIllingObs} cannot be used to infer the vanishing of $PDH$.

\begin{prop}
The Hopf link $\Leg \subset \Rtwonminusone$ has a filling by an oriented Lagrangian $[-1, 1] \times S^{n-1}$ and satisfies $\newRSFT\LegGrouped = 0$ when using the $N^{\Partition} = 2$ partition.
\end{prop}

\begin{proof}[Sketch of the proof]
We'll only sketch the proof, heavily relying on standard references.

For the existence of the Lagrangian filling, which is well known, it is a standard fact that each $\Leg_{i}$ bounds a Lagrangian disk and up to ambient Hamiltonian isotopy we can view these disks as $\R^{n}, J\R^{n} \cap \disk^{2n}$ where $\disk^{2n} \subset \C^{n}$ is equipped with the standard symplectic structure and we view $\Rtwonminusone$ as the complement of a point in $S^{2n - 1} = \partial \disk^{2n}$. The choice of orientation tells us that the sign of the intersection is positive, so that we can perform a Lagrange-Polterovich surgery \cite{Polterovich:Surgery} at the intersection point to obtain an oriented Lagrangian filling $\Lag$ which is smoothly $\disk^{n}\# \disk^{n} \simeq [-1, 1]\times S^{n-1}$. Exactness of the filling is given to us for free as $n > 2$ so $H_{1}(\Lag) = 0$.

To establish $\newRSFTA = 0$, we return to working within $\Rtwonminusone$ and apply the displacement technique of \cite{EES:Duality} or invariance of $LCH$ under double point moves as in \cite{EES:LegendriansInR2nPlus1}. Alternatively we can look to Figure \ref{Fig:HopfDifferential} as the computation will be essentially the same.

There are exactly two mixed chord $\chord_{1} \in \Chords_{1, 2}, \chord_{2} \in \Chords_{2, 1}$. Apply a generic Legendrian isotopy which keeps $\Leg_{1}$ fixed and unlinks the pair with $\Leg_{2}$ following some family $\Leg_{2}(T), T \in [0, 1+\epsilon]$ for which $\Leg_{2} = \Leg_{2}(0)$ with the pair unlinked for $T > 1$. For $T < 1$ the mixed chords vary smoothly and converge to a quadratic double point at $T = 1$. In the Lagrangian projection, this appears as the application of a double point move as described in \S \ref{Sec:HighDimDoublePt}. For $T = 1 - \epsilon$ with $\epsilon$ small, there will be exactly one holomorphic strip in the Lagrangian projection which is asymptotic to the double points determined by the chords. This contributes a holomorphic disks to $\partial (\chord_{1}\chord_{2})$ with two positive punctures and no negative punctures. By energy considerations, there can be no other disks. Therefore $\partial(\chord_{1}\chord_{2}) = 1$ and so $\newRSFT = 0$.
\end{proof}

\appendix

\section{Proofs of invariance}\label{App:Invariance}

Here we establish filtered stable tame isomorphism invariance of $\newRSFTA$ as defined in Definition \ref{Def:FilteredStableTameIso}, proving Theorem \ref{Thm:Main}. We break up the invariance proof into three steps:
\be
\item Triple point moves for $\Leg$ in $3$-manifolds are addressed in \S \ref{Sec:TriplePoint}.
\item Double point moves for $\Leg$ in $3$-manifolds are addressed in \S \ref{Sec:3dDoublePt}.
\item Invariance for $\Leg$ in high dimensional manifolds is addressed in \S \ref{Sec:HighDimInvariance}.
\ee

The invariance proofs are modeled on those appearing in \cite{Chekanov:LCH, Ekholm:Z2RSFT, EES:LegendriansInR2nPlus1, EES:PtimesR, EHK:LagrangianCobordisms, ENS:Orientations, Ng:RSFT} and we will rely on references to speed up the exposition whenever possible. The main technical difficulty we'll encounter is the management of multiple simultaneous stabilizations occurring during application of double point moves, which are handled algebraically.

Homotopy invariance of cobordism maps under deformation of partitioned Lagrangian cobordisms $\LagGrouped$ follows from \cite{Ekholm:Z2RSFT} and will not be addressed. Likewise with $\ring = \field[H_{1}(\Leg)]$ coefficients, changes in capping path and basepoints determine automorphisms of $\newRSFTA$ induced by automorphisms of $\ring$ following standard arguments appearing in the references.

When $H^{1}(\SympBase) \neq 0$, $\newRSFTA$ will in general depend on the homotopy class of the trivialization of the determinant line bundle $\Det_{\SympBase}$ used to define gradings in \S \ref{Sec:Gradings}. Indeed the homotopy classes of such trivializations is an affine space modeled on $H^{1}(\SympBase) \simeq [(\SympBase, \pt), (S^{1}, \pt)]$ (that is, homotopy classes of pointed maps) and homotopically non-trivial modification of a trivialization will result in a grading shift on generators.

\subsection{Triple point moves}\label{Sec:TriplePoint}

Here we show that modifications of a $\dim \Leg = 1$ Legendrian by a triple point move induces a stable-tame isomorphism of planar diagram algebras. There are two types of moves to consider, both shown in Figure \ref{Fig:TriplePoint}. The modification of $\Leg$ takes place in a $\R_{t} \times \disk$ for a disk $\disk \subset W$ where the ambient contact manifold is $\R_{t} \times W$. Our proof is independent of the partition $\Partition$.

Each type of move may be undone by rotating a portion of the diagram by an angle of $\pi$ and then reapplying a move of the same type. Hence it suffices to consider each move in a single direction ($\rightarrow$ in the figure). The labels $q_{k}$ will help us track holomorphic disks and can be ignored for now.

\begin{figure}[h]
\begin{overpic}[scale=.5]{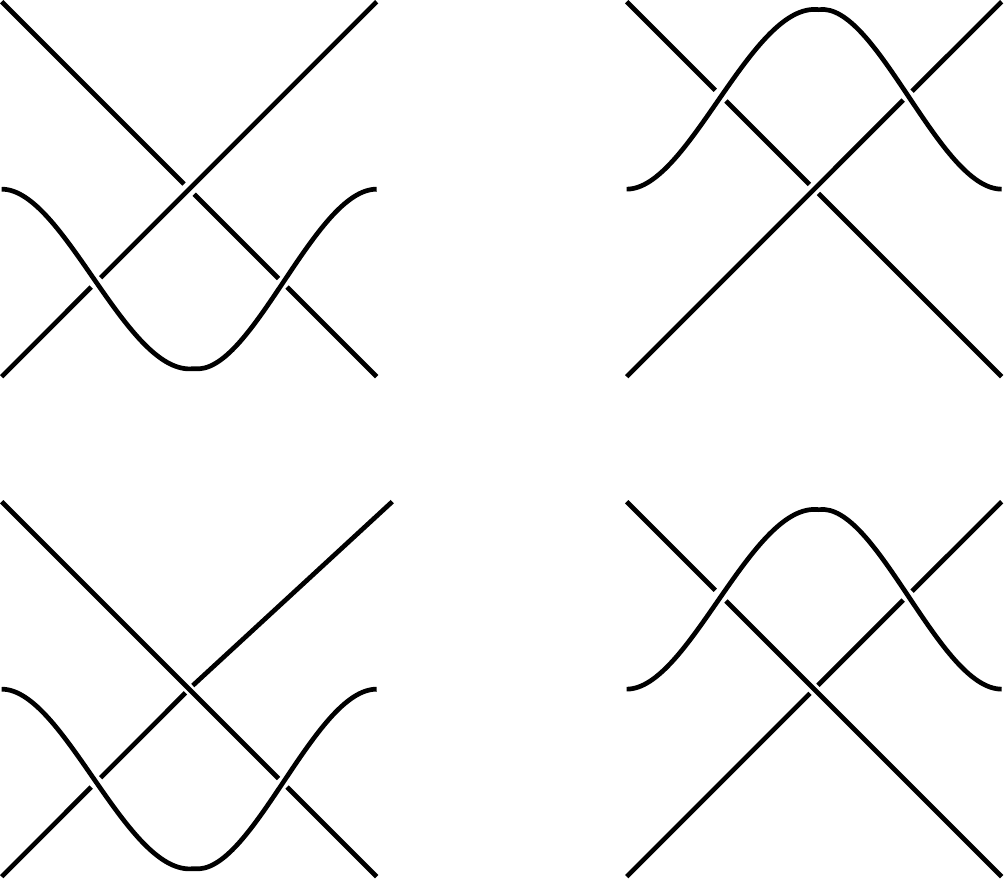}
\put(-10, 85){$\Leg_{i_{1}}$}
\put(-10, 67){$\Leg_{i_{2}}$}
\put(-10, 50){$\Leg_{i_{3}}$}
\put(44, 68){\vector(1, 0){12}}
\put(48, 70){I}
\put(36, 71){$q_{1}$}
\put(36, 82){$q_{2}$}
\put(-2, 82){$q_{3}$}
\put(-2, 71){$q_{4}$}
\put(-2, 53){$q_{5}$}
\put(36, 53){$q_{6}$}

\put(17, 74){$a$}
\put(8, 52){$b$}
\put(28, 52){$c$}
\put(80, 60){$a$}
\put(71, 70){$c$}
\put(89, 70){$b$}

\put(-10, 37){$\Leg_{i_{1}}$}
\put(-10, 19){$\Leg_{i_{2}}$}
\put(-10, 0){$\Leg_{i_{3}}$}
\put(44, 18){\vector(1, 0){12}}
\put(48, 20){II}

\put(17, 25){$a$}
\put(8, 3){$b$}
\put(27, 3){$c$}
\put(80, 11){$a$}
\put(71, 21){$c$}
\put(89, 21){$b$}

\end{overpic}
\caption{Triple point Reidemeister moves of type I and II.}
\label{Fig:TriplePoint}
\end{figure}
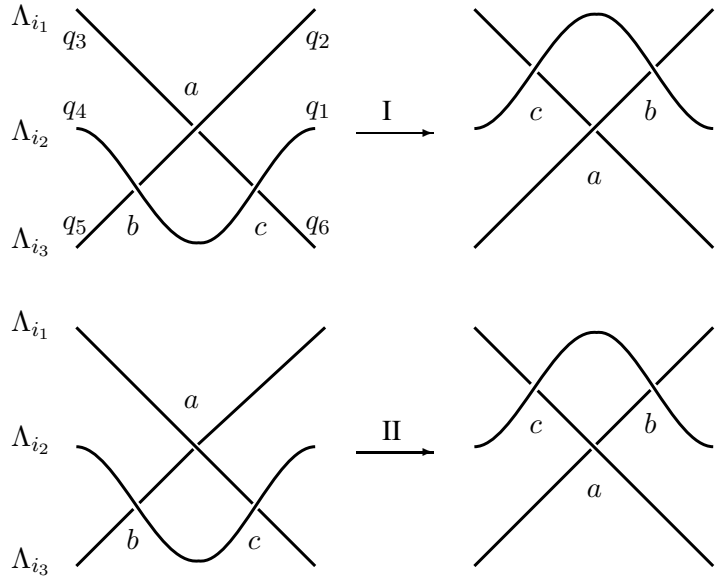

In each row of Figure \ref{Fig:TriplePoint} we have indicated a one-to-one correspondence between the chords before and after the move is applied. The correspondence preserves the assignments of starting and ending points of chords to connected components of $\Leg$. Therefore a triple point move induces a one-to-one correspondence between generators of $\Algebra$ before and after the move the applied. We write $\chord_{i}$ for the chords of the $\Leg$ located outside of the local picture with chords $a$, $b$, and $c$ lying within the local picture. Denote by $\partial^{-}$ the differential of $\Algebra$ before a triple point move is applied and $\partial^{+}$ for the differential after the move is applied.

\begin{prop}\label{PropTriplePoint}
Consider disks $u_{I}$ and $u_{II}$ with boundary words
\begin{equation*}
\boundaryWord(u_{I}) = c^{+}b^{-}a^{-}, \quad \boundaryWord(u_{II}) = a^{+}c^{+}b^{-}
\end{equation*}
as shown in Figure \ref{Fig:TriplePointCob}. Then the maps
\begin{equation*}
\Phi_{I} = \Id + \mu_{u_{I}}: \Algebra \rightarrow \Algebra, \quad \Phi_{II} = \Id + \mu_{u_{II}}: \Algebra \rightarrow \Algebra
\end{equation*}
induce stable tame isomorphisms $(\Algebra, \partial^{-}, \filtration) \rightarrow (\Algebra, \partial^{+}, \filtration)$ between the $\newRSFTA$ algebras before and after application of the triple point moves of type I and II, respectively.
\end{prop}

\begin{figure}[h]
\begin{overpic}[scale=.5]{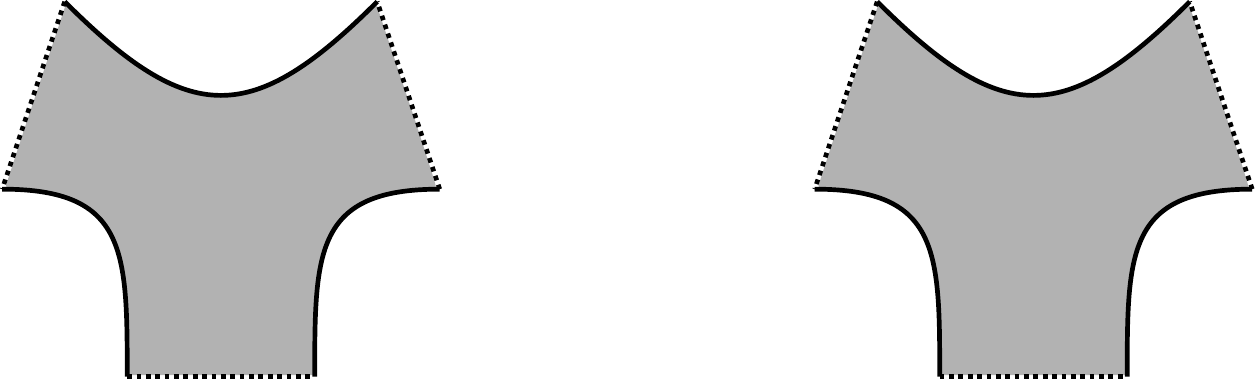}
\put(37, 20){$a^{-}$}
\put(16, -5){$b^{-}$}
\put(-5, 20){$c^{+}$}
\put(12, 15){$\planarDiagram(u_{I})$}

\put(101, 20){$a^{+}$}
\put(81, -5){$b^{-}$}
\put(60, 20){$c^{+}$}
\put(75, 15){$\planarDiagram(u_{II})$}
\end{overpic}
\vspace{3mm}
\caption{Planar diagrams for the triangles $u_{I}$ and $u_{II}$.}
\label{Fig:TriplePointCob}
\end{figure}
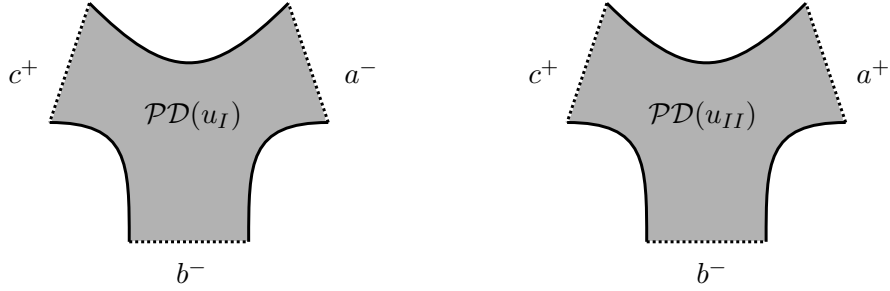

Note that $u_{I}$ gives exactly the cobordism map on $LCH$ associated to the Lagrangian cobordism determined by the trace of the Legendrian isotopy as described in \cite{EHK:LagrangianCobordisms}. We conjecture that $u_{II}$ is likewise a cobordism map. Our proof strategy mimics that of \cite{Ng:RSFT}.

\begin{proof}

In each subfigure of Figure \ref{Fig:TriplePoint}, there is a $\ind = 1$ holomorphic triangle $u_{\triangle}$ which is entirely contained in the diagram. The computation of $\ind(u_{\triangle})=1$ follows from the fact that the Lagrangian projection is an embedding having only convex corners, cf. \cite{Avdek:LSFT}. The boundary words of these triangles are
\begin{equation}\label{Eq:BoundaryOfSmallTriangle}
\word(u_{\triangle}) = \begin{cases}
a^{+}b^{+}c^{-} & \text{type I} \\
b^{+}c^{-}a^{-} & \text{type II}.
\end{cases}
\end{equation}
We'll call these holomorphic triangles \emph{tiny triangles} throughout the proof. Observe that the tiny triangle for a type I move has two positive punctures and so cannot contribute to the differential for the Chekanov-Eliashberg algebra.

The proof will follow easily from the introduction of some new language. Recall that $\disk \subset W$ is the region in which the triple point move has been applied. Let $u^{\pm}$ be an $\ind = 1$ holomorphic disk. Unless $u^{\pm}$ is a tiny triangle, a connected components of $(\pi_{xy}u^{\pm})^{-1}(\disk)$ must intercept some point $q_{k_{1}} \in \Leg$ as $u^{\pm}$ enters $\disk \subset W$ (following the boundary) and then some other $q_{k_{2}}\in \Leg$ as it exits $\disk \subset W$ where the $q_{k}$ are as shown in Figure \ref{Fig:TriplePoint}.

If such a connected component does not have a puncture at at one of $a^{\pm},b^{\pm},c^{\pm}$, then $k_{2} = k_{1} + 3$ where the subscripts are modulo $6$. Such a connected component will not be modified by the triple point move. Otherwise, we'll have $k_{2} = k_{1} - 1$ or $k_{2} = k_{1} - 2$ in which we say that $(\pi_{xy}u^{\pm})^{-1}(\disk \subset W)$ is a \emph{disk segment} and we write $\check{u}^{\pm}$ for $u^{\pm}$ with its disk segments removed.

\renewcommand{\arraystretch}{1.5}
\begin{center}
\begin{table}[ht]
\begin{tabular}{| c | c| c| c|}
\hline
$v^{-}$ at I & $v^{+}$ at I & $v^{-}$ at II & $v^{+}$ at II\\
\hline \hline
$(q_{1}q_{6})c^{-}$ & $(q_{1}q_{6})c^{-} + (q_{1}q_{6})b^{-} a^{-}$ & $(q_{1}q_{6})c^{-}$ & $(q_{1}q_{6})c^{-}+(q_{1}q_{6})b^{-} a^{+}$\\
$(q_{1}q_{5})b^{-}$ & $(q_{1}q_{5})b^{-}$ & $(q_{1}q_{5})b^{-}$ & $(q_{1}q_{5})b^{-}$\\ 
\hline
$(q_{2}q_{1})b^{+} + (q_{2}q_{1})a^{-}c^{+}$ & $(q_{2}q_{1})b^{+}$ & $(q_{2}q_{1})b^{+} + (q_{2}q_{1})a^{+}c^{+}$ & $(q_{2}q_{1})b^{+}$\\
$(q_{2}q_{6})a^{-}$ & $(q_{2}q_{6})a^{-}$ & $(q_{2}q_{6})a^{+}$ & $(q_{2}q_{6})a^{+}$\\
\hline
$(q_{3}q_{1})c^{+}$ & $(q_{3}q_{1})c^{+}$ & $(q_{3}q_{1})c^{+}$ & $(q_{3}q_{1})c^{+}$\\ 
$(q_{3}q_{2})a^{+}$ & $(q_{3}q_{2})a^{+} + (q_{3}q_{2})c^{+}b^{-}$ & $(q_{3}q_{2})a^{-}$ & $a^{-}+ (q_{3}q_{2})c^{+}b^{-}$\\ 
\hline
$(q_{4}q_{2})b^{-}$ & $(q_{4}q_{2})b^{-}$ & $(q_{4}q_{2})b^{-}$ & $(q_{4}q_{2})b^{-}$\\
$(q_{4}q_{3})c^{-} + (q_{4}q_{3})b^{-}a^{-}$ & $(q_{4}q_{3})c^{-}$ & $(q_{4}q_{3})c^{-} + (q_{4}q_{3})b^{-}a^{+}$ & $(q_{4}q_{3})c^{-}$\\
\hline
$(q_{5}q_{3})a^{-}$ & $(q_{5}q_{3})a^{-}$ & $(q_{5}q_{3})a^{+}$ & $(q_{5}q_{3})a^{+}$\\
$(q_{5}q_{4})b^{+}$ & $(q_{5}q_{4})b^{+} + (q_{5}q_{4})a^{-}c^{+}$ & $(q_{5}q_{4})b^{+}$ & $(q_{5}q_{4})b^{+} + (q_{5}q_{4})a^{+}c^{+}$\\
\hline
$(q_{6}, q_{5})a^{+} + (q_{6}, q_{5})c^{+}b^{-}$ & $(q_{6}, q_{5})a^{+}$ & $(q_{6}, q_{5})a^{-} + (q_{6}, q_{5})c^{+}b^{-}$ & $(q_{6}, q_{5})a^{-}$\\
$(q_{6}q_{4})c^{+}$ & $(q_{6}q_{4})c^{+}$ & $(q_{6}q_{4})c^{+}$ & $(q_{6}q_{4})c^{+}$\\
\hline
\end{tabular}
\vspace{3mm}
\caption{Segments of $\ind = 1$ disks incident to chords in a triple point move.}
\label{Table:DiskSegments}
\end{table}
\end{center}

Each disk segment $v^{\pm}$ has a boundary word $\boundaryWord(v^{\pm})$ as shown in Table \ref{Table:DiskSegments} containing some elements of the $(q_{k}, q_{k-1})$ or $(q_{k}, q_{k-2})$. There is exactly one $v^{\pm}_{k, k-1}$ and exactly one $v^{\pm}_{k, k-2}$ for each $\pm$ and each $k=1, \dots, 6$.

Each $\check{u}^{\pm}$ likewise has a boundary word $\boundaryWord(\check{u}^{\pm})$. We can then express $\boundaryWord(u^{\pm})$ as a composition of boundary words $\boundaryWord(\check{u}^{\pm})$ and the $\boundaryWord(v^{\pm})$ by gluing at the $(q_{k}, q_{k-1})$ or $(q_{k}, q_{k-2})$.

Therefore we can write $\mu_{u^{\pm}}$ as a composition of operators $\mu_{\check{u}^{\pm}}$ and $\mu_{v}^{\pm}$ associated to $\boundaryWord(\check{u}^{\pm})$ and $\boundaryWord(v^{\pm})$. Writing $\orbitVS$ for the vector space of words and $Q$ for the vector space generated by the $(q_{k}, q_{k-1})$ and $(q_{k}, q_{k-2})$, the $\mu_{\check{u}^{\pm}}$ and $\mu_{v^{\pm}}$ act on $\tensorAlg(\orbitVS \oplus Q)$. Define $\pi^{\pm}: \tensorAlg(V \oplus Q) \to \tensorAlg(V)$ to be the algebra homomorphism determined by $\pi|_{V} = \Id_{V}$ and $\pi|_{Q}= \mu_{v^{\pm}_{k, k-1}} + \mu_{v^{\pm}_{k, k-2}}$. Then
\begin{equation}\label{Eq:TriplePointEquivalence}
\begin{gathered}
\partial^{\pm} =  \pi^{\pm}\circ \left( \mu_{\triangle} + \sum \mu_{\check{u}^{\pm}} \right).
\end{gathered}
\end{equation}

Let $\Phi_{\ast}$ be one of either $\Phi_{I}$ or $\Phi_{II}$ and let $\mu_{\ast}$ be the associated operator $\mu_{u_{I}}$ or $\mu_{u_{II}}$ so that $\Phi_{\ast} = \Id + \mu_{\ast}$. We seek to show that $\partial^{+} = \Phi_{\ast}\partial^{-}\Phi_{\ast}$. In light of Equation \eqref{Eq:TriplePointEquivalence}, this can be verified by applying the operator $\mu_{\ast}$ before and after each of the terms in the $v^{-}$ columns of Table \ref{Table:DiskSegments} and seeing that we obtain the associated $v^{+}$. 

For example, at the third row of the table in the type I columns we have
\begin{equation*}
v_{2, 1}^{-} = (q_{2}q_{1})b^{+} + (q_{2}q_{1})a^{-}c^{+}, v_{2, 1}^{+} = (q_{2}q_{1})b^{+}.
\end{equation*}
Applying $\mu_{\ast}$ on the right of $v_{2, 1}^{-}$ we get $(q_{2}q_{1})b^{+} + 2(q_{2}q_{1})a^{-}c^{+} = (q_{2}q_{1})b^{+}$. Then then $\mu_{\ast}$ on the left has no effect, so we get $ (q_{2}q_{1})b^{+} = v^{+}_{2, 1}$ again. Analogous computations can be applied to each $v^{-}$ and $v^{+}$ pair in Table \ref{Table:DiskSegments}, yielding $\partial^{+} = \Phi_{\ast}\partial^{-}\Phi_{\ast}$ as desired.
\end{proof}

\subsection{Double point moves in $\R_{t} \times W$}\label{Sec:3dDoublePt}

Now we address double point Reidemeister moves for $\dim\Leg = 1$. As in the previous subsection, we locally modify $\Leg$ within a $\R_{t} \times \disk \subset \R_{t} \times W$ for some disk $\disk$ in $W$.

\begin{prop}\label{Prop:DoublePointInv}
The stable tame isomorphism class of $\newRSFTA$ is invariant under double point Reidemeister moves as shown in Figure \ref{Fig:DoublePoint}.
\end{prop}

There are two cases to consider: If $\Leg_{i_{1}}$ and $\Leg_{i_{2}}$ belong to the same $\Leg^{\Partition}_{j}$, the chords $a$ and $b$ are $\Partition$ pure and we can prove Proposition \ref{Prop:DoublePointInv} by following existing proofs of invariance \cite{Chekanov:LCH, ENS:Orientations} almost word for word. We will assume $\Leg_{i_{1}}$ and $\Leg_{i_{2}}$ belong to different pieces of $\Leg$, writing
\begin{equation*}
\Leg_{i_{1}} \subset \Leg^{\Partition}_{1}, \quad \Leg_{i_{2}} \subset \Leg^{\Partition}_{2}.
\end{equation*}
We will detail a proof of Proposition \ref{Prop:DoublePointInv} in this case, again by modifying the strategy of \cite{Chekanov:LCH, ENS:Orientations}. The proof here will be more difficult, as in our context a double point move can simultaneously introduce many stabilization of our algebra.

\begin{figure}[h]
\begin{overpic}[scale=.5]{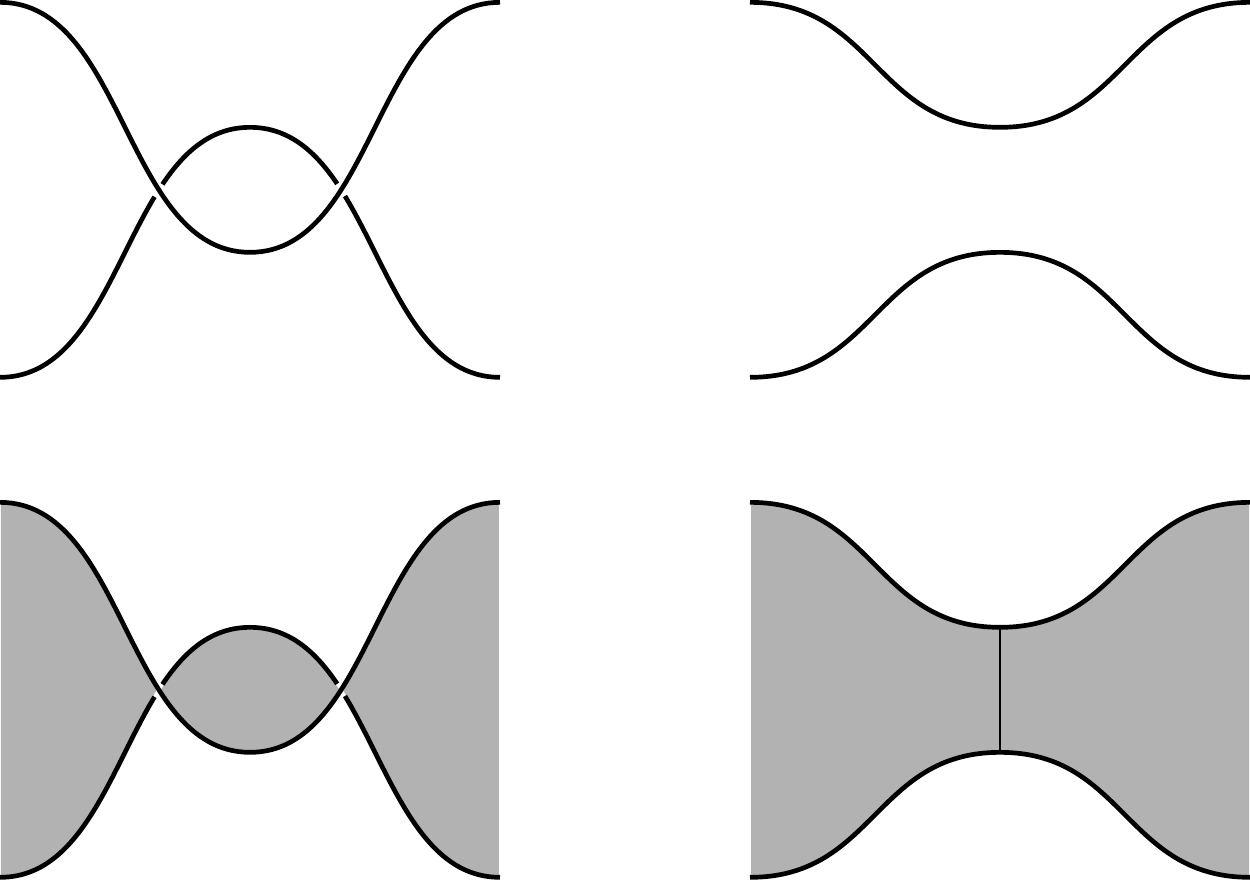}
\put(6, 54){$a$}
\put(32, 54){$b$}
\put(44, 55){\vector(1, 0){12}}
\put(6, 14){$+$}
\put(32, 14){$-$}
\put(18, 14){$u_{ab}$}
\put(-6, 68){$\Leg_{i_{2}}$}
\put(-6, 39){$\Leg_{i_{1}}$}
\put(75, 14){$\gamma$}
\end{overpic}
\caption{A double point Reidemeister move (top row) along with disk segments appearing in the proof Lemma \ref{Lemma:PartialChain} (bottom row).}
\label{Fig:DoublePoint}
\end{figure}
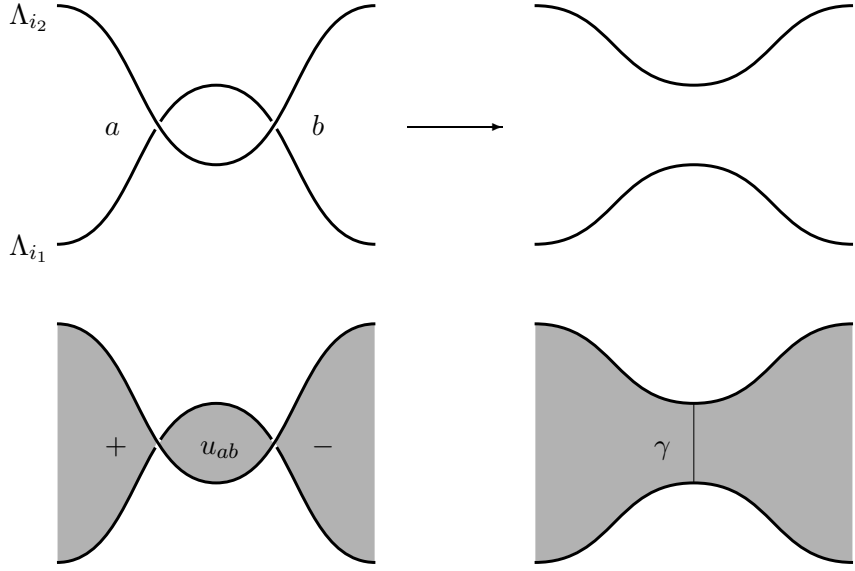

There are two chords, $a$ and $b$, eliminated by the double point move which satisfy $\action(a) = \action(b) + \epsilon$ with $\epsilon > 0$ given by the area of the disk $u_{ab}$ shown in Figure \ref{Fig:DoublePoint} for which
\begin{equation*}
\boundaryWord(u_{ab}) = a^{+}b^{-}.
\end{equation*}
We can assume that $\epsilon$ is arbitrarily small.

Because both $a, b \in \Chords^{\Partition}_{1,2}$, they cannot appear simultaneously in any $\word$ and neither can appear in a $\word$ of length $1$. We have one to one correspondences between words of the form
\begin{equation*}
(\word_{l}a\word_{r}), \quad (\word_{l}b\word_{r})
\end{equation*}
where the $\word_{l}$ and $\word_{r}$ are sequences of chords for which the $(\word_{l}a\word_{r})$ are admissible $\Partition$ cyclic words of chords. For such pairs of words, we have
\begin{equation*}
\action(\word_{l}a\word_{r}) = \action(\word_{l}b\word_{r}) + \epsilon.
\end{equation*}
The disk $u_{ab}$ has $\ind = 1$ and can have its planar diagram inscribed in the planar diagram of any $(\word_{l}a\word_{r})$. Therefore
\begin{equation*}
|(\word_{l}a\word_{r})| = |(\word_{l}b\word_{r})| + 1.
\end{equation*}
Because $\action(\word)$ depends only on the unordered collection of chords in a given $\word$, the actions of many of these words will agree. We can write these words as $(\word^{i}_{l, j}a\word^{i}_{r, j}), (\word^{i}_{l, j}b\word^{i}_{r, j})$ with the superscripts ordered so that
\begin{equation*}
i < i' \implies \action(\word^{i}_{l, j}a\word^{i}_{r, j}) \leq \action(\word^{i'}_{l, j'}a\word^{i'}_{r, j'}), \quad \action(\word^{i}_{l, j}b\word^{i}_{r, j}) \leq \action(\word^{i'}_{l, j'}b\word^{i'}_{r, j'})
\end{equation*}
with $i = 1, \cdots, N$ and $j = 1,\dots,N_{i}$ for some $N$ and $N_{i}$. Using $\epsilon$ arbitrarily small, in particular smaller that the action of any single chord, we label all words $\word_{i}$ not containing $a$ or $b$ with indices $\word_{i}$ so that
\begin{equation}\label{Eq:DoublePointActionOrdering}
\begin{gathered}
\action\left(\word_{-n_{0}}\right) \leq \cdots \leq \action\left(\word_{0}\right) \\
< \action\left(\word^{1}_{l, 1}b\word^{1}_{r, 1}\right) = \cdots = \action\left(\word^{1}_{l, N_{1, j}}b\word^{1}_{r, N_{1, j}}\right)\\
< \action\left(\word^{1}_{l, 1}a\word^{1}_{r, 1}\right) = \cdots = \action\left(\word^{1}_{l, N_{1, j}}a\word^{1}_{r, N_{1, j}}\right) \\
< \action\left(\word_{1}\right) \leq \cdots \leq \action\left(\word_{n_{1}}\right) \\
\cdots \\
< \action\left(\word_{n_{N_{i}-2}+1}\right) \leq \cdots \leq \action\left(\word_{n_{N_{i}-1}}\right) \\
< \action\left(\word^{N_{i}}_{l, 1}b\word^{N_{i}}_{r, 1}\right) = \cdots = \action\left(\word^{N_{i}}_{l, N_{i, j}}b\word^{N_{i}}_{r, N_{i, j}}\right)\\
< \action\left(\word^{N_{i}}_{l, 1}a\word^{N_{i}}_{r, 1}\right) = \cdots = \action\left(\word^{N_{i}}_{l, N_{i, j}}a\word^{N_{i}}_{r, N_{i,j}}\right)
\\
< \action\left(\word_{n_{N_{i}-1}+1}\right) \leq \cdots \leq \action\left(\word_{n_{N_{i}}}\right)
\end{gathered}
\end{equation}

Let $(\Algebra^{-}, \partial^{-}, \filtration)$ be the max-filtered DGA before the application of the double point move and $(\Algebra^{+}, \partial^{+}, \filtration^{+})$ for the corresponding mfDGA after the move has been applied. Then $\Algebra^{+}$ is a filtered, graded subalgebra of $\Algebra^{-}$, although the inclusion $\Algebra^{+} \rightarrow \Algebra^{-}$ is not necessarily a chain map.

For each pair $(\word^{i}_{l, j}a\word^{i}_{r, j}), (\word^{i}_{l, j}b\word^{i}_{r, j})$ we define abstract variables $e^{i}_{j, 1}$ and $e^{i}_{j, 2}$ with filtration levels given by the word lengths of the $(\word^{i}_{l, j}a\word^{i}_{r, j})$ and gradings
\begin{equation*}
|e^{i}_{j, 1}| = |(\word^{i}_{l, j}a\word^{i}_{r, j})|, \quad |e^{i}_{j, 2}| = |e^{i}_{j, 1}| - 1 = |(\word^{i}_{l, j}b\word^{i}_{r, j})|.
\end{equation*}
Define $\Stab\Algebra^{+}$ to be the extension of $\Algebra^{+}$ obtained by adjoining all of the variables $e^{i}_{j,1}, e^{i}_{j, 2}$ with differential
\begin{equation*}
\partial|_{\Algebra^{+}} = \partial^{+}, \quad \partial e^{i}_{j, 1} = e^{i}_{j, 2}, \quad \partial e^{i}_{j, 2} = 0
\end{equation*}

There is a projection map
\begin{equation*}
\pi: \Stab\Algebra^{+} \rightarrow \Algebra^{+}
\end{equation*}
defined by setting $e^{i}_{j, 1}, e^{i}_{j, 2} = 0$. Clearly $(\Stab\Algebra^{+}, \partial^{+}, \filtration)$ is obtained from $(\Algebra, \partial^{+}, \filtration)$ by a sequence of $(f, \deg)$-stabilizations with
\begin{equation*}
f = \filtLevel(\word^{i}_{l, j}a\word^{i}_{r, j}), \quad \deg = |(\word^{i}_{l, j}a\word^{i}_{r, j})|
\end{equation*}
For $\const \in \R$ define $\Algebra^{-}_{<\const}$ ($\Algebra^{-}_{\leq\const}$) to be the subalgebra of $\Algebra^{-}$ generated by words $\word$ of with $\action(\word) < \const$ (respectively, $\action(\word) \leq \const$).

The strategy for the remainder of the proof is to inductively build a sequence of algebra isomorphisms $\Phi_{k}: \Algebra^{-} \rightarrow \Stab\Algebra^{+}$ which are chain maps when restricted to successively larger subalgebras $\Algebra^{-}_{\leq\action^{-}_{\const}}$ with $\const$ increasing. When we reach $\const = \action(\word_{n_{N_{i}}})$ the proof will be complete. The following notation and lemmas will start the induction.

The disk $u_{ab}$ satisfies $\mu_{u_{ab}}(\word^{i}_{l, j}a\word^{i}_{r, j}) = (\word^{i}_{l, j}b\word^{i}_{r, j})$ and clearly contributes to $\partial^{-}(\word^{i}_{l, j}a\word^{i}_{r, j})$. Due to our action estimates, there are no other terms in $\partial^{-}(\word^{i}_{l, j}a\word^{i}_{r, j})$ which are multiples of $(\word^{i}_{l, j}b\word^{i}_{r, j})$. We conclude that there are elements
\begin{equation}\label{Eq:DoublePointDiff}
x^{i}_{j}, y^{i}_{j} \in \Algebra^{-}_{<\action(\word^{i}_{l, j}b\word^{i}_{r, j})}, \quad \partial^{-}(\word^{i}_{l, j}a\word^{i}_{r, j}) = (\word^{i}_{l, j}b\word^{i}_{r, j}) + x^{i}_{j} + y^{i}_{j}.
\end{equation}
Here the $x^{i}_{j}$ and $y^{i}_{j}$ are such that
\be
\item the $x^{i}_{j}$ count contributions of disks which do not have $a$ as a positive puncture and \item the $y^{i}_{j}$ contain neither an $a$ nor a $b$.
\ee
Consequently each summand of $x^{i}_{j}$ will have exactly one word which contains an $a$.

Since $\partial^{2} = 0$ we know $\partial^{-}(x^{i}_{j} + y^{i}_{j}) = \partial^{-}(\word^{i}_{l, j}b\word^{i}_{r, j})$. Using the $x^{i}_{j}, y^{i}_{j}$ define
\begin{equation}\label{Eq:DoublePointPhiNot}
\Phi_{0}: \Algebra^{-} \rightarrow \Stab\Algebra^{+}, \quad \Phi_{0}\word = \begin{cases}
\word & a, b \notin \word, \\
e^{i}_{j, 1} & \word = (\word^{i}_{l, j}a\word^{i}_{r, j}),  \\
e^{i}_{j, 2} + \check{x}^{i}_{j} + y^{i}_{j} & \word = (\word^{i}_{l, j}b\word^{i}_{r, j})
\end{cases}
\end{equation}
where $\check{x}^{i}_{j}$ is obtained be replacing each $(\word^{i}_{l, j}a\word^{i}_{r, j})$ factor with a $e^{i}_{j, 1}$. In other words, $\check{x}^{i}_{j} = \Phi_{0}x^{i}_{j}$

The following two lemmas are modifications of Lemmas 6.3 and 6.4 from \cite{ENS:Orientations}.

\begin{lemma}\label{Lemma:PhiZeroChainMap}
Every $\word$ satisfying at least one of the following conditions satisfies $\partial^{+}\Phi_{0}\word = \Phi_{0}\partial^{-}\word$:
\be
\item $\word$ does not contain $a$ or $b$ and does not have an $a$ or $b$ in $\partial^{-}\word$.
\item $\word$ is a generator of $\Algebra^{-}_{< \action(\word^{1}_{l, 1}b\word^{1}_{r, 1})}$.
\item $\word$ does not contain chords which end on both $\Leg^{\Partition}_{1}$ and $\Leg^{\Partition}_{2}$.
\item $\word$ contains $a$ or $b$.
\ee
\end{lemma}

\begin{proof}
For the first item above, $\partial^{+}\word = \partial^{-}\word$ and so $\partial^{+}\Phi_{0}\word = \Phi_{0}\partial^{-}\word$.

In the second case $\word$ can not contain an $a$ or a $b$ by Equation \eqref{Eq:DoublePointActionOrdering}. Moreover $\partial^{-}\word$ cannot have any terms containing an $a$ or a $b$ by action considerations. Therefore we have reduced to the first case which has already been established.

In the third case, the negative asymptotics of disks contributing to $\partial^{-}\word$ have endpoints on the $\Leg^{\Partition}_{k}$ which are touched by the endpoints of the chords of $\word$. Hence $\word$ fits the hypothesis of the first item.

Now we consider the fourth case with $\word$ containing an $a$. So suppose that $\word = (\word^{i}_{l, j}a\word^{i}_{r, j})$. Then $\partial^{+}\Phi_{0}\word = \partial^{+}e^{i}_{j, 1} = e^{i}_{j, 2}$ and
\begin{equation}\label{Eq:DoublePointPhiZeroA}
\Phi_{0}\partial^{-}\word = \Phi_{0}((\word^{i}_{l, j}b\word^{i}_{r, j}) + x^{i}_{j} + y^{i}_{j}) = e^{i}_{j, 2} + 2(\check{x}^{i}_{j}  + y^{i}_{j}) = e^{i}_{j, 2}
\end{equation}
as desired.

Finally, we consider the fourth case with $\word$ containing an $b$, so $\word = (\word^{i}_{l, j}b\word^{i}_{r, j})$. Then
\begin{equation*}
\begin{gathered}
\partial^{+}\Phi_{0}\word = \partial^{+}(e^{i}_{j, 2} + \check{x}^{i}_{j} + y^{i}_{j}) = \partial^{+}(\check{x}^{i}_{j} + y^{i}_{j}) = \partial^{+}\Phi_{0}x^{i}_{j} + \partial^{+}y^{i}_{j}\\
\Phi_{0}\partial^{-}\word = \Phi^{0}\partial^{-}(x^{i}_{j} + y^{i}_{j}) = \Phi^{0}\partial^{-}x^{i}_{j} + \partial^{+}y^{i}_{j}.
\end{gathered}
\end{equation*}
We need to show that $\partial^{+}\Phi_{0}x^{i}_{j} = \Phi_{0}\partial^{-}x^{i}_{j}$.

\begin{figure}[h]
\begin{overpic}[scale=.5]{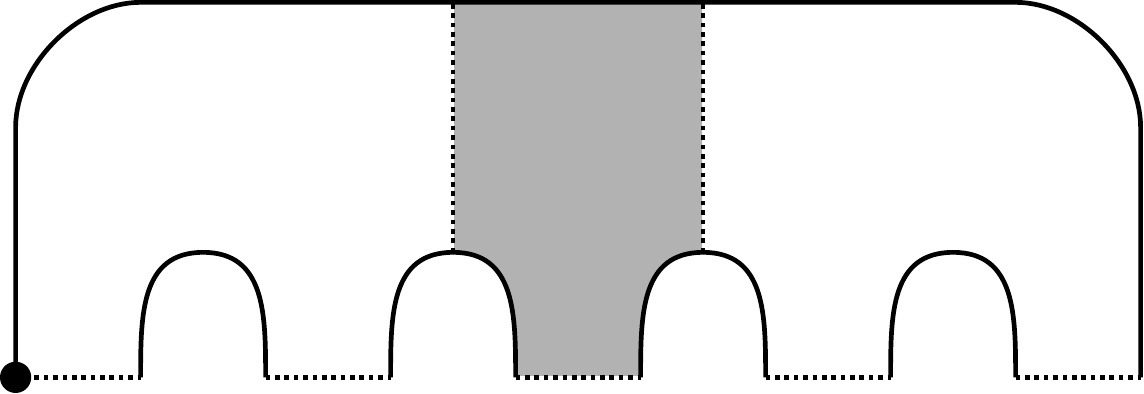}
\put(13, 20){$\planarDiagram(\word_{a})$}
\put(46, 20){$\planarDiagram(u)$}
\put(78, 20){$\planarDiagram(z_{k})$}
\end{overpic}
\caption{A disk contributing to $\partial^{-}z_{k}$ cannot have $a$ or $b$ as a negative puncture.}
\label{Fig:SplittingLemma}
\end{figure}

Each summand of $x^{i}_{j}$ contains exactly one factor $\word_{a}$ containing an $a$. Say the remaining factors are $z_{k}$ which cannot have an $a$ so our summand is $\word_{a}\prod z_{k}$. Then $b \notin z_{k}$ by the definition of the $x^{i}_{j}$ and $y^{i}_{j}$. It is clear from drawing planar diagrams that $a, b \notin \partial^{-} z_{k}$ as can be seen from Figure \ref{Fig:SplittingLemma}.

In summary $a, b \notin z_{k}$ and $a, b \notin \partial^{-}z_{k}$. We have already established that $\partial^{+}\Phi_{0}z_{k} = \Phi_{0}\partial^{-}z_{k}$ by the first item in the statement of the lemma. We have also established $\partial^{+}\Phi_{0}\word_{a} = \Phi_{0}\partial^{-}\word_{a}$ in Equation \eqref{Eq:DoublePointPhiZeroA} above. Therefore $\partial^{+}\Phi_{0}(\word_{a}\prod z_{k}) = \Phi_{0}\partial^{-}(\word_{a}\prod z_{k})$ by the Leibniz rules for $\partial^{\pm}$. Summing over the terms in $x^{i}_{j}$, we have $\partial^{+}\Phi_{0}x^{i}_{j} = \Phi_{0}\partial^{-}x^{i}_{j}$ as desired, completing the proof.
\end{proof}

\begin{lemma}\label{Lemma:PartialChain}
$\pi \partial^{+}\Phi_{0} = \pi\Phi_{0}\partial^{-}$.
\end{lemma}

\begin{proof}
Using the previous lemma and the fact that $\pi$ is an algebra morphism, we assume that $\word$ contains no $a$ or $b$ but has chords ending on both $\Leg^{\Partition}_{1}$ and $\Leg^{\Partition}_{2}$. Write
\begin{equation*}
\partial^{-}\word = \word_{\Algebra} + \word_{a} + \word_{b}, \quad \partial^{+}\word = \word_{\Algebra} + \word_{\complement}
\end{equation*}
where $\word_{\Algebra}$ are the terms which contain no $a$ or $b$, $\word_{a}$ is the sum of all terms containing $a$ in one of its factors, and $\word_{b}$ is the sum of all terms containing $b$ in one of its factors. So $\word_{\complement}$ is determined by $\word_{\Algebra}$ and $\partial^{+}\word$. Then
\begin{equation*}
\pi\partial^{+}\Phi_{0}\word - \pi\Phi_{0}\partial^{-}\word = \word_{\complement} - \pi\Phi_{0}\word_{b}.
\end{equation*}

The disks contributing to $\word_{\complement}$ are as shown in the right hand side of the bottom row of Figure \ref{Fig:DoublePoint}. The endpoints of the arc $\gamma$ in the figure determines a unique (up to boundary-relative isotopy) embedding of an arc in $\planarDiagram(\word)$ which we'll also call $\gamma$. Cutting $\planarDiagram(\word)$ along $\gamma \subset \planarDiagram(\word)$ yields a word $(\word^{i}_{l, j}a\word^{i}_{r, j})$ for some $i, j$. Every disk $u_{\complement}$ contributing to $\word_{\complement}$ can also be cut along $\gamma$ to obtain a disk $u_{l}$ (where $l$ stands for ``left'') contributing to $\partial^{-}(\word^{i}_{l, j}a\word^{i}_{r, j})$. These $u_{l}$ have $a$ as a positive puncture and cannot have $b$ as a negative puncture. In other words, the $u_{l}$ count exactly the contributions to the $y^{i}_{j}$ part of $\partial^{-}(\word^{i}_{l, j}a\word^{i}_{r, j})$.

\begin{figure}[h]
\begin{overpic}[scale=.4]{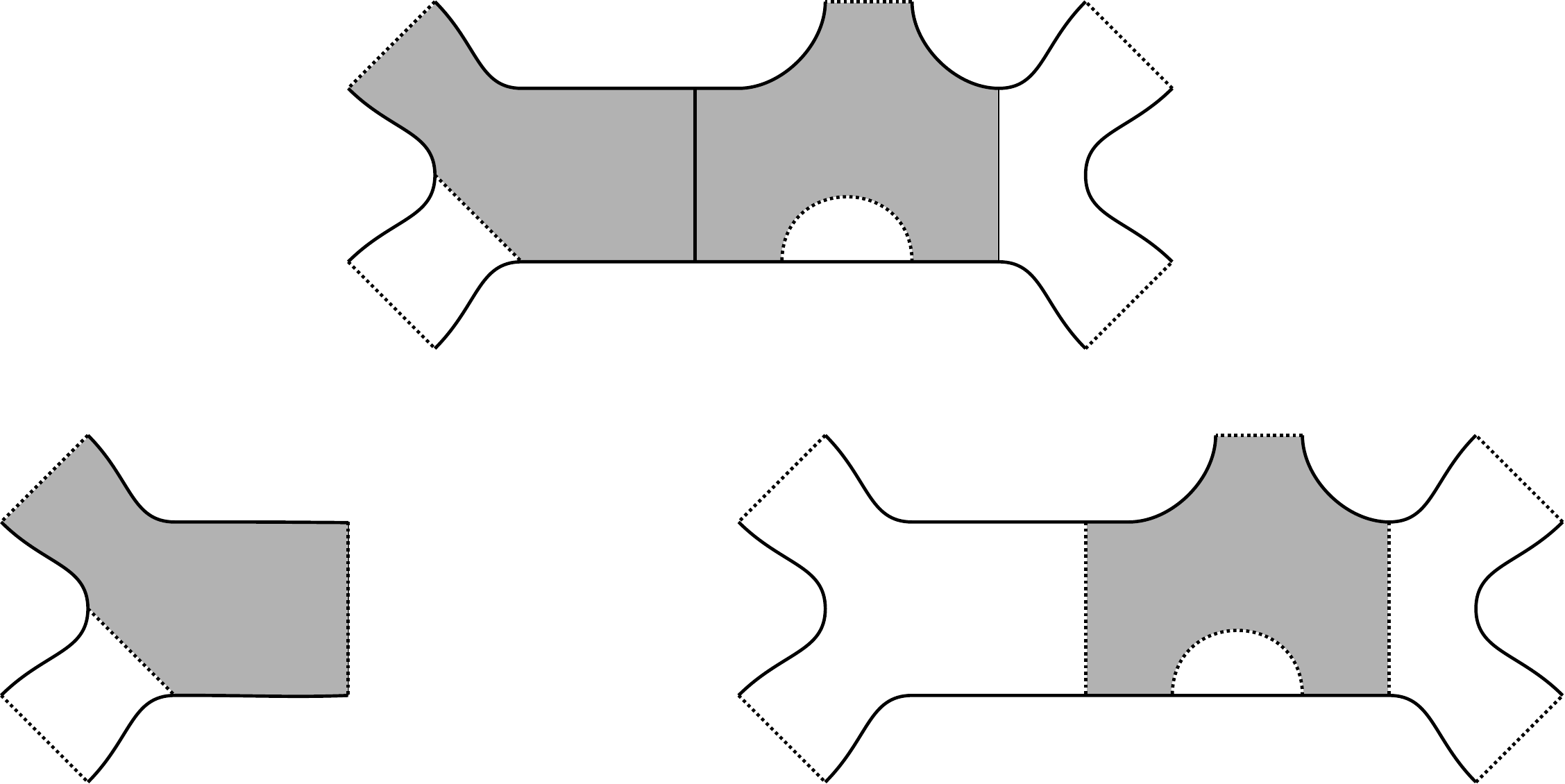}
\put(50, 40){$\planarDiagram(\word_{\complement})$}
\put(41, 38){$\gamma$}
\put(12, 12){$\planarDiagram(u_{l})$}
\put(24, 12){$a^{+}$}
\put(75, 12){$\planarDiagram(u_{r})$}
\put(65, 12){$b^{-}$}
\put(85, 35){$\planarDiagram(\word)$}
\put(83, 36){\vector(-1, 0){8}}
\put(90, 32){\vector(0, -1){8}}
\end{overpic}
\caption{Inscriptions $\planarDiagram(\word_{\complement}), \planarDiagram(u_{r}) \subset \planarDiagram(\word)$, and $\planarDiagram(u_{l})\subset \planarDiagram((\word^{i}_{l, j}a\word^{i}_{r, j}))$.}
\label{Fig:DoublePointCurveSplit}
\end{figure}

After cutting $u_{\complement}$ along $\gamma$ we also have a $u_{r}$ which contributes to $\partial^{-}\word$ and have $b$ as a negative puncture. These disks are responsible for the $\word_{b}$ part of $\partial^{-}\word$.

According to Equation \eqref{Eq:DoublePointPhiNot}, the last term $\pi\Phi_{0}\word_{b}$ is obtained by replacing every $(\word^{i}_{l, j}b\word^{i}_{r, j})$ in $\word_{b}$ with its associated $y^{i}_{j}$ and we must show that this term agrees with $\word_{\complement}$.\footnote{The corresponding $x^{i}_{j}$ terms from Equation \eqref{Eq:DoublePointPhiNot} are annihilated by $\pi$ as each summand of $\Phi_{0}x^{i}_{j}$ will have $e^{i}_{j, 1}$ as a factor.} By the above cutting and pasting argument, this is exactly the count of contributions to $\word_{\complement}$. The $u_{l}$ and $u_{r}$ can both be patched together to get a disk contributing to $\word_{\complement}$ if and only if all disks involved are admissible. Thus $\word_{\complement}$ is exactly $\pi\Phi_{0}\word_{b}$ as desired.
\end{proof}

Lemma \ref{Lemma:PhiZeroChainMap} shows that $\Phi_{0}$ is a chain map when restricted to $\Algebra^{-}_{\leq \action(\word^{1}_{l, N_{1}}a\word^{1}_{r, N_{1}})}$. We will now define maps $\Phi_{k}$ for $k > 0$ which will be chain maps when restricted to successively larger subalgebras of $\Algebra^{-}$ using the action orderings of words appearing in Equation \eqref{Eq:DoublePointActionOrdering}.

A degree $1$ endomorphism $\widehat{h}$ of the vector space underlying $\Stab\Algebra^{+}$ is determined by the following formula for generators,
\begin{equation*}
\widehat{h}(\word) = \begin{cases}
e^{i}_{j, 1} & \word = e^{i}_{j, 2},\\
0 & \text{otherwise}.
\end{cases}
\end{equation*}
Then $\widehat{h}$ determines a degree $1$ map $h: \Stab\Algebra^{+} \to \Stab\Algebra^{+}$ by linear extension of the formulas
\begin{equation*}
h(1) = 0, \quad h(\word \thicc{\word}) = \widehat{h}(\word)\thicc{\word} + \word h(\thicc{\word})
\end{equation*}
for generators $\word$ and monomials $\thicc{\word}$. Then
\begin{equation}\label{Eq:StabChainHomotopy}
\Id_{\Stab\Algebra^{+}} - \pi = \partial^{+}h + h\partial^{+}.
\end{equation}

Looking back to Equation \eqref{Eq:DoublePointActionOrdering}, $\Algebra^{+} \subset \Algebra^{-}$ is generated by the words $\word_{k}$ not containing $a$ or a $b$, and are indexed such that $\action(\word_{i}) \leq \action(\word_{i+1})$. Starting with $\Phi_{0}$ as in Equation \eqref{Eq:DoublePointPhiNot} above, inductively define unital algebra morphisms
\begin{equation*}
\psi_{k}: \Stab\Algebra^{+} \rightarrow \Stab\Algebra^{+}, \quad \Phi_{k} = \psi_{k}\Phi_{k-1}: \Algebra^{-} \rightarrow \Stab\Algebra^{+}
\end{equation*}
with $\psi_{k}$ defined on generators of $\Stab\Algebra^{+}$ by
\begin{equation}\label{Eq:DoublePointPhik}
\psi_{k}\word = \begin{cases}
\word & \word \neq \word_{k}\\
\word_{k} + h(\partial^{+}\word_{k} - \Phi_{i-1}\partial^{-}\word_{k}) & \word = \word_{k}.
\end{cases}
\end{equation}

Then we have the identities,
\begin{equation*}
\pi h = 0, \quad \pi\Phi_{k} = \pi\Phi_{k-1} = \cdots = \pi \Phi_{0}, \quad \partial^{-}\word_{k} \in \Algebra^{-}_{<\action(\word_{k})}.
\end{equation*}
Now assume that $\partial^{+}\Phi_{k-1} = \Phi_{k-1}\partial^{-}$ when restricted to $\Algebra^{-}_{< \action(\word_{k})}$. Then if $\word$ satisfies $\partial^{-}\word \in \Algebra^{-}_{< \action(\word_{k})}$
\begin{equation*}
\begin{aligned}
\Phi_{k}\partial^{-}\word &= \Phi_{k-1}\partial^{-}\word \\
&= (\partial^{+}h + h\partial^{+} + \pi)\Phi_{k-1}\partial^{-}\word\\
&= \partial^{+}h\Phi_{k-1}\partial^{-}\word + h\partial^{+}\partial^{+}\Phi_{k-1}\word + \pi\Phi_{0}\partial^{-}\word\\
&= \partial^{+}h\Phi_{k-1}\partial^{-}\word + \pi\Phi_{0}\partial^{-}\word\\
&= \partial^{+}h\Phi_{k-1}\partial^{-}\word + \pi\partial^{+}\Phi_{0}\word\\
&= \partial^{+}h\Phi_{k-1}\partial^{-}\word + (\Id_{\Stab\Algebra^{+}} + \partial^{+}h + h\partial^{+})\partial^{+}\word\\
&= \partial^{+}(\Id_{\Stab\Algebra^{+}} + h\Phi_{k-1}\partial^{-} + h\partial^{+})\word\\
&= \partial^{+}\Phi_{k}\word.
\end{aligned}
\end{equation*}
Here the first line uses the fact that $\Phi_{k} = \Phi_{k-1}$ on $\Algebra^{-}_{< \action(\word_{k-1})}$. The second uses Equation \eqref{Eq:StabChainHomotopy}. The third line uses $\pi\Phi_{k-1} = \pi\Phi_{0}$, the fourth uses $\partial^{+}\partial^{+} = 0$, and the fifth uses Lemma \ref{Lemma:PartialChain}. In the sixth line we again apply Equation \eqref{Eq:StabChainHomotopy} together with $\Phi_{0}\word_{k} = \word_{k}$. Finally, we rearrange some terms and apply the definition of $\Phi_{k}$.

The above computation applies to $\word = \word_{k}$ or $(\word^{i}_{j}b\word^{i}_{j})$ for which $(\word^{i}_{j}b\word^{i}_{j}), (\word^{i}_{j}b\word^{i}_{j}) < \action(\word_{k+1})$. To complete the inductive step in our proof, we must ensure that $\Phi_{k}$ is still a chain map on the words $(\word^{i}_{j}a\word^{i}_{j})$ containing an $a$. We compute
\begin{equation*}
\begin{aligned}
\Phi_{k}\partial^{-}(\word^{i}_{l,j}a\word^{i}_{r,j}) &= \Phi_{k}((\word^{i}_{l,j}b\word^{i}_{r,j}) + \check{x}^{i}_{j} + y^{i}_{j})\\
&= \Phi_{k}(\word^{i}_{l,j}b\word^{i}_{r,j}) + \Phi_{k-1}(\check{x}^{i}_{j} + y^{i}_{j})\\
&= e^{i}_{j, 2} = \partial^{+}e^{i}_{j, 1} = \partial^{+}\Phi_{k}(\word^{i}_{l,j}a\word^{i}_{r,j}).
\end{aligned}
\end{equation*}
Here the second line uses the fact that $\action(\check{x}^{i}_{j}), \action(y^{i}_{j}) \leq \action(w_{k-1})$.

This completes the induction step in our argument. Again looking back to the indices of Equation \eqref{Eq:DoublePointActionOrdering}, for $k = n_{N_{i}}$ we will have that $\Phi_{k}$ is a stable tame isomorphism $\Algebra^{-} \rightarrow \Stab\Algebra^{+}$. The proof of Proposition \ref{Prop:DoublePointInv} is then complete.

\subsection{Invariance in high dimensions}\label{Sec:HighDimInvariance}

Now we address stable table isomorphism invariance of $\newRSFTA$ in contact manifolds of the form $\R_{t} \times \SympBase$ with $\dim \SympBase > 2$. Let $\Leg(0)$ and $\Leg(1)$ be a pair of chord generic Legendrians which are Legendrian isotopic. Also let $J_{\SympBase}(0), J_{\SympBase}(1)$ be a pair of almost complex structures for which each $\Leg(0)$ and $\Leg(1)$ are both admissible with respect to their corresponding $J_{\SympBase}(T)$.

\begin{prop}\label{Prop:HighDimInv}
The mfDGAs $\newRSFTA(\Leg(0), J(0))$ and $\newRSFTA(\Leg(1), J(1))$ are filtered stable tame isomorphic.
\end{prop}

We will follow the strategies of \cite{EES:LegendriansInR2nPlus1, EHK:LagrangianCobordisms}. By the fact that $\newRSFTA$ moduli spaces can be identified with $\CE$ moduli spaces using Lemma \ref{Lemma:LCHShift}, we need no new analytical tools and can rely on those of \cite{EES:LegendriansInR2nPlus1}.

\begin{rmk}
Invariance proofs in \cite{EES:Orientations, EES:PtimesR} modify the general strategy of \cite{EES:LegendriansInR2nPlus1} and appear more specifically adapted to counting $\CE$ curves with a single positive puncture.
\end{rmk}

By consideration of generic $1$-parameter families $(\Leg(T), J_{\SympBase}(T)), T \in [0, 1]$ we can prove Proposition \ref{Prop:HighDimInv} by establishing invariance under three types of modification:
\be
\item[(I0)] The chords for each $\Leg(T)$ are the same for each $T$ and for each $\boundaryWord$ the $T$-parameterized moduli spaces $\ModSpace(\boundaryWord, J_{\SympBase}(T))$ have only disks $u$ with $\ind(u) \geq 1$.
\item[(I1)] With one of $\Leg(T) = \Leg(0)$ or $J_{\SympBase}(T)$ fixed and the other varying, there may appear \emph{handle slide disks} (ie. $\ind(u) = 0$ disks) in the $T$-parameterized moduli spaces.
\item[(I2)] The $\Leg(T)$ can undergo a \emph{double point move}.\footnote{These are called self-tangencies or birth-death moments in \cite{EES:LegendriansInR2nPlus1, EES:PtimesR, EES:Orientations}.}
\ee

The italicized terminology will be described below. By \cite[Lemma 2.8]{EES:PtimesR} these cases can be treated separately and further broken down into
\be
\item $J_{\SympBase}(T)$ is varying and $\Leg$ is constant or 
\item $J_{\SympBase}$ is constant and $\Leg(T)$ is varying.
\ee

\subsubsection{Varying complex structure}

We first consider time varying $J_{\SympBase}(T)$ with $\Leg$ fixed. Choose a $\R$ family $\hat{J}_{\SympBase}(s)$ of compatible almost complex structures which are constant in $s$ outside of some $[-\const, \const]$ which interpolates between $J_{\SympBase}(0)$ and $J_{\SympBase}(1)$. Define a cobordism map $\Phi: (\Algebra, \partial^{-}) \rightarrow (\Algebra, \partial^{+})$ by counting $\ind = 0$ holomorphic curves with boundary on trivial cylinder $\Lag_{\Leg} = \R_{s} \times \Leg$ and complex structure $J(s)\partial_{s} = \partial_{t}$ and $J(s)|_{T\SympBase} = \check{J}_{\SympBase}(s)$. To acheive genericity, we may apply an arbitrarily small, compactly supported Hamiltonian perturbation to $\Lag_{\Leg}$. Here $\partial^{-}$ is the $\newRSFTA$ differential associated to $J_{\SympBase}(0)$ and $\partial^{+}$ is the $\newRSFTA$ differential associated to $J_{\SympBase}(1)$.

The morphism $\Phi$ is a tame isomorphism if and only if $C_{\word} = \langle \Phi \word, \word\rangle \neq 0$ for every generator $\word \in \Algebra$. The coefficients $C_{\word}$ count unions of holomorphic strips $u$ which have the same positive and negative asymptotics. In other words $\boundaryWord(u) = \chord^{+}\chord^{-}$ for some $\chord$. This means $\action(u) = 0$ so that the $\pi_{\SympBase}u$ must be constant and each $u$ must be \emph{the} trivial strip. So $C_{\word} = 1$ completing the proof.

Moving forward, assume that $J_{\SympBase}$ is constant in $T$.

\subsubsection{Invariance in the presence of handle slide disks}

Consider when $\Leg(T)$ is an isotopy through admissible Legendrians for which the chords vary continuously with $T$. For a boundary word $\boundaryWord$ we write $\ModSpace^{T}_{\SympBase}\left(\boundaryWord\right)$ for the $\ModSpace_{\Leg}$ moduli space of holomorphic disks with asymptotics determined by $\boundaryWord$ and boundary intervals mapped to $\pi_{\SympBase}\Leg(T)$. Likewise
\begin{equation*}
\ModSpace^{[0, 1]}_{\SympBase}\left(\boundaryWord\right) = \left\{ (T, u)\ : T \in [0, 1], u \in \ModSpace^{T}_{\SympBase}\left(\boundaryWord\right) \right\}.
\end{equation*}
Using Lemma \ref{Lemma:LCHShift}, the following lemma is a restatement of \cite[Proposition 2.9]{EES:LegendriansInR2nPlus1}.

\begin{prop}
For a generic admissible isotopy $\Leg(T)$ as above, the following hold:
\be
\item If $\ind\left(\boundaryWord\right) = d \leq 1$ then $\ModSpace^{[0, 1]}_{\SympBase}\left(\boundaryWord\right)$ is a transversely cut out $d$-manifold.
\item When $d = 0$, $\ModSpace^{[0, 1]}_{\SympBase}\left(\boundaryWord\right)$ consists of a finite set $\{ (T_{j}, u_{j}) \}$ with the $T_{j}$ all distinct.
\ee
\end{prop}

In a generic admissible isotopy as above, we say that a $(T, u)$ with $\ind(u) = 0$ is a \emph{handle slide disk}. Note that Lemma \ref{Lemma:Lifting} implies that if $u \in \ModSpace_{\Lag_{\Leg(T)}}\left(\boundaryWord\right)$ has $\ind(u) = d$ then $\ModSpace^{T}_{\SympBase}\left(\boundaryWord\right)$ has expected dimension $d-1$ and $\ModSpace^{[0, 1]}_{\SympBase}\left(\boundaryWord\right)$ has expected dimension $d$.

Invariance under (I0) modification of $\Leg$ follows from a standard moduli space cobordism argument, cf. \cite[Lemma 2.8]{EES:LegendriansInR2nPlus1}. We continue with the (I1) case, supposing that there is a single handle slide disk appearing at time $T = 0$ for an isotopy $\Leg(T)$ with $T \in [-1, 1]$.

Following \cite[\S 6]{EHK:LagrangianCobordisms}, construct exact Lagrangian cobordisms $\Lag_{\epsilon}: \Leg(- \epsilon) \rightarrow \Leg(\epsilon)$ and inverse cobordisms $\Lag_{-\epsilon}: \Leg(\epsilon) \rightarrow \Leg(- \epsilon)$. We assume genericity so that the $\Lag_{\pm\epsilon}$ induces cobordism maps
\begin{equation*}
\newRSFTA(\Leg(- \epsilon)) \xrightarrow{\Phi_{\epsilon}} \newRSFTA(\epsilon) \xrightarrow{\Phi_{-\epsilon}} \newRSFTA(- \epsilon)
\end{equation*}
by counting $\ind = 0$ holomorphic disks. Define
\begin{equation*}
C_{\word, \pm\epsilon} = \langle \Phi_{\pm\epsilon}\word, \word\rangle
\end{equation*}
to be the ``diagonal coefficients'' of the cobordism maps. If there is some $\epsilon$ for which each $C_{\word, \epsilon}$ is non-zero, then $\Phi_{\epsilon}$ is a tame isomorphism.

So suppose that there is some $\word$ for which $C_{\word, \epsilon} = 0$ for all $\epsilon$. The concatenation of cylindrical cobordisms $\Lag_{-\epsilon}\# \Lag_{\epsilon}$ is isotopic to the trivial cylinder $\Lag_{\Leg(-\epsilon)}$. An isotopy between these cobordisms will induce chain homotopy maps
\begin{equation*}
\Id_{\Algebra} - \Phi_{-\epsilon}\Phi_{\epsilon} = h_{\epsilon}\partial_{-} + \partial_{-} h_{\epsilon}
\end{equation*}
where each $h_{\epsilon}: \Algebra \rightarrow \Algebra$ has degree $1$ and counts perturbed holomorphic disks. See \cite[Lemma 3.13]{EHK:LagrangianCobordisms} (which closely matches the language of this article) or the original construction in \cite[Appendix B]{Ekholm:Z2RSFT}.

If $C_{\word, \epsilon} = 0$ for all $\epsilon$, then the preceding equation tells us
\begin{equation*}
\langle (h_{\epsilon}\partial_{-\epsilon} + \partial_{-\epsilon} h_{\epsilon})\word, \word \rangle = \langle \Id_{\Algebra_{-\epsilon}} \word, \word \rangle = 1
\end{equation*}
for all $\epsilon$ where $\partial_{-\epsilon}$ is the differential for $\Leg(-\epsilon)$. Let's say that the perturbations used to count curves for the $h_{\epsilon}$ maps tend to $0$ with $\epsilon$ so that their contributions converge to unperturbed curves as $\epsilon \rightarrow 0$. Assume that there is a sequence $\epsilon_{N}$ converging to zero for which the $\langle (h_{\epsilon_{N}}\partial_{-\epsilon_{N}} + \partial_{-\epsilon_{N}} h_{\epsilon_{N}})\word, \word \rangle = 1$. Passing to a subsequence of $\epsilon_{N}$, we can assume exactly one of 
\begin{equation*}
\langle \partial_{-\epsilon_{N}} h_{\epsilon_{N}}\word, \word \rangle = 1, \quad \langle h_{\epsilon_{N}}\partial_{-\epsilon_{N}}\word, \word \rangle = 1
\end{equation*}
holds. We'll assume the latter equation is satisfied: The argument in the former case is the essentially same. The quantity $\langle h_{\epsilon_{N}}\partial_{-\epsilon_{N}}\word, \word \rangle$ counts multi-vertex bubble trees with disks contributing to $\partial_{-\epsilon} \word$ on top, so that the bubble tree glues to a strip with exactly one incoming and exactly one outgoing puncture. Each of $h_{\epsilon_{n}}$ and $\partial_{-\epsilon_{n}}$ must be non-zero. The bubble trees will converge to a bubble tree of disks with boundary on $\Leg(0)$ in a Gromov limit, and hence they will have energy tending to $0$ as $n \rightarrow \infty$. Since the energies of disks contributing to $\partial_{-\epsilon}$ are bounded below by the minimal action difference between chords, this is impossible. So we conclude that $C_{\word, \epsilon} = 1$ for $\epsilon$ sufficiently small.

By contradiction, we've shown $C_{\word, \epsilon} \neq 0$ so that $\Phi_{\epsilon}$ is a tame isomorphism. This establishes stable tame isomorphism invariance of $\newRSFTA$ in the presence of handle slide disks.

\subsubsection{High dimensional double point moves}\label{Sec:HighDimDoublePt}

In \cite{EES:LegendriansInR2nPlus1} analytical arguments regarding compactness and gluing of holomorphic disks during a double point isotopy facilitate the use of Chekanov's original algebraic arguments \cite{Chekanov:LCH, ENS:Orientations} in the high dimensional setting. We have already adapted these arguments from the $\CE$ context to be of use to $\newRSFTA$ in \S \ref{Sec:3dDoublePt}. Here we set up the double point move and sketch what modifications are required to establish invariance in the high dimensions.

Double point moves can be described by the following local model. Let
\begin{equation*}
\lambda_{1}(q, T), \lambda_{2}(q, T): [-1, 1] \rightarrow \disk
\end{equation*}
be a $T \in [-1, 1]$ family of paths parameterizing the lower and upper strands of the Legendrians in the top row of Figure \ref{Fig:DoublePoint}. Say that at $T = -1$ we have the left subfigure with chords $a, b$, at $T = 0$ we have a quadratic tangency, and at $T = 1$ we have the right subfigure.\footnote{Any of the aforementioned references provide an explicit model for the $\lambda_{j}(q, T)$ but it will help us to have the picture in mind.} We assume that for $T \neq 0$ the chords $a, b$ are non-degenerate, converging to a degenerate chord $c$ at $T = 0$. Say that within the diagram $\beta_{\disk} \in \Omega^{1}(\disk)$ is such that our contact form is $dt + \beta_{\disk}$ so that (possibly after a $t$ translation), there are some functions $f_{1}, f_{2}: [-1, 1]_{q} \times [0, 1]_{T} \rightarrow \R$ for which the Legendrian strands of $\Leg(T)$ are parameterized by $q\mapsto (f_{1}(q, T)), \lambda_{1}(q, T))$ and $q\mapsto (f_{2}(q, T)), \lambda_{2}(q, T))$ in $\R_{t} \times \disk$ with $\sup f_{1} < \inf f_{2}$.

We upgrade this picture to higher dimensions by considering parameterized, $1$-parameter families of Legendrians
\begin{equation*}
\begin{gathered}
\Leg_{1}(T), \Leg_{2}(T): [-1, 1]_{q} \times \disk^{n-2}_{\vec{q}} \rightarrow \R_{t} \times \disk \times \disk^{2n-2}\\
\Leg_{1}(T)(q, \vec{q}) = (f_{1}(q, T), \lambda_{1}(q, T), \vec{q}), \quad \Leg_{2}(T)(q, \vec{q}) = (f_{2}(q, T), \lambda_{2}(q, T), J_{0}\vec{q})
\end{gathered}
\end{equation*}
Here we are viewing $\disk^{2n-2}$ as the unit disk in $\C^{2n-2}$ with the standard almost complex structure $J_{0}$. Using coordinates $x_{i}, y_{i}$ on $\disk^{2n-2}$, the contact form
\begin{equation*}
dt + \beta_{\disk} + \half \sum (x_{i}dy_{i} - y_{i}dx_{i})
\end{equation*}
on $\R_{t} \times \disk \times \disk^{2n-2}$ is such that the $\Lambda_{j}(T)$ are all Legendrian for $j=0, 1$ and $T \in [0, 1]$. As in the low-dimensional case, $\Leg(-1)$ has two more chords than $\Leg(1)$ which we still call $a, b$, and there is a degenerate chord $c$ for $\Leg(0)$.

Consider use of the local model to make a compact modification of some $\Leg \subset \R_{t} \times \SympBase$ via an inclusion $(\disk \times \disk^{2n-2}, \beta) \subset (\SympBase, \beta)$. We equip $\disk \times \disk^{2n-2}$ with the standard complex structure which we'll also call $J_{0}$ and assume that $J_{\SympBase}$ is an adapted almost complex structure on $\SympBase$ which extends $J_{0}$ to all of $\SympBase$.

We write $(\Algebra^{-}, \partial^{-})$ for the algebra associated to some $\Leg(T), T \in [-1, 0)$ -- these are all the same by (I0) invariance and our assumption that there are no handle slide disks. Likewise write $(\Algebra^{+}, \partial^{+})$ for the algebra associated to some $\Leg(T), T \in (0, 1]$.\footnote{Note that \cite{EES:LegendriansInR2nPlus1} exchanges $(\Algebra^{\pm}, \partial^{\pm})$ for $(\Algebra^{\mp}, \partial^{\mp})$.}

As in \S \ref{Sec:3dDoublePt}, we say that $\Leg_{k}(T)$ belongs to some component $\Leg_{i_{k}}$ for $k = 1, 2$ and we can consider the cases in which these belong to the same or separate pieces $\Leg^{\Partition}_{j}$ of $\Leg$. We again assume the second case, $\Leg_{i_{k}} \subset \Leg^{\Partition}_{k}, k=1, 2$, as the first is essentially identical to that of \cite{Chekanov:LCH, EES:LegendriansInR2nPlus1}. In this case, the chords $a, b$ add many generators to our algebra (in contrast to the citations) and for $T$ close to $0$, the generators of $\Algebra^{-}$ may be ordered as in Equation \eqref{Eq:DoublePointActionOrdering}.

For $T < 0$ the holomorphic disk $u_{ab}: \R \times [0, \pi] \rightarrow \R_{s} \times \R_{t} \times \disk$ of Figure \ref{Fig:DoublePoint} determines a holomorphic disk $U_{ab} = (u_{ab}, 0)$ in $\R_{t} \times \disk \times \disk^{2n-2}$ with boundary on $\Leg(T)$ having energy
\begin{equation*}
\action(U_{ab}) = \action(a) - \action(b)
\end{equation*}
converging to $0$ as $T \rightarrow 0$. Consequently for $T \in [-\epsilon, \epsilon]$ and $\epsilon$ sufficiently small, $U_{ab}$ will be the only holomorphic disk in $\R_{t} \times \SympBase$ with boundary word $a^{+}b^{-}$. Thus we obtain Equation \eqref{Eq:DoublePointDiff} in parallel with \cite[Lemma 2.16]{EES:LegendriansInR2nPlus1} and define $\Phi_{0}$ as in Equation \eqref{Eq:DoublePointPhiNot}. Then Lemma \ref{Lemma:PhiZeroChainMap} follows from purely algebraic considerations.

Next we must prove Lemma \ref{Lemma:PartialChain}. Here we see the combinatorial splitting of holomorphic disks $u_{\complement} \mapsto (u_{l}, u_{r})$ analytically. The details are provided by the compactness-gluing results of \cite[Proposition 2.18]{EES:LegendriansInR2nPlus1}. In the $T \rightarrow 0$ limit, curves $u_{\complement}$ degenerate into nodal disks pinched along the arc $\gamma \subset \planarDiagram(u_{\complement})$. The combinatorics of planar diagrams used to show that $u_{l}$ contributes to the $y^{i}_{j}$ part of $\partial^{-}(\word^{i}_{l, j}a\word^{i}_{r, j})$ and that $u_{r}$ contributes to $\partial^{\pm}\word$ are unchanged.

With Lemmas \ref{Lemma:PhiZeroChainMap} and \ref{Lemma:PartialChain} established, we inductively define maps $\Phi_{k}$ as in Equation \eqref{Eq:DoublePointPhik}. The remainder of the proof follows from algebraic manipulation and does not require modification. Thus we have established the stable tame isomorphism invariance of $\newRSFTA$ under (I2) modification. This completes the proof of Proposition \ref{Prop:HighDimInv}.

\textsc{Institut de Math\'{e}matiques de Jussieu-Paris Rive Gauche, Sorbonne Université, Paris, France}\par\nopagebreak
\textit{Email:} \href{mailto:russell@imj-prg.fr}{russell@imj-prg.fr}\par\nopagebreak
\textit{URL:} \href{https://www.russellavdek.com/}{russellavdek.com}

\end{document}